\documentclass[11pt]{amsart}
\usepackage{amssymb}
\usepackage{calrsfs}
\usepackage[all]{xy}
\usepackage{tikz}
\usepackage{xcolor}
\usepackage[letterpaper,margin=1in]{geometry}

\usepackage{hyperref}

\newtheorem{dummy}{dummy}[section]
\newtheorem{lemma}[dummy]{Lemma}

\newtheorem{theorem}[dummy]{Theorem}

\newtheorem{conjecture}[dummy]{Conjecture}
\newtheorem{corollary}[dummy]{Corollary}
\newtheorem{proposition}[dummy]{Proposition}
\theoremstyle{definition}

\newtheorem{example}[dummy]{Example}
\newtheorem{remark}[dummy]{Remark}

\newcommand{\bA}{\mathbb{A}}
\newcommand{\bB}{\mathbb{B}}

\newcommand{\bC}{\mathbb{C}}

\newcommand{\bG}{\mathbb{G}}

\newcommand{\bP}{\mathbb{P}}

\newcommand{\R}{\mathbb{R}}

\newcommand{\Z}{\mathbb{Z}}

\newcommand{\cA}{\mathcal{A}}
\newcommand{\cB}{\mathcal{B}}
\newcommand{\cC}{\mathcal{C}}

\newcommand{\cE}{\mathcal{E}}
\newcommand{\cF}{\mathcal{F}}

\newcommand{\cO}{\mathcal{O}}

\newcommand{\cR}{\mathcal{R}}

\newcommand{\cV}{\mathcal{V}}

\newcommand{\GL}{\mathrm{GL}}

\DeclareMathOperator{\Sing}{Sing}

\DeclareMathOperator{\Perf}{Perf}
\DeclareMathOperator{\Qcoh}{QCoh}
\DeclareMathOperator{\Spec}{Spec}
\DeclareMathOperator{\Dsing}{DSing}
\DeclareMathOperator{\Coh}{Coh}

\DeclareMathOperator{\Ext}{Ext}

\DeclareMathOperator{\holim}{holim}

\DeclareMathOperator{\PreMF}{PreMF}
\DeclareMathOperator{\MF}{MF}
\DeclareMathOperator{\HH}{HH}

\DeclareMathOperator{\Aut}{Aut}
\DeclareMathOperator{\Auteq}{Auteq}
\DeclareMathOperator{\twoper}{2Per}
\DeclareMathOperator{\Tot}{Tot}
\DeclareMathOperator{\Sym}{Sym}
\DeclareMathOperator{\Indcoh}{IndCoh}
\DeclareMathOperator{\DGCat}{DGCat}
\DeclareMathOperator{\gsing}{\mathfrak{Sing}}
\DeclareMathOperator{\SingSupp}{SingSupp}
\DeclareMathOperator{\Supp}{Supp}
\newcommand{\clpart}{\mathrm{Classical}}

\newcommand{\modules}{\text{-mod}}

\newcommand{\cofiber}{\mathrm{cofiber}}

\newcommand{\Fw}{\mathrm{Fuk}^{w}}
\newcommand{\Fc}{\mathrm{Fuk}^{c}}

\newcommand{\RFuk}{\mathrm{RFuk}}
\newcommand{\Fuk}{\mathrm{Fuk}}

\newcommand{\id}{\mathrm{id}}
\newcommand{\gr}{\mathrm{gr}}
\newcommand{\pt}{\mathrm{pt}}
\newcommand{\op}{\mathrm{op}}
\newcommand{\Star}{\mathrm{Star}}
\newcommand{\vertices}{\mathrm{Vert}}

\title[Singularity categories of normal crossings surfaces]{Singularity categories of normal crossings surfaces, descent, and mirror symmetry}

\author{James Pascaleff}
\address{James Pascaleff, University of Illinois at Urbana-Champaign}
\email{jpascale@illinois.edu}

\author{Nicol\`o Sibilla}
\address{Nicol\`o Sibilla, 
  SISSA,
  Via Bonomea 265, 
  34136 Trieste, Italy}
\email{nsibilla@sissa.it}

\begin{document}
\maketitle

\begin{abstract} Given a smooth 3-fold $Y$, a line bundle $L \to Y$, and a section $s$ of $L$ such that the vanishing locus of $s$ is a normal crossings surface $X$ with graph-like singular locus, we present a way to reconstruct the singularity category of $X$ as a homotopy limit of several copies of the category of matrix factorizations of $xyz : \bA^{3} \to \bA^{1}$ (the mirror to the Fukaya category of the pair of pants). This extends our previous result for the case where $L$ is trivialized. The key technique is the classification of non-two-periodic autoequivalences of the category of matrix factorizations. We also present a conjectural mirror for these singularity categories in terms of the Rabinowitz wrapped Fukaya categories of Ganatra-Gao-Venkatesh for certain symplectic four-manifolds, and relate this construction to work of Lekili-Ueda and Jeffs.
\end{abstract}

\section{Introduction}
\label{sec:intro}

Let $Y$ be a smooth 3-fold over a field $k$ and let $L$ be a line bundle over $Y$. Choose a section $s \in \Gamma(Y,L)$ and let $X = s^{-1}(0)$ be the divisor of zeros. This paper is concerned with the category of singularities $\Dsing(X)$, which we regard as stable $\infty$-category; it is at least $k$-linear, but a part of the story is enhancing this linearity over a larger ring. One presentation for this category is as the quotient (homotopy cofiber) of coherent complexes modulo perfect complexes,
\begin{equation}
  \Dsing(X) = \Coh(X)/\Perf(X),
\end{equation}
but another presentation will be more convenient for our arguments (and in fact we usually consider the idempotent completion of this category).

The basic geometric insight is that $\Dsing(X)$ is sensitive to the geometry of a neighborhood of the singular locus of $X$. Because $X$ is normal crossings, this singular locus $\Sing(X)$ consists of a collection of curves meeting a triple points. We always assume that $X$ has \emph{graph-like singular locus}, meaning that all irreducible components of $\Sing(X)$ are isomorphic to $\bP^{1}$ or $\bA^{1}$, and that there are at most two triple points on each $\bP^{1}$ and at most one triple point on each $\bA^{1}$. Then the combinatorics of $\Sing(X)$ is described by a trivalent graph. 

Suppose that $X$ is a normal crossings surface with graph-like singular locus. Then $\Sing(X)$ may be presented as a scheme as the gluing (colimit) of several copies of $\Sing(\{xyz = 0\} \subset \bA^{3})$, that is, the union of the coordinate axes in $\bA^{3}$. If we now think about singularity categories, we have
\begin{equation}
  \Dsing(\{xyz=0\}) \cong \MF(\bA^{3},xyz),
\end{equation}
where the right-hand side is the category of matrix factorizations of the function $xyz: \bA^{3} \to \bA^{1}$. For a morphism $f : Y \to \bA^{1}$, with $Y$ smooth, there is in general an equivalence of the category of singularities $\Dsing(f^{-1}(0))$ with the category matrix factorizations $\MF(Y,f)$. This is originally due to Orlov \cite{orlov-d-branes}, although in this paper we mainly follow Preygel's approach \cite{preygel-thesis} to this theory. For this paper it is very important to know that the category $\MF(Y,f)$ carries a $2$-periodic structure, which is to say a distinguished choice of isomorphism $t : \id \to \id[2]$ between the identity functor and the double shift functor. This $2$-periodic structure \emph{depends} on the pair $(Y,f)$ and not merely on the scheme $f^{-1}(0)$; analyzing this dependence is crucial to this paper. Nevertheless, we may still say that the Orlov equivalence induces a $2$-periodic structure on  $\Dsing(f^{-1}(0))$.

In light of the foregoing, when $X$ has graph-like singular locus, it is natural to ask whether $\Dsing(X)$ is equivalent to a homotopy limit of several copies of $\Dsing(\{xyz=0\})$. Because the latter category admits $2$-periodic structures, there are two versions of of this question:
\begin{enumerate}
\item Is $\Dsing(X)$ equivalent to a homotopy limit of several copies of $\Dsing(\{xyz=0\})$ where all functors in the limit diagram are compatible with certain $2$-periodic structures on the pieces?
\item Is $\Dsing(X)$ equivalent to a homotopy limit of several copies of $\Dsing(\{xyz=0\})$, where the functors in the limit diagram are merely $k$-linear?
\end{enumerate}
An affirmative answer to the first question naturally implies that $\Dsing(X)$ itself admits a $2$-periodic structure. However, the category $\Dsing(X)$ is \emph{not necessarily} $2$-periodic, so there is an obstruction. When $\Dsing(X)$ is not $2$-periodic, the second question may still have an affirmative answer, but then there is the question of what kind of structure on $\Dsing(X)$ is induced by local $2$-periodic structures on the pieces.

When $X = s^{-1}(0)$ is the vanishing divisor of a section of a line bundle $s \in \Gamma(Y,L)$, the $2$-periodicity of $\Dsing(X)$ (or lack thereof) corresponds precisely to the triviality (or lack thereof) of the line bundle $L$.
\begin{enumerate}
\item When $L \cong \cO_{Y}$ is trivial, a choice of trivialization allows us to regard $s$ as a function $f: Y \to \bA^{1}$. Then we have the Orlov equivalence $\Dsing(X) \cong \MF(Y,f)$, and the right hand side is $2$-periodic. This $2$-periodicity comes from a natural transformation $t : \id \to \id[2]$ acting on $\Coh(X)$ (whose construction goes back to Gulliksen \cite{gulliksen} and Eisenbud \cite{eisenbud}), which becomes an isomorphism on $\Dsing(X)$.
\item When $L$ is not assumed to be trivial, there is still a natural transformation on $\Coh(X)$, but it is of the form $t : \id \to \otimes L|_{X}^{-1}[2]$, where $\otimes L|_{X}^{-1}$ denotes the functor of tensoring with a line bundle. This induces on $\Dsing(X)$ an \emph{$L|_{X}$-twisted $2$-periodic structure}, meaning that the double shift $[2]$ is isomorphic to tensoring with $L|_{X}$ (so $t$ becomes an isomorphism). This case is addressed by Seidel \cite[p.~87]{seidel-subalgebras}, Polishchuk-Vaintrob \cite{polishchuk-vaintrob}, Burke-Walker \cite{burke-walker}, and quite possibly others.
\end{enumerate}

Having observed that $\Dsing(X)$ may be $2$-periodic in a twisted sense, we can ask whether $\Dsing(X)$ may be obtained as the homotopy limit of several copies of $\Dsing(\{xyz=0\})$, where the functors in the diagram are $2$-periodic in a twisted sense. Our first result is that this is indeed the case.

First a combinatorial definition: Let $G$ be an undirected graph. We construct a category $J(G)$. The object set of $J(G)$ is the disjoint union of the sets of vertices and edges of $G$. There is exactly one nonidentity arrow for each flag $(v,e)$ consisting of a vertex $v$ and incident edge $e$; this arrow has source $v$ and target $e$.

\begin{theorem}
  \label{thm:twisted-gluing}
  Let $Y$ be a smooth threefold, $L$ a line bundle on $Y$, and $s \in \Gamma(Y,L)$ a section. Suppose that $X = s^{-1}(0)$ is normal crossings and has graph-like singular locus. Let $G(X)$ be the trivalent graph describing the singular locus of $X$. Let $\cV = \MF(\bA^{3},xyz)$ and let $\cE = \Coh^{(2)}(\bG_{m})$ be the $2$-periodic folding of $\Coh(\bG_{m})$. Then there is a diagram $F : J(G(X)) \to \DGCat$ of DG categories such that
  \begin{equation}
    \Dsing(X) \cong \holim \left(F: J(G(X)) \to \DGCat \right)
  \end{equation}
  and
  \begin{enumerate}
  \item  $F(v) = \cV$ for each vertex $v$ of $G(X)$,
  \item $F(e) = \cE$ for each edge $e$ of $G(X)$, and
  \item for each flag $(v,e)$, the functor $F(v,e) : \cV \to \cE$ is a composition of the following:
    \begin{enumerate}
    \item restrictions to open subsets,
    \item Kn\"{o}rrer periodicity equivalences,
    \item shifts,
    \item equivalences induced by automorphisms of $(\bA^{3},xyz)$ and $\bG_{m}$,
    \item certain non-$2$-periodic autoequivalences $\phi_{n}$ of $\cE = \Coh^{(2)}(\bG_{m})$ described below.
    \end{enumerate}
  \end{enumerate}
\end{theorem}

We remark that descent for $\Dsing$ and $\MF$ holds in great generality: see Section \ref{sec:descent-dsing-mf}. Our contribution is to describe the precise form of the limit.

To complete the statement of the theorem, we need to describe the non-$2$-periodic autoequivalences of $\Coh^{(2)}(\bG_{m})$ that we use. At a basic level, a coherent sheaf on $\bG_{m}$ is a module over $\cO(\bG_{m}) = k[x,x^{-1}]$. A $2$-periodic complex is also a module over $k[u,u^{-1}]$, where $\deg(u) = 2$. So an object in $\Coh^{(2)}(\bG_{m})$ is a module over $k[x^{\pm 1},u^{\pm 1}]$. We can twist this action by a graded automorphism of the ring $k[x^{\pm 1},u^{\pm 1}]$. For  $n \in \Z$, we may consider the automorphism
\begin{equation}
  \phi_{n} \in \Aut^{\gr}( k[x^{\pm 1},u^{\pm 1}] ), \quad \phi_{n}(x) = x, \quad \phi_{n}(u) = x^{n}u.
\end{equation}
The autoequivalence of $\Coh^{(2)}(\bG_{m})$ induced by $\phi_{n}$ is what is referred to in Theorem \ref{thm:twisted-gluing}.

Section \ref{sec:graph-like} gives more detailed statements that help pin down the exact form of the diagram in other respects.

It may be helpful to remark why the group $\Aut^{\gr}(k[x^{\pm 1},u^{\pm 1}])$ is relevant. We need to restrict to graded automorphisms so as to preserve the $\Z$-grading on $\Coh^{(2)}(\bG_{m})$. However, such automorphisms are not necessarily compatible with the $2$-periodic structure: to be compatible would mean that the automorphism is $k[u,u^{-1}]$-linear, and such an automorphism would send $u$ to $u$.

\subsection{From twisted to untwisted}
While $L$-twisted $2$-periodic categories arise naturally when considering the zero locus of a section of $L$, many of the categories appearing in mirror symmetry, such as categories of matrix factorizations and Fukaya categories of general type varieties, are always genuinely $2$-periodic. One way to obtain an untwisted $2$-periodic category from an $L$-twisted $2$-periodic one is to apply a localization that trivializes the line bundle $L$ in an appropriate sense. We call such a localization a \emph{relative singularity category} of $X$, and by contrast we call $\Dsing(X)$ the \emph{absolute singularity category}.

Relative singularity categories have been studied by Efimov-Positselski \cite{efimov-positselski}, Blanc-Robalo-To\"{e}n-Vezzosi \cite{BRTV}, and Pippi \cite{pippi}, and they arise most naturally when considering Landau-Ginzburg models whose total space is itself singular. Let $f : Y' \to \bA^{1}$ be a flat morphism and let $X = f^{-1}(0)$ be the zero fiber; we do not assume that $Y'$ is regular. The inclusion $i : X \to Y'$ is an local complete intersection morphism and therefore has finite tor-dimension. This implies that the pullback $i^{*}$ sends $\Coh(Y')$ to $\Coh(X)$, and it sends $\Perf(Y')$ to $\Perf(X)$, so we have an induced functor
\begin{equation}
  i^{*} : \Dsing(Y') \to \Dsing(X)
\end{equation}
The relative singularity category $\Dsing(X,Y')$ is defined to be the homotopy cofiber of this functor. It may also be presented as a quotient
\begin{equation}
  \Dsing(X,Y) = \Coh(X)/\langle\Perf(X), i^{*}\Coh(Y')\rangle.
\end{equation}
%when $Y$ is regular, $\Coh(Y) = \Perf(Y)$ and so $i^{*}\Coh(Y) = i^{*}\Perf(Y) \subseteq \Perf(X)$, so that we recover the absolute singularity category.

Starting from the situation where $Y$ is regular and $s \in \Gamma(Y,L)$ is a section, we can obtain a relative singularity category for $X = s^{-1}(0)$ as follows. Choose another section $s' \in \Gamma(Y,L)$, and consider the rational map $f = s/s' : Y \dashrightarrow \bP^{1}$. Let $Y'$ be the closure of the graph of $f$; this $Y'$ will generally have singularities if the base locus $s^{-1}(0) \cap (s')^{-1}(0)$ is not regular. We may then consider the relative singularity category $\Dsing(X,Y')$; note that this depends on the choice of the second section $s'$. In terms of $2$-periodic structures, we have now have two natural transformations on $\Coh(X)$
\begin{equation}
  t : \id \to \otimes L|_{X}^{-1}[2], \quad s'|_{X} : \id \to L|_{X}.
\end{equation}
To pass from $\Coh(X)$ to $\Dsing(X)$, we localize with respect to $t$, and to pass from $\Dsing(X)$ to $\Dsing(X,Y')$, we localize with respect to $s'|_{X}$. The end result is that both $t$ and $s'|_{X}$ become isomorphisms, and by composing them we obtain an isomorphism $\id \to \id[2]$, so that $\Dsing(X,Y')$ is genuinely $2$-periodic.

Returning to the specific case where $X$ is a normal crossings surface with graph-like singular locus, we consider the case where $Y'$ has nodes at several points along the singular locus of $X$. In this case, we can describe $\Dsing(X,Y')$ in terms of puncturing $X$ at those points.

\begin{theorem}
  Let $Y$ be a smooth threefold, $L$ a line bundle on $Y$, and $s$ section of $L$ such that $X = s^{-1}(0)$ is a normal crossings surface with graph-like singular locus. Let $s'$ be another section of $L$ whose restriction to $\Sing(X)$ has only simple zeros at points which are not triple points. Let $Y'$ be the singular threefold constructed from the pencil spanned by $s$ and $s'$. Let $Z = \Sing(Y') \cap X$. Then there are equivalences of categories
  \begin{equation}
    \Dsing(X,Y') \cong \MF(Y,L,s)[s'^{-1}] \cong \Dsing(X\setminus Z).
  \end{equation}
  % Let $\pi : \widetilde{Y}' \to Y'$ be a resolution of $Y'$ that is an isomorphism over the nonsingular locus, and $\widetilde{X}$ be the strict transform of $X$, then there is an equivalence of categories
  % \begin{equation}
  %   \MF(Y,L,s)[s'^{-1}] \cong \Dsing(\widetilde{X}\setminus \pi^{-1}(Z)).
  % \end{equation}
\end{theorem}

This result is proved in Section \ref{sec:relative-sing}, where a connection to a resolution of $Y'$ is also discussed.

\subsection{A-model description: principal canonical bundles and their twists}
In Section \ref{sec:symplectic-conjectures}, we study the counterparts to all of these structures in terms of Fukaya categories. It turns out that some of the natural counterparts are so-called \emph{Rabinowitz wrapped Fukaya categories}, constructed by Ganatra-Gao-Venkatesh \cite{ganatra-talk} in forthcoming work. Therefore our statements must remain at the level of conjectures, but they follow from what we consider to be reasonable expectations about these categories.

When the normal crossings surface $X$ arises as the fiber of a morphism $f : Y \to \bA^{1}$ from a regular $3$-fold, the category $\Dsing(X)$ is $2$-periodic. Our previous work \cite{PS21} shows that $\Dsing(X)$ is equivalent to the Fukaya category of a Riemann surface $\Sigma$ (provided $X$ satisfies an additional orientability condition). This can no longer be true in the $L$-twisted case, and in order to handle this case we propose a way to ``geometrize'' the $2$-periodic structure. This turns out to connect to conjectures and mirror symmetry expectations about complements of divisors in compact Calabi-Yau varieties due to Lekili-Ueda \cite{lekili-ueda-milnor,lekili-ueda-complements}, and also to Jeffs \cite{jeffs} work on Fukaya categories of singular varieties. The basic idea is that, given a general type variety $Z$, we may interpret its $\Z/2\Z$-graded Fukaya category as a $2$-periodic category, and then interpret that as the $\Z$-graded Rabinowitz Fukaya category of the principal $\bC^{\times}$-bundle associated to the canonical bundle of $Z$; denote this symplectic manifold by $P(K_{Z})$. The manifold $P(K_{Z})$ carries a canonical grading (trivialization of $2c_{1}$), and we use this to (conjecturally) define a $\Z$-graded Rabinowitz Fukaya category $\RFuk^{\gr}(P(K_{Z}))$. The manifold $P(K_{Z})$ also carries a Hamiltonian $S^{1}$-action, and this induces a map $C_{*}(\Omega S^{1}) \to HH^{*}(\RFuk^{\gr}(P(K_{Z})))$, but the generator of $\Omega S^{1} \cong \Z$ is sent to an element of degree $2$. This gives rise to a subalgebra in $HH^{*}(\RFuk^{\gr}(P(K_{Z})))$ isomorphic to $k[u,u^{-1}]$, and this realizes the $2$-periodic structure on $\RFuk^{\gr}(P(K_{Z}))$.

\begin{conjecture}
  \label{conj-canonical-bundle}
  Let $Z$ be a general type symplectic manifold. Let $\Fuk(Z)$ be the Fukaya category of $Z$, regarded as a $k[u,u^{-1}]$-linear category, $\deg(u) = 2$. Let $\RFuk^{\gr}(P(K_{Z}))$ be the graded Rabinowitz Fukaya category of the principal canonical bundle of $Z$, regarded as a $k[u,u^{-1}]$-linear category via the map $k[u,u^{-1}] \to HH^{*}(\RFuk^{\gr}(P(K_{Z})))$ induced by the Hamiltonian $S^{1}$-action. Then there is a $k[u,u^{-1}]$-linear equivalence of categories
  \begin{equation}
    \Fuk(Z) \cong \RFuk^{\gr}(P(K_{Z})).
  \end{equation}
\end{conjecture}

The conjectural functor roughly takes an oriented Lagrangian $L \subset Z$ to the total space of its ``canonical bundle'' $\det_{\R} T^{*}L$; the latter is not necessarily Lagrangian, but it can be deformed to a Lagrangian provided that $L$ is balanced in $Z$. This functor may also be thought of as a kind of Orlov equivalence. In fact, the conjecture above is very similar to a theorem of M. Umut Isik \cite{umut-isik} (slightly enhanced by Arinkin-Gaitsgory \cite[Appendix H]{arinkin-gaitsgory}) stating that, for any quasi-smooth derived scheme $X$, the category $\Coh(X)$ may be realized as a \emph{graded} singularity category.

We note a connection between this conjecture and some conjectures of Lekili-Ueda: see \cite[Conjecture 1.5]{lekili-ueda-milnor} and \cite{lekili-ueda-complements}. Let $V$ be a compact Calabi-Yau variety, regarded as a symplectic manifold, let $Z \subset V$ be an ample divisor, and let $U = V \setminus Z$ be the complement, which is an exact symplectic manifold. We denote by $\Fw(U)$ and $\Fc(U)$ the wrapped Fukaya category and compact Fukaya category, respectively; these categories are $\Z$-graded owing to the trivialization of the canonical bundle of $V$. Based on announced results of Ganatra-Gao-Venkatesh \cite{ganatra-talk}, we expect that in this situation (possibly under some additional hypothesis), there is a localization sequence
\begin{equation}
  \label{ggv-localization}
  \Fc(U) \to \Fw(U) \to \RFuk^{\gr}(P(N_{Z})),
\end{equation}
where $P(N_{Z})$ is the principal bundle associated to the normal bundle of $Z$ in $V$. Now because $V$ is Calabi-Yau, we have $N_{Z} \cong K_{Z}$, and so we expect
\begin{equation}
  \Fw(U)/\Fc(U) \cong \RFuk^{\gr}(P(K_{Z})).
\end{equation}
\begin{theorem}
  Let $Z$ be an ample divisor in a compact Calabi-Yau manifold $V$, and let $U = V \setminus Z$. Suppose that the localization sequence \eqref{ggv-localization} is valid, and that Conjecture \eqref{conj-canonical-bundle} is true. Then there is a $k$-linear equivalence of categories
  \begin{equation}
    \Fw(U)/\Fc(U) \cong \Fuk(Z),
  \end{equation}
  where the right-hand side is the underlying $\Z$-graded $k$-linear category of a $k[u,u^{-1}]$-linear category.
\end{theorem}

In the case where $Z = \Sigma$ is a Riemann surface, passing from $\Sigma$ to $P(K_{\Sigma})$ gives us more room to maneuver, and we can conjecturally realize the $L$-twisted $2$-periodic categories $\Dsing(X)$ as Rabinowitz Fukaya categories of certain symplectic $4$-manifolds obtained from $P(K_{\Sigma})$ by cutting along circles in $\Sigma$ and regluing. In this process, the circles we cut along correspond to edges in the graph $G(X)$, and the regluing process involves a twist that encodes the degree of the line bundle $L$ along the corresponding component of $\Sing(X)$.
%The following statement is conditional on the descent property for Rabinowitz Fukaya categories that parallels the pants decomposition of a surface.

\begin{conjecture}
  \label{conj:intro-rfuk-graph-like}
  Let $X$ be a normal crossings surface with graph-like singular locus and orientable dual intersection complex, which arises as the zero locus of a section $s$ of a line bundle $L$ on a regular $3$-fold $Y$. Let $\Sigma$ be the Riemann surface that is dual to the singular locus of $X$. Let $M$ be the symplectic $4$-manifold obtained from $P(K_{\Sigma})$ by regluing as described above.
  %Suppose that Rabinowitz Fukaya categories satisfy descent with respect to pants decompositions of $\Sigma$.
  Then there is a $k$-linear equivalence of categories
  \begin{equation}
    \Dsing(X) \cong \RFuk^{\gr}(M).
  \end{equation}
  Furthermore, $\RFuk^{\gr}(M)$ admits an autoequivalence $\Lambda$ and a $\Lambda$-twisted $2$-periodic structure, such that under the equivalence above, tensoring with $L$ corresponds to $\Lambda$, and the $L$-twisted $2$-periodic structure on $\Dsing(X)$ matches the $\Lambda$-twisted $2$-periodic structure on $\RFuk^{\gr}(M)$.
\end{conjecture}

Lastly, we make the connection to $2$-periodic categories. We shall see that, while $P(K_{\Sigma})$ may be obtained as a punctured tubular neighborhood of a curve (Riemann surface) $\Sigma$ that arises as a divisor in a Calabi-Yau surface, the manifolds $M$ that are mirror to $L$-twisted singularity categories arise as punctured tubular neighborhoods of \emph{nodal} curves in Calabi-Yau surfaces. According to Maxim Jeffs \cite{jeffs}, the Fukaya category of a nodal curve is a localization of the Fukaya category of its smoothing. This suggests immediately that the Fukaya category of a nodal divisor in a Calabi-Yau surface should correspond to the relative singularity category $\Dsing(X,Y')$ for appropriate choice of $Y'$. Making a very precise match between these categories is challenging because there are moduli on both sides: the areas of the irreducible components of the nodal curve on the $A$-side, and the positions of the singularities in the total space on the $B$-side. Nevertheless, we can state a result in an existential form.

\begin{theorem}
  Let $s$ be a section of a line bundle $L$ on a regular $3$-fold $Y$ such that $X = s^{-1}(0)$ is normal crossings with graph-like singular locus and orientable dual intersection complex, and let $s'$ be another section which is generic with respect to $X$. Let $Y'$ be the closure of the graph of the pencil spanned by $s$ and $s'$. Let $\Sigma$ be the Riemann surface that is dual to the singular locus of $X$. Then the relative singularity category $\Dsing(X,Y')$ is equivalent to the quotient of $\Fuk(\Sigma)$ by a collection of objects supported on simple closed curves.
\end{theorem}

In Section \ref{sec:symplectic-conjectures}, we give more precise forms of this conjecture, and we explain how it works in the case of mirror symmetry for the quartic surface in $\bP^{3}$, where $\Sigma$ is the intersection of the quartic surface with the toric boundary, and $X$ is the mirror of the complement of $\Sigma$.

\subsection{Techniques: derived algebraic geometry, singular support, and descent}

In order to obtain our results, we need a theory of descent for $L$-twisted $2$-periodic categories. This means that we not only need to keep track of restriction functors, but also the way that these functors interact with the $L$-twisted $2$-periodic structures on the various categories. It turns out that derived algebraic geometry provides a convenient set of conceptual tools for doing this. Thus in Section \ref{sec:twisted-lg} we review the theory of matrix factorizations for twisted Landau-Ginzburg models from the perspective of derived algebraic geometry.

From this theory there emerges a geometric notion that both technically useful and beneficial to our intuition: the notion of \emph{singular support} of a coherent sheaf \cite{arinkin-gaitsgory}. This singular support is a closed subset of a certain variety $\gsing(X)$ which is roughly ``the total space of the degree $-1$ cotangent complex of $X$.'' The idea is that $\Coh(X)$ ``spreads out'' over $\gsing(X)$, and this helps us to understand the relationships between $\Coh(X)$, $\Perf(X)$, and $\Dsing(X)$: While $\Coh(X)$ lives on $\gsing(X)$, $\Perf(X)$ lives on $X \subset \gsing(X)$, and $\Dsing(X)$ lives on $\gsing^{\circ}(X) = \gsing(X) \setminus X$

We shall also find a direct mirror symmetry interpretation of $\gsing(X)$: in Section \ref{sec:t-duality-gsing} we show that the symplectic manifold $M$ from Conjecture \ref{conj:intro-rfuk-graph-like} is nothing but the result of applying T-duality to $\gsing^{\circ}(X)$.

In Section \ref{sec:descent-dsing-mf}, we derive the descent statements we need from the general theory of $1$-affineness \cite{gaitsgory-1-affine}.

In Section \ref{sec:graph-like}, we apply this descent theory to the case where $X$ is normal crossings with graph-like singular locus, and prove our main structural results about $\Dsing(X)$.

Section \ref{sec:relative-sing} considers the relationships between absolute singularity categories in the $L$-twisted $2$-periodic context and relative singularity categories.

In Section \ref{sec:symplectic-conjectures}, we present our conjectures on the symplectic counterparts to the story considered in the previous sections.

\subsection{Acknowledgments}
The authors wish to thank Yanki Lekili for correspondence that spurred us to work out the $L$-twisted generalization of our previous results, and which led us to new discoveries. We thank Nathan Dunfield for discussions that were helpful in understanding our constructions in Section \ref{sec:symplectic-conjectures}.

This work was presented in the Western Hemisphere Virtual Symplectic Seminar (WHVSS) in March 2022. We thank the organizers for the opportunity to speak. Sheel Ganatra and Mohammed Abouzaid provided valuable feedback about the conjectures in Section \ref{sec:symplectic-conjectures}.

 This work was completed while JP was a visitor at the Scuola Internazionale Superiore di Studi Avanzati (SISSA). He thanks SISSA for providing excellent working conditions. JP was partially supported by SISSA and the Simons Foundation.

\section{Twisted Landau-Ginzburg models and singularity categories}
\label{sec:twisted-lg}

In this section we recall some theoretical background on twisted Landau-Ginzburg models, twisted matrix factorizations, and singularity categories. We also include some discussion of the theory of singular support for coherent sheaves, since we have found this concept very helpful in visualizing the geometric constructions we shall make.

Some references that take a point of view close to the one presented here include Preygel \cite{preygel-thesis}, Arinkin-Gaitsgory \cite{arinkin-gaitsgory}, Blanc-Robalo-To\"{e}n-Vezzosi \cite{BRTV}, and Pippi \cite{pippi}. Of course the general theory of derived categories of singularities is much older. We note that the twisted case has been studied by Seidel \cite[p.~87]{seidel-subalgebras}, Polishchuk-Vaintrob \cite{polishchuk-vaintrob}, Burke-Walker \cite{burke-walker}, and quite possibly others. 

A \emph{twisted Landau-Ginzburg model} is a triple $(Y,L,s)$ consisting of a smooth variety $Y$, a line bundle $L$ on $Y$ and a section $s \in \Gamma(Y,L)$. By contrast, an untwisted Landau-Ginzburg model is one where $L$ is trivial and has been trivialized, so that $s$ may be regarded as a function $f: Y \to \bA^{1}$. There is a conceptual move that allows both the twisted and untwisted cases to be developed in a completely parallel manner: we simply regard the total space $Y$ as the \emph{base scheme} for all constructions.

\subsection{The derived group scheme associated to $(Y,L)$}
Let $(Y,L)$ be the underlying variety and line bundle of a twisted Landau-Ginzburg model. The total space of $L$ is the spectrum relative to $Y$ of the symmetric algebra of the dual $L^{\vee}$,
\begin{equation}
  \Tot(L) = \Spec_{Y}(\Sym_{\cO_{Y}}(L^{\vee})).
\end{equation}
The zero section of $L$ gives a morphism $0: Y \to \Tot(L)$. We define $\bB_{Y}$ as the derived fibered product of $0$ with itself,
\begin{equation}
  \xymatrix{
    \bB_{Y} \ar[r] \ar[d] & Y\ar[d]^-{0} \\
    Y\ar[r]^-{0} & \Tot(L).
  }
\end{equation}
Since the morphism $0$ is not flat, $\bB_{Y}$ has a genuine derived scheme structure. An explicit presentation for $\bB_{Y}$ is the spectrum relative to $Y$ of a sheaf of non-positively graded DG algebras,
\begin{equation}
  \bB_{Y} \cong \Spec_{Y}(L^{\vee}[1]\oplus \cO_{Y})
\end{equation}
the DG algebra $L^{\vee}[1]\oplus \cO_{Y}$ is a square-zero extension of $\cO_{Y}$ and has vanishing differential; with respect to a local trivialization of $L$, it is an exterior $\cO_{Y}$-algebra on a degree $-1$ generator.

The derived scheme $\bB_{Y}$ also has the structure of a group scheme over $Y$ (or what might be called a groupoid with object scheme $Y$). Let us write $\bB_{Y} = Y \times_{\Tot(L)}Y$. The composition morphism $\bB_{Y} \times \bB_{Y} \to \bB_{Y}$ is constructed as the composite
\begin{equation}
  \xymatrix{
    (Y \times_{\Tot(L)}Y)\times(Y \times_{\Tot(L)}Y) \cong Y \times_{\Tot(L)}Y \times_{\Tot(L)} Y \ar[r]^-{p_{13}}&  Y \times_{\Tot(L)} Y
    }
  \end{equation}
  The unit morphism $Y \to Y\times_{\Tot(L)}Y$ is the diagonal, and the inversion map swaps the factors in $Y\times_{\Tot(L)}Y$. These operations induce a monoidal structure on $\Indcoh(\bB_{Y})$ which we denote by $\circ$.

  Intuitively, we can think of $\bB_{Y}$ as the ``fiberwise loop space of $\Tot(L)$ based at the zero section,'' and then the structure maps are loop composition, constant loops, and reversal of the parameterization of loops. 

  \subsection{Koszul dual description}
  Now we apply Koszul duality relative to $Y$ to the sheaf of algebras $L^{\vee}[1]\oplus \cO_{Y}$ that represents $\bB_{Y}$. This algebra has a unique $\cO_{Y}$-linear augmentation, and we have
  \begin{equation}
    \Ext_{L^{\vee}[1]\oplus \cO_{Y}}(\cO_{Y},\cO_{Y}) = \Sym_{\cO_{Y}}(L[-2]).
  \end{equation}
  This lifts to an equivalence of DG categories
  \begin{equation}
    \Indcoh(\bB_{Y}) \simeq \Sym_{\cO_{Y}}(L[-2])\modules,
  \end{equation}
  and this equivalence matches the monoidal structure on $\Indcoh(\bB_{Y})$ given by loop composition with the tensor product on $\Sym_{\cO_{Y}}(L[-2])$-modules.\footnote{Because $\Sym_{\cO_{Y}}(L[-2])$ has positive-degree pieces, it cannot be regarded as the sheaf of functions on a derived scheme. It may be thought of as the ``affinization'' of a derived stack.} By passing to compact objects we get a monoidal equivalence
  \begin{equation}
    (\Coh(\bB_{Y}),\circ) \simeq (\Perf(\Sym_{\cO_{Y}}(L[-2])),\otimes).
  \end{equation}

  \subsection{Zero locus of a section}
  Now let $s \in \Gamma(Y,L)$ be a section, which we may regard as a morphism $s : Y \to \Tot(L)$. Then we consider $X = s^{-1}(0)$, or more precisely the fibered product
  \begin{equation}
    \xymatrix{
      X \ar[r] \ar[d] & Y\ar[d]^-{0} \\
      Y\ar[r]^-{s} & \Tot(L).
    }
  \end{equation}
  If $s$ is not constantly equal to zero, then $X$ is a discrete (that is, not derived) scheme. Just like $\bB_{Y}$, $X$ may be presented as the relative spectrum of a sheaf of DG algebras over $Y$,
  \begin{equation}
    X = \Spec_{Y}(\xymatrix{L^{\vee}[1] \ar[r]^{s} &\cO_{Y}}),
  \end{equation}
  where the algebra structure is the a square-zero extension of $\cO_{Y}$, and the differential is $s$ regarded as a map $L^{\vee}\to \cO_{Y}$.

  Regarding $X$ as a scheme over $Y$, it carries an action of the group scheme $\bB_{Y}$. Writing $X = Y \times_{s,0} Y$ and $\bB_{Y} = Y\times_{0,0} Y$, so that the subscripts on the $\times$ symbol denote the morphisms toward $\Tot(L)$ used to construct the fibered product, we have a map
  \begin{equation}
    \xymatrix{
      X \times \bB_{Y} = (Y\times_{s,0}Y)\times(Y\times_{0,0}) \cong Y\times_{s,0}Y\times_{0,0}Y \ar[r]^-{p_{13}} & Y\times_{s,0}Y = X.
    }
  \end{equation}
  At the level of categories, we get an action of $(\Indcoh(\bB_{Y}),\circ)$ on $\Indcoh(X)$, or equivalently an action of $(\Sym_{\cO_{Y}}(L[-2])\modules,\otimes)$ on $\Indcoh(X)$. This action is finite in the sense that it comes from an action on compact objects, meaning that $(\Perf(\Sym_{\cO_{Y}}(L[-2])),\otimes)$ acts on $\Coh(X)$.

\subsection{The singularity category and matrix factorizations}
We shall define a category $\MF(Y,L,s)$ to be a localization of $\PreMF(Y,L,s)$ and show that it is equivalent to $\Dsing(X)$. 

  Following Preygel's notation, we denote by $\PreMF(Y,L,s)$ the category $\Coh(X)$ \emph{equipped with} the structure of a module category over $\Perf(\Sym_{\cO_{Y}}(L[-2])),\otimes)$. The notation $\PreMF^{\infty}(Y,L,s)$ stands for $\Indcoh(X)$ equipped with the structure of a module category over $(\Sym_{\cO_{Y}}(L[-2])\modules,\otimes)$.

The $\cO_{Y}$-algebra $\Sym_{\cO_{Y}}(L[-2])$ has a localization given by ``inverting $L[-2]$.'' We denote this algebra by $\cR$; it has graded pieces in every positive and negative even degree, and the degree $2k$ piece is $L^{\otimes k}$. The algebra $\cR$ is to $\Sym_{\cO_{Y}}(L[-2])$ as the ring of Laurent polynomials is to the ring of polynomials. Since $\Sym_{\cO_{Y}}(L[-2])$ is a subalgebra of $\cR$, we have a symmetric monoidal functor $\Sym_{\cO_{Y}}(L[-2])\modules \to \cR\modules$. We then define
\begin{align}
  \MF^{\infty}(Y,L,s) &= \PreMF^{\infty}(Y,L,s) \otimes_{\Sym_{\cO_{Y}}(L[-2])\modules} \cR\modules\\
  \MF(Y,L,s) &= \PreMF(Y,L,s) \otimes_{\Perf(\Sym_{\cO_{Y}}(L[-2]))} \Perf(\cR)
\end{align}
Note that $\MF^{\infty}(Y,L,s)$ and $\MF(Y,L,s)$ are $\cR$-linear categories in the appropriate senses.

On the other hand, we have the (idempotent completed) quotient categories
\begin{align}
  \Dsing^{\infty}(X) &= \Indcoh(X)/\Qcoh(X)\\
  \Dsing(X) &= \Coh(X)/\Perf(X)
\end{align}

There are evident functors $\Indcoh(X) \to \MF^{\infty}(Y,L,s)$ and $\Coh(X) \to \MF(Y,L,s)$. It turns out that the kernels of these functors are $\Qcoh(X)$ and $\Perf(X)$ respectively, and thus they induce $k$-linear equivalences of categories
\begin{align}
  \Dsing^{\infty}(X) &\cong \MF^{\infty}(Y,L,s),\\
  \Dsing(X) &\cong \MF(Y,L,s).
\end{align}
We wish to emphasize the point that, while these are equivalences of $k$-linear categories, the right-hand sides carry a richer structure of an $\cR$-linear category, and this $\cR$-linear structure actually depends in a somewhat subtle way on the choice of line bundle $L$ and the section $s$, even though the underlying $k$-linear category only depends on the underlying scheme $X$.

\subsection{The relationship between $\cR$-actions and $L$-twisted $2$-periodic structures}
The action of $\Sym_{\cO_{Y}}(L[-2])\modules$ on $\PreMF^{\infty}(Y,L,s)= \Indcoh(X)$ and the action of $\cR\modules$ on $\MF^{\infty}(Y,L,s)$ admit more explicit descriptions in terms of natural transformations from $\id \to \otimes L^{\vee}[2]$.

First observe that, since $X$ is a scheme over $Y$, $\Indcoh(X)$ admits an action of $(\Qcoh(Y),\otimes)$, while $\Coh(X)$ admits an action of $(\Perf(Y),\otimes)$. A natural question is then, what is required to extend the action $\Qcoh(Y)$ to an action of $\Sym_{\cO_{Y}}(L[-2])\modules$, or to extend the action of $\Perf(Y)$ to an action of $\Perf(\Sym_{\cO_{Y}}(L[-2]))$?

We would like to give a geometric answer to this question, but since $\Sym_{\cO_{Y}}(L[-2])$ does not correspond to a derived scheme, this is slightly awkward. A more geometric analog is the algebra $\Sym_{\cO_{Y}}(L)$, where we have simply redefined $L$ to live in degree zero, and we shall consider this first. Observe that 
\begin{equation}
  \Tot(L^{\vee}) = \Spec_{Y}(\Sym_{\cO_{Y}}(L))
\end{equation}
Another construction of $\Tot(L^{\vee})$ is as follows. The line bundle $L^{\vee}$ over $Y$ is classified by a map $Y \to \pt/\bG_{m}$. The universal line bundle\footnote{The universal $\bG_{m}$-principal bundle is $\bG_{m}/\bG_{m} \cong \pt \to \pt/\bG_{m}$; it is obtained from the universal line bundle $\bA^{1}/\bG_{m} \to \pt/\bG_{m}$ by removing $0 \in \bA^{1}$.} over $\pt/\bG_{m}$ is given by the map $\bA^{1}/\bG_{m} \to \pt/\bG_{m}$, and $\Tot(L^{\vee})$ is recovered as the fibered product
\begin{equation}
  \xymatrix{
    \Tot(L^{\vee}) \ar[r]\ar[d] & \bA^{1}/\bG_{m} \ar[d]\\
    Y \ar[r] & \pt/\bG_{m}.
    }
  \end{equation}
  Therefore we have a tensor product decomposition
  \begin{equation}
    \Sym_{\cO_{Y}}(L)\modules = \Qcoh(\Tot(L^{\vee})) \cong \Qcoh(Y) \otimes_{\Qcoh(\pt/\bG_{m})} \Qcoh(\bA^{1}/\bG_{m}).
  \end{equation}

  The structure maps in this tensor product may be explicitly described:
  \begin{itemize}
  \item The category $\Qcoh(\pt/\bG_{m})$ is generated by all the equivariant lifts of $k$; denoting these by $k(n)$ for $n \in \Z$, we have $k(n) \otimes k(m) = k(n+m)$. The objects $k(n)$ and $k(m)$ are orthogonal if $n \neq m$.  To give an action of $\Qcoh(\pt/\bG_{m})$ on a category $\cC$ is the same as giving an autoequivalence of $\cC$. The functor $\Qcoh(\pt/\bG_{m}) \to \Qcoh(Y)$ sends $k(1)$ to the monoidally invertible object $L^{\vee}$.
  \item The category $\Qcoh(\bA^{1}/\bG_{m})$ is generated by all the equivariant lifts of $\cO_{\bA^{1}}$, and the space of morphisms $\cO_{\bA^{1}}(n) \to \cO_{\bA^{1}}(m)$ is one dimensional when $m \geq n$. To give an action of $\Qcoh(\bA^{1}/\bG_{m})$ on a category $\cC$ is the same as giving an autoequivalence $\Phi$ of $\cC$ and a natural transformation $t : \id \to \Phi$. The functor $\Qcoh(\pt/\bG_{m}) \to \Qcoh(\bA^{1}/\bG_{m})$ sends $k(n)$ to $\cO_{\bA^{1}}(n)$.
  \item The functor $\Qcoh(Y) \to \Qcoh(\Tot(L^{\vee}))$ is the pullback under the projection $\pi : \Tot(L^{\vee}) \to Y$.
  \item The functor $\Qcoh(\bA^{1}/\bG_{m}) \to \Qcoh(\Tot(L^{\vee}))$ sends $\cO_{\bA^{1}}(1)$ to $\pi^{*}(L^{\vee})$ (as it must in order for the diagram to commute), and it sends the canonical morphism $\cO_{\bA^{1}}(0) \to \cO_{\bA^{1}}(1)$ to the tautological section of $\pi^{*}(L^{\vee})$ over $\Tot(L^{\vee})$
  \end{itemize}

  The structure of $\Sym_{\cO_{Y}}(L[-2])\modules$ admits an analogous tensor product decomposition,
  \begin{equation}
    \Sym_{\cO_{Y}}(L[-2])\modules \cong \Qcoh(Y) \otimes_{\Qcoh(\pt/\bG_{m})} \Qcoh(\bA^{1}/\bG_{m}),
  \end{equation}
  where the structure maps are as follows:
  \begin{itemize}
  \item The functor $\Qcoh(\pt/\bG_{m}) \to \Qcoh(Y)$ sends $k(1)$ to $L^{\vee}[2]$. (Note that this functor does not arise from pullback by a morphism $Y \to \pt/\bG_{m}$ since it does not preserve connective objects.)
  \item The functor $\Qcoh(Y) \to \Sym_{\cO_{Y}}(L[-2])\modules$ is a pullback which we again denote by $\pi^{*}$. We have $\pi^{*}(\cO_{Y}) = \Sym_{\cO_{Y}}(L[-2])$ and
    \begin{equation}
      \pi^{*}(L^{\vee}[2]) = L^{\vee}[2] \otimes \Sym_{\cO_{Y}}(L[-2]) = L^{\vee}[2] \oplus \cO_{Y} \oplus L[-2] \oplus L^{\otimes 2}[-4] \oplus \cdots
    \end{equation}
  \item The functor $\Qcoh(\bA^{1}/\bG_{m}) \to \Sym_{\cO_{Y}}(L[-2])\modules$ sends $\cO_{\bA^{1}}(1)$ to $\pi^{*}(L^{\vee}[2])$, and sends the canonical morphism $\cO_{\bA^{1}}(0) \to \cO_{\bA^{1}}(1)$ to the tautological morphism $\Sym_{\cO_{Y}}(L[-2]) \to \pi^{*}(L^{\vee}[2])$.
  \end{itemize}
  This decomposition makes evident the following fact: Given a category $\cC$ with a $\Qcoh(Y)$-action, the problem of extending it to an action of $\Sym_{\cO_{Y}}(L[-2])\modules$ is equivalent to finding a natural transformation $t : \id \to \otimes L^{\vee}[2]$.
  \begin{remark}
    One may ask why we have formulated this theory in terms of actions of $\Sym_{\cO_{Y}}(L[-2])$ rather than just working with natural transformations $t : \id \to \otimes L^{\vee}[2]$ directly. The reason is because we later need to study the descent properties of these constructions. According to Gaitsgory's theory of $1$-affineness, descent properties with respect to coverings of $Y$ are governed by the action of $\Qcoh(Y)$. The action of $\Sym_{\cO_{Y}}(L[-2])\modules$ neatly packages together the action of $\Qcoh(Y)$ and the action of the natural transformation $t$.
  \end{remark}

  \begin{remark}
    An alternative way of decomposing $\Sym_{\cO_{Y}}(L[-2])$ as a tensor product involves changing the grading on the category $\Qcoh(\bA^{1}/\bG_{m})$. Note that $\Qcoh(\bA^{1})$ is equivalent to the cocompletion of the Fukaya category of an annulus with a single stop on one boundary component, graded by a line field that has winding number zero around each boundary component; it is generated by a single arc $C$ connecting the two boundaries. The equivariant category $\Qcoh(\bA^{1}/\bG_{m})$ is the Fukaya category of the universal cover of this annulus. We may however use a different line field on the annulus that gives the generating morphism $C \to C$ degree $2$, and then pass to the universal cover. Let us call $\cF$ the Fukaya category of the universal cover with this new grading. Then we have
    \begin{equation}
      \Sym_{\cO_{Y}}(L[-2])\modules \cong \Qcoh(Y)\otimes_{\Qcoh(\pt/\bG_{m})} \cF,
    \end{equation}
    where now the functor $\Qcoh(\pt/\bG_{m}) \to \Qcoh(Y)$ sends $k(1)$ to $L^{\vee}$ as in the unshifted case of $\Tot(L^{\vee})$.
  \end{remark}

  The next thing to do is to understand the passage from $\Sym_{\cO_{Y}}(L[-2])$ to $\cR$, the localization that inverts $L[-2]$. In the unshifted case of $\Tot(L^{\vee}) = Y \times_{\pt/\bG_{m}} \bA^{1}/\bG_{m}$, this corresponds to deleting the zero section of $\Tot(L^{\vee})$, and passing to
  \begin{equation}
    \Tot(L^{\vee})\setminus Y = Y \times_{\pt/\bG_{m}}\bG_{m}/\bG_{m}
  \end{equation}
  Although $\bG_{m}/\bG_{m}$ can be thought of as a point, it is easier to describe the pullback functor $\Qcoh(\bA^{1}/\bG_{m})\to \Qcoh(\bG_{m}/\bG_{m})$ by regarding $\Qcoh(\bG_{m}/\bG_{m})$ as generated by all equivariant lifts $\cO_{\bG_{m}}(n)$, such that there is a one-dimensional space of morphisms between any pair of lifts. Thus the passage from $\bA^{1}/\bG_{m}$ to $\bG_{m}/\bG_{m}$ turns the morphism $\cO(0) \to \cO(1)$ into an isomorphism. In the shifted case, we also have
  \begin{equation}
    \cR \cong \Qcoh(Y)\otimes_{\Qcoh(\pt/\bG_{m})} \Qcoh(\bG_{m}/\bG_{m})
  \end{equation}
  where the generating object of $\Qcoh(\pt/\bG_{m})$ maps to $L^{\vee}[2]$, and so the effect of passing from $\Sym_{\cO_{Y}}(L[-2])$ to $\cR$ is to localize to make the natural transformation $t : \id \to \otimes L^{\vee}[2]$ into an isomorphism. In summary
  \begin{equation}
    \MF^{\infty}(Y,L,s) = \Indcoh(X) \otimes_{\Sym_{\cO_{Y}}(L[-2])\modules} \cR\modules 
  \end{equation}
  is nothing but the localization of $\Indcoh(X)$ with respect to $t : \id \to \otimes L^{\vee}[2]$. In other words, $\MF^{\infty}(Y,L,s)$ carries an $L$-twisted $2$-periodic structure, which is encoded by the action of $\cR\modules$.

  \begin{remark}
    This whole discussion could be written in a more geometric way if we could regard $\Sym_{\cO_{Y}}(L[-2])\modules$ as $\Qcoh(\Tot(L^{\vee}[2]))$. Since $\Sym_{\cO_{Y}}(L[-2])$ is not connective, $\Tot(L^{\vee}[2])$ is a stacky object. Starting again from the data $(Y,L,s)$, the theory presented above ``spreads out'' $\Indcoh(X)$ over $\Tot(L^{\vee}[2])$. Passing from $\Indcoh(X)$ to $\MF^{\infty}(Y,L,s)$ is restricting to the complement of the zero section $\Tot(L^{\vee}[2])\setminus Y$. This idea can be made more precise using the theory of singular support presented in the next subsection.
  \end{remark}

  \subsection{Singular support}
  \label{sec:sing-supp}
The structures on $\Coh(X)$ and $\MF(Y,L,s)$ studied above depend in a somewhat subtle way on the choice of $(Y,L,s)$. It is therefore useful to have a parallel, if slightly weaker, theory that is intrinsic to the singular scheme $X$ itself (without any embedding into an ambient $Y$). We shall use for this purpose the theory of singular support for coherent sheaves as developed by Arinkin-Gaitsgory \cite{arinkin-gaitsgory}, following earlier work of Benson-Iyengar-Krause \cite{BIK}.

A DG scheme $X$ is called \emph{quasi-smooth} if its cotangent complex $T^{*}(X)$ is a perfect complex of Tor-amplitude $[-1,0]$; this is equivalent to saying that the fiber $T^{*}_{x}(X)$ at any geometric point $x$ of $X$ is acyclic in degrees below $-1$. The class of quasi-smooth DG schemes includes classical local complete intersections such as the hypersurfaces considered above.

Let $X$ be a quasi-smooth classical scheme. There is a classical scheme $\gsing(X)$, which is, roughly speaking, a stratified vector bundle over $X$ whose fibers are of the form $H^{-1}(T^{*}_{x}(X))$, the degree $-1$ cohomology of the tangent complex. Since the cotangent complex $T^{*}(X)$ is perfect and has amplitude $[-1,0]$, it has a dual $T(X)$ which is also perfect and has amplitude $[0,1]$, this is the tangent complex. We may then define
\begin{equation}
  \gsing(X) = (\Spec_{X}(\Sym_{\cO_{X}}(T(X)[1])))^{\clpart}
\end{equation}
to be the underlying classical scheme of the relative spectrum of the sheaf of DG algebras $\Sym_{\cO_{X}}(T(X)[1])$. On the other hand, since we are shifting and taking the classical part, we may write this purely in terms of classical algebras as
\begin{equation}
  \gsing(X) = \Spec_{X}(\Sym_{\cO_{X}}(H^{1}(T(X))))
\end{equation}
where $H^{1}(T(X))$ is regarded as a coherent sheaf on $X$. Note that there is a natural $\bG_{m}$ action on $\gsing(X)$ that rescales the fibers of the map $\gsing(X) \to X$.

Let us compute $\gsing(X)$ when $X$ arises from a triple $(Y,L,s)$ as above. $X$ is presented as DG scheme over $Y$ by
\begin{equation}
  X = \Spec_{Y}(\xymatrix{L^{\vee}[1] \ar[r]^{s} &\cO_{Y}}).
\end{equation}
The cotangent complex of $X$ is then
\begin{equation}
  \label{eq:cotangent-complex}
T^{*}(X) = \left(  \xymatrix{L^{\vee}|_{X}[1] \ar[r]^{ds} & T^{*}(Y)|_{X}}\right).
\end{equation}
Now the traditional singular locus of $X$, $\Sing(X)$, is the locus where $ds$ vanishes, 
and so the cohomology in degree $-1$ is
\begin{equation}
  H^{-1}(T_{x}^{*}(X)) \cong \begin{cases}
    L^{\vee}_{x}, & \text{if $x \in \Sing(X)$},\\
    0, & \text{if $x \not\in \Sing(X)$}.
  \end{cases}
\end{equation}
Thus we may write
\begin{equation}
  \gsing(X) = X \cup \Tot(L^{\vee}|_{\Sing(X)}).
\end{equation}
This expression shows that there are embeddings $\gsing(X) \to \Tot(L^{\vee}|X) \to \Tot(L^{\vee})$, where the last is the total space of $L^{\vee}$ as a vector bundle over $Y$.

Now we wish to make a totally pedantic point, but one which has nevertheless given us a lot of trouble. The scheme $\gsing(X)$ is intrinsically determined by the geometry of $X$, but the isomorphism $H^{-1}(T^{*}_{x}(X)) \cong L^{\vee}_{x}$ for $x \in \Sing(X)$, and hence the embedding $\gsing(X) \to \Tot(L^{\vee})$, actually depends on the choice of the section $s$. Suppose that we rescale $s$ by a local invertible function $g$. Then the differential appearing in \eqref{eq:cotangent-complex} is rescaled by $g$ as well. Of course, this is isomorphic to the original complex, as it must be since $T^{*}(X)$ is intrinsic to $X$, but the isomorphism involves rescaling $L^{\vee}|_{X}$ by $g^{-1}$.

The idea of singular support for coherent sheaves is that, to any object $\cF \in \Coh(X)$, we may associate a subscheme $\SingSupp(\cF)$ of $\gsing(X)$ that measures how $\cF$ interacts with the singularities of $X$. There is an explicit point-wise description of $\SingSupp(\cF)$ as follows. Let $(x,\xi) \in \gsing(X)$, where $\xi \in H^{-1}(T^{*}_{x}(X))$:
\begin{itemize}
\item If $\xi = 0$, then $(x,\xi) \in \SingSupp(\cF)$ if and only if $x \in \Supp(\cF)$.
\item If $\xi \neq 0$, perform the following test to determine if $(x,\xi) \in \SingSupp(\cF)$: Locally near $x$, write $X$ as a complete intersection $U \times_{V} \pt$ for some morphism of smooth schemes $q : U \to V$ and some point $\pt \in V$. Let $f$ be a local function on $V$ that vanishes at $\pt$ and whose differential is $\xi$. Let $X' \subset U$ be the subscheme defined by the vanishing of $q^{*}(f)$. Let $i : X \to X'$ be the inclusion. Then $(x,\xi)$ \emph{does not belong} to $\SingSupp(\cF)$ if and only if $i_{*}(\cF)$ is perfect on a Zariski neighborhood of $x$.
\end{itemize}
Note that it is clear from this definition that $\SingSupp(\cF)$ is conical, that is, invariant under rescaling the fibers of $\gsing(X) \to X$. 

For the particular case where $X$ comes from a triple $(Y,L,s)$, we may describe $\SingSupp(\cF)$ in terms of the ordinary notion of support in commutative algebra, as long as we work in an affine patch. So suppose that $X,Y$ are affine and that $L$ is trivialized. Then the algebra $\Sym_{\cO_{Y}}(L[-2])$ becomes $\cO(Y)[u]$, the ring of regular functions on $Y$ with an additional degree $2$ variable $u$ adjoined. Also $\gsing(X)$ embeds into $\Tot(L^{\vee}) \cong Y \times \bA^{1}$. The variable $u$ is to be interpreted as a coordinate on the $\bA^{1}$ factor, even though it has degree $2$.

 Let $\cF \in \Coh(X)$. The $\Perf(\cO(Y)[u])$-action on $\Coh(X)$ induces on the total Ext algebra $\Ext^{\bullet}(\cF,\cF)$ the structure of an $\cO(Y)[u]$-module. The support of this module (forgetting the grading) is a subscheme of $Y\times \bA^{1}$ that coincides with $\SingSupp(\cF) \subset \gsing(X) \subset Y \times \bA^{1}$.

 When $X$ and $Y$ are not affine, and $L$ is not trivial, these local descriptions of the singular support patch together. Alternatively, we could consider sheafifying the total Ext algebra $\Ext^{\bullet}(\cF,\cF)$ to obtain a sheaf of $\Sym_{\cO_{Y}}(L[-2])$-modules, and take the support of that. For further details on these constructions we refer the interested reader to \cite{arinkin-gaitsgory}.

 The basic fact about singular support that connects to singularity categories is this: an object $\cF \in \Coh(X)$ is perfect if and only if $\SingSupp(\cF) \subset X \subset \gsing(X)$, that is, its singular support is contained in the zero section. Thus the passage from $\Coh(X)$ to $\Dsing(X)$ corresponds to deleting the zero section in $\gsing(X)$, just as tensoring with $\cR$ corresponds to deleting the ``zero section'' in $\Tot(L^{\vee}[2])$. We see here that this zero section is something intrinsically constructed from $X$.

 \begin{remark}
   \label{rem:efimov}
   This perspective clarifies, at least in the authors' minds, the relationship between $\Dsing(X)$ and the \emph{categorical formal punctured neighborhood of infinity} defined by Efimov \cite{efimov-punctured}. This construction begins with an arbitrary smooth DG category $\cB$ produces a new DG category $\widehat{\cB}_{\infty}$ that measures the failure of $\cB$ to be proper. Efimov shows that when this construction is applied to $\Coh(X)$ for a proper scheme $X$, the result is equivalent to $\Dsing(X)^{\op}$. On the mirror side, when this construction is applied to the wrapped Fukaya category $\Fw(U)$ of a Weinstein manifold $U$, it is supposed to produce a category closely related to, and possibly equivalent to, $\Fw(U)/\Fc(U)$ and the Rabinowitz Fukaya category $\RFuk(U)$.

   The appearance of $\Dsing(X)^{\op}$ as the ``formal punctured neighborhood of infinity'' of $\Coh(X)$ is somewhat unintuitive, since $X$, being proper, does not have an ``infinity'' to speak of. Thinking in terms of $\gsing(X)$ makes things clearer. Although $X$ may be proper, $\gsing(X)$ is not proper when $X$ is singular. The objects of $\Coh(X)$ with proper singular support are precisely the perfect complexes, so $\Dsing(X)$ really is the restriction of $\Coh(X)$ to a neighborhood of infinity in $\gsing(X)$.

   We should clarify that at present we do not know how to make a direct rigorous connection between the theory of singular support and Efimov's definition; this would be an interesting topic for another paper.
 \end{remark}

\subsection{Where are the matrix factorizations?}
The reader may now notice that at no point have we considered what would rightly be called ``matrix factorizations,'' in the sense of factorizations of scalar operators as matrices, or modules over curved algebras. Genuine matrix factorizations arise when we take the Koszul dual of the presentation of $X$ as scheme over $Y$
\begin{equation}
  X = \Spec_{Y}(\xymatrix{L^{\vee}[1] \ar[r]^{s} &\cO_{Y}}).
\end{equation}
Recall that $\bB_{Y}$ has the same presentation but with zero differential, and its Koszul dual is $\Sym_{\cO_{Y}}(L[-2])$. The presence of the differential in the DG algebra for $X$ means that this DG algebra does not admit an $\cO_{Y}$-linear augmentation (the $k$-linear augmentations are the $k$-points of $X$, so there is no augmentation over the generic point of $Y$). The lack of an augmentation means that the Koszul dual is curved. The Koszul dual of $X$ is the sheaf of curved algebras
\begin{equation}
  (\Sym_{\cO_{Y}}(L[-2]),s)
\end{equation}
where the curvature is $s \in \Gamma(Y,L)$ regarded as a degree $2$ element of the underlying algebra. The curved modules over this curved algebra are in some sense the genuine matrix factorizations, though we still need to invert $L[-2]$ in order to get a model for $\Dsing(X)$.

While the whole theory could be developed along these lines, there are perhaps a few reasons for preferring the algebra $\cA = \xymatrix{L^{\vee}[1] \ar[r]^{s} &\cO_{Y}}$ over its Koszul dual $\cA^{!} = (\Sym_{\cO_{Y}}(L[-2]),s)$, namely
\begin{itemize}
\item $\cA$ is coherent (even perfect) as a $\cO_{Y}$-module, while $\cA^{!}$ is only quasicoherent,
\item $\cA$ is not curved, while $\cA^{!}$ is curved,
\item $\cA$ is connective, while $\cA^{!}$ is not.
\end{itemize}
On the other hand, the presentation in terms of $\cA^{!}$ makes the action of $\Perf(\Sym_{\cO_{Y}}(L[-2]))$ evident, since the tensor product of a $\Sym_{\cO_{Y}}(L[-2])$-module of curvature $s$ with one of curvature $0$ has curvature $s$. The perspective we have taken (following Preygel) is a kind of middle ground, where we work with $\Coh(X)$ and $\cA$ directly, but the action of $\Perf(\Sym_{\cO_{Y}}(L[-2]))$ somehow ``remembers'' that objects of $\Coh(X)$ can be represented as curved modules.

\section{Descent for $\Dsing$ and $\MF$}
\label{sec:descent-dsing-mf}
In this section we shall study the descent properties of $\Dsing$ for schemes and $\MF$ for twisted Landau-Ginzburg models. The results are not entirely new and can be found in some form in the work of Preygel \cite{preygel-thesis} (for the untwisted case), Blanc-Robalo-To\"{e}n-Vezzosi \cite{BRTV}, and Pippi \cite{pippi}. However, our presentation differs from those references be emphasizing the notion of $1$-affineness.

\subsection{Sheaves of categories and the notion of $1$-affineness}
The heading of this subsection is the title of Gaitsgory's paper \cite{gaitsgory-1-affine}, some of the results of which we now recall.

A sheaf of categories $\cC$ on a prestack $Y$ is, at least roughly, a functorial assignment, for every affine DG scheme $S$ over $Y$, of a DG category $\Gamma(S,\cC)$ that carries a monoidal action of the category $\Qcoh(S)$. The functoriality is expressed by the requirement that, for a map $f : S_{1} \to S_{2}$ of affine schemes over $Y$, there is an equivalence
\begin{equation}
  \Qcoh(S_{1})\otimes_{\Qcoh(S_{2})} \Gamma(S_{2},\cC) \to \Gamma(S_{1},\cC).
\end{equation}
Given a sheaf of categories on $Y$, we may form its global sections category $\Gamma(Y,\cC)$, which is acted on by $\Qcoh(Y)$. This gives a functor from sheaves of categories over $Y$ to $\Qcoh(Y)$-module categories.

There is also a functor from $\Qcoh(Y)$-module categories to sheaves of categories given by tensoring. Given a $\Qcoh(Y)$-module category $C$, we define a sheaf of categories $\widetilde{C}$ by setting, for $S \to Y$ an affine DG scheme over $Y$, 
\begin{equation}
  \Gamma(S,\widetilde{C}) = \Qcoh(S) \otimes_{\Qcoh(Y)} C.
\end{equation}
This localization functor is the left adjoint of the global sections functor, and the unit of the adjunction gives a map $C \to \Gamma(S,\widetilde{C})$. The prestack $Y$ is called \emph{$1$-affine} if the global sections functor and the localization functor are inverse equivalences between the category of sheaves of categories on $Y$ and the category of $\Qcoh(Y)$-module categories. If $Y$ is $1$-affine, the map $C \to \Gamma(S,\widetilde{C})$ is always an equivalence.

According to Gaitsgory's Theorem 2.1.1, every quasi-compact quasi-separated DG scheme is $1$-affine. This class includes all of the schemes we wish to study in this paper.

To say that $Y$ is $1$-affine means that any $\Qcoh(Y)$-module category $C$ satisfies a form of descent, in the sense that $C$ may be reconstructed from its localizations at affine schemes over $Y$. If $Y$ is a scheme (admits a Zariski atlas), this property is Zariski descent. Very fortunately, by results of Lurie \cite[Theorem 5.4]{DAG-XI} and Gaitsgory \cite[Theorem 1.5.2]{gaitsgory-1-affine}, any sheaf of categories \emph{automatically} satisfies descent with respect to certain \emph{finer} topologies, including the \'{e}tale topology and the fppf topology (our main interest is in the \'{e}tale topology).

\subsection{The sheaves $\Qcoh$, $\Indcoh$, and $\Dsing^\infty$}
Let $X$ be a prestack for which $\Indcoh$ can be defined (e.g. a Noetherian DG scheme). According to Gaitsgory \cite{gaitsgory-indcoh}, both $\Qcoh$ and $\Indcoh$ define sheaves of categories on $Y$. In particular, $\Indcoh$ is a module over $\Qcoh$, and the natural functor $\Qcoh \to \Indcoh$ (which on compact objects is the inclusion $\Perf \to \Coh$) is a map of sheaves of categories. We have defined the singularity category of $Y$ to be the cofiber
\begin{equation}
  \Dsing^{\infty}(X) = \cofiber(\Qcoh(X) \to \Indcoh(X))
\end{equation}
The functor $\Qcoh(X) \to \Indcoh(X)$ is a map of $\Qcoh(X)$-module categories, so $\Dsing^{\infty}(X)$ is a $\Qcoh(X)$-module category. The localization of this module is the sheaf of categories that assigns to an affine $S \to X$ the category
\begin{align}
  & \Qcoh(S) \otimes_{\Qcoh(X)} \Dsing^{\infty}(X)\\
  &\cong \Qcoh(S) \otimes_{\Qcoh(X)} \cofiber(\Qcoh(X) \to \Indcoh(X))\\
  &\cong  \cofiber(\Qcoh(S) \otimes_{\Qcoh(X)}\Qcoh(X) \to \Qcoh(S) \otimes_{\Qcoh(X)}\Indcoh(X))\\
  &\cong \cofiber(\Qcoh(S) \to \Indcoh(S))\\
  &\cong \Dsing^{\infty}(S).
\end{align}
To go from the second line to the third, we have commuted a colimit (tensor product) with a colimit (cofiber). In the next step, we use the fact that $\Qcoh$ and $\Indcoh$ are sheaves of categories.

The only question that remains is whether the global sections of the sheaf
\begin{equation}
  (S\to X) \mapsto \Dsing^{\infty}(S)
\end{equation}
coincides with $\Dsing^{\infty}(X)$. If we are willing to assume that $X$ is $1$-affine, this is automatically true, and thus $\Dsing^{\infty}(X)$ satisfies (Zariski, \'{e}tale, fppf) descent. We summarize this discussion in the following theorem.
\begin{theorem}
  \label{thm:descent-dsing}
  Let $X$ be a quasi-compact quasi-separated Noetherian DG scheme. Then $\Dsing^{\infty}(X)$ satisfies \'{e}tale descent over $X$: if $c : \coprod_{i} U_{i} \to X$ is an \'{e}tale covering, then there is an equivalence of categories
  \begin{equation}
    \Dsing^{\infty}(X) \to \Tot \left(\xymatrix{ \prod_{i}\Dsing^{\infty}(U_{i}) \ar@<0.5ex>[r] \ar@<-0.5ex>[r] & \prod_{i,j}\Dsing^{\infty}(U_{i}\times_{X} U_{j}) \ar@<1ex>[r] \ar[r] \ar@<-1ex>[r] & \cdots} \right) 
  \end{equation}
  (Here $\Tot$ denotes the totalization of a cosimplicial DG category.)
\end{theorem}

\begin{remark}
  \label{rem:recollement}
  There is an alternative proof that $\Dsing^{\infty}$ satisfies descent based on recollement. Suppose that $X$ is a separated Noetherian classical scheme (this already implies that $X$ is $1$-affine, so the proof in this remark is not more general than the one given above). A result of Krause \cite{krause}, recast into the language of $\Indcoh$, states that there is a recollement diagram
  \begin{equation}
    \xymatrix{
      \Dsing^{\infty}(X)  \ar@<0ex>[r]  & \Indcoh(X) \ar@<-1ex>[l]\ar@<0ex>[r]\ar@<1ex>[l] & \Qcoh(X) \ar@<-1ex>[l]\ar@<1ex>[l]
    }
  \end{equation}
  where the functor $\Indcoh(X) \to \Qcoh(X)$ is the right adjoint of the functor $\Qcoh(X) \to \Indcoh(X)$ considered above. This recollement may thought of as being associated to the decomposition
  \begin{equation}
    \gsing(X) = X \cup (\gsing(X) \setminus X)
  \end{equation}
  of $\gsing(X)$ into a closed subscheme and its complementary open subscheme: $\Qcoh(X)$ lives on $X$, and $\Dsing^{\infty}(X)$ lives on $\gsing(X) \setminus X$.

  In this recollement, $\Dsing^{\infty}(X)$ is presented as the fiber (rather than cofiber) of a functor $\Indcoh(X) \to \Qcoh(X)$. Because limits commute with limits, taking the fiber commutes with the limit over the \v{C}ech diagram of a covering $c : \coprod_{i} U_{i}\to X$. Thus descent for $\Dsing^{\infty}$ follows from descent for $\Indcoh$ and $\Qcoh$ for this class of schemes.
\end{remark}

\subsection{$\MF$ for twisted Landau-Ginzburg models as a sheaf of categories}
Now we consider a twisted Landau-Ginzburg model $(Y,L,s)$, with $X = s^{-1}(0)$. We recall that
\begin{equation}
  \MF^{\infty}(Y,L,s) = \Indcoh(X) \otimes_{\Sym_{\cO_{Y}}(L[-2])\modules} \cR\modules
\end{equation}
where $\cR$ is the localization of $\Sym_{\cO_{Y}}(L[-2])$ that inverts $L[-2]$. The underlying $k$-linear category of $\MF^{\infty}(Y,L,s)$ is $\Dsing^{\infty}(X)$, so it satisfies a form of descent with respect to coverings of $X$, but we would like to show that $\MF^{\infty}$ actually satisfies descent as an $\cR$-linear category with respect to coverings of $Y$ that are compatible with the extra data $(L,s)$.

First we observe that $\Indcoh(X)$, $\Sym_{\cO_{Y}}(L[-2])\modules$, and $\cR\modules$ are all $\Qcoh(Y)$-module categories. In fact, each one is determined by a sheaf of DG algebras on $Y$ (respectively, the Koszul DGA of $X$, $\Sym_{\cO_{Y}}(L[-2])$, and $\cR$). The tensor product that defines $\MF^{\infty}(Y,L,s)$ is furthermore a $\Qcoh(Y)$-linear construction. Now let $f : S \to Y$ be an affine DG scheme mapping to $Y$. Then from the fact that $\Indcoh$ is a sheaf of categories \cite{gaitsgory-indcoh}, we have
\begin{align}
  \Qcoh(S) \otimes_{\Qcoh(Y)} \Indcoh(X) &\cong \Indcoh(S \times_{Y}X)\\
  \Qcoh(S) \otimes_{\Qcoh(Y)} \Sym_{\cO_{Y}}(L[-2])\modules &\cong \Sym_{\cO_{S}}(f^{*}L[-2])\modules\\
  \Qcoh(S) \otimes_{\Qcoh(Y)} \cR\modules &\cong f^{*}\cR\modules.
\end{align}
Note that $f^{*}\cR\modules$ carries an action $\cR\modules$.
From this we deduce that
\begin{align}
  \Qcoh(S)\otimes_{\Qcoh(Y)} \MF^{\infty}(Y,L,s) &\cong \Indcoh(S \times_{Y}X) \otimes_{\Sym_{\cO_{S}}(f^{*}L[-2])\modules}f^{*}\cR\modules\\ &= \MF^{\infty}(S,f^{*}L,f^{*}s)
\end{align}
where we have used the fact that the DG scheme $S \times_{Y} X$ is the derived zero locus of $f^{*}s$ on $S$. Note that $\MF^{\infty}(S,f^{*}L,f^{*}s)$ still carries an action of $\cR\modules$. Thus we see that the assignment
\begin{equation}
  (S\to Y) \mapsto \MF^{\infty}(S,f^{*}L,f^{*}s)
\end{equation}
is a sheaf of categories on $Y$ that takes values in $\cR$-linear categories. Because $Y$ is $1$-affine (we have assumed it is a smooth variety), the global sections of this sheaf are $\MF^{\infty}(Y,L,s)$, and we have the following theorem.
\begin{theorem}
  \label{thm:descent-mf}
  Let $(Y,L,s)$ be a twisted Landau-Ginzburg model. Then $\MF^{\infty}(Y,L,s)$ satisfies \'{e}tale descent over $Y$: if $c : \coprod_{i} U_{i} \to Y$ is an \'{e}tale covering, then there is an equivalence of $\cR$-linear categories
  \begin{equation}
    \MF^{\infty}(Y,L,s) \to \Tot \left(\xymatrix{ \prod_{i}\MF^{\infty}(U_{i},c^{*}L,c^{*}s) \ar@<0.5ex>[r] \ar@<-0.5ex>[r] & \prod_{i,j}\MF^{\infty}(U_{i}\times_{Y} U_{j},c^{*}L,c^{*}s) \ar@<1ex>[r] \ar[r] \ar@<-1ex>[r] & \cdots} \right) 
  \end{equation}
  (All of the categories and maps in this diagram are $\cR$-linear.)
\end{theorem}

\begin{remark}
  The preceding theorem expresses an ``$L$-twisted $2$-periodic descent property for $\MF$,'' because everything is $\cR$-linear. If we had not introduced $\cR$, we would need to think about ``gluing autoequivalences and natural transformations'' over a limit diagram, which is a fundamentally more difficult problem. The algebra $\cR$ neatly expresses the idea of an autoequivalence with a natural transformation that is local on $Y$.
\end{remark}

\section{Normal crossings surfaces with graph-like singular locus}
\label{sec:graph-like}
Let $X$ be a singular surface with normal crossings. Recall that this means that the singularities of $X$ are formally locally isomorphic to the intersection of several coordinate hyperplanes in $\bA^{3}$. We also assume that $X$ has \emph{graph-like} singular locus. This means that the singular locus $\Sing(X)$ is isomorphic, as a variety, to a union of several copies of $\bP^{1}$ and $\bA^{1}$, which must meet at triple points by the normal crossings condition, and furthermore we require that each $\bA^{1}$ component contains exactly one triple point, and each $\bP^{1}$ component contains exactly two triple points.

\subsection{The singular support scheme} Now we propose to study $\gsing(X)$ when $X$ is a normal crossings surface with graph-like singular locus. Because a normal crossings singularity is abstractly a hypersurface singularity, the rank of $H^{-1}(T^{*}_{x}(X))$ is at most $1$ for any point $x$ in $X$, and it vanishes precisely over the regular locus of $X$. Thus $\gsing(X)$ is the union of $X$ and a rank $1$ vector bundle over $\Sing(X)$ whose fiber at $x$ is $H^{-1}(T^{*}_{x}(X))$.

Let $C \subset \Sing(X)$ be an irreducible component of the singular locus. Then the restriction of $\gsing(X)$ to $C$ is the total space of a line bundle that we denote by $\gsing(X)|_{C}$. We then consider the degree of this line bundle $\deg(\gsing(X)|_{C})$, which is a discrete invariant of the singularities of $X$ along the curve $C$.

The curve $C$ is contained in two irreducible components of $X$, call them $X_{1}$ and $X_{2}$, which are smooth surfaces. Then we may consider the degree of the normal bundle of $C$ in $X_{i}$, or what is the same, the self intersection number $(C)^{2}_{X_{i}}$ of $C$ in $X_{i}$.

In our applications, we will usually have $C \simeq \bP^{1}$, but the following proposition is true more generally.
\begin{proposition}
  \label{prop:triple-point}
  Let $X$ be a normal crossings surface, and let $X_{1}$ and $X_{2}$ be two irreducible components of $X$ meeting in a smooth proper curve $C$. Then
  \begin{equation}
    (C)^{2}_{X_{1}} + (C)^{2}_{X_{2}}  + \deg(\gsing(X)|_{C}) + \deg(\omega_{X}|_{C}) = \deg(\omega_{C}).
  \end{equation}
  where $\omega_{X}$ is the dualizing complex of $X$ and $\omega_{C}$ is the canonical bundle of $C$.
\end{proposition}
\begin{proof}
  This is a version of the adjunction formula. From the inclusions $C \subset X_{i}$, $i = 1,2$, we have natural exact sequences of cotangent bundles
  \begin{equation}
    0 \to T^{*}_{C}X_{i}|_{C} \to T^{*}X_{i}|_{C} \to T^{*}C \to 0, \quad i = 1, 2.
  \end{equation}
  These sequences demonstrate that
  \begin{equation}
    (C)^{2}_{X_{i}} + \deg(\omega_{X_{i}}|_{C}) = \deg(\omega_{C}), \quad i = 1,2,
  \end{equation}
  which is the usual adjunction formula.
  The difference of the two maps $T^{*}X_{i} \to T^{*}C$, $i=1,2$, gives us an exact sequence
  \begin{equation}
    0 \to E \to T^{*}X_{1}|_{C} \oplus T^{*}X_{2}|_{C} \to T^{*}C \to 0,
  \end{equation}
  whence
  \begin{equation}
    \deg(\omega_{X_{1}}|_{C}) + \deg(\omega_{X_{2}}|_{C}) = \deg(E) + \deg(\omega_{C}).
  \end{equation}
  Combining these relations yields
  \begin{equation}
    (C)^{2}_{X_{1}} + (C)^{2}_{X_{2}} + \deg(E) = \deg(\omega_{C}),
  \end{equation}
  and it remains to determine $\deg(E)$.

  We claim that there is a natural map of complexes from $E$ to the restriction of the cotangent complex $T^{*}(X)|_{C}$ whose cone is canonically identified with $\gsing(X)|_{C}$. Locally near a point on $C$, $X$ may be presented as the zero locus of a section $s$ of a line bundle $L$ on a smooth 3-fold $Y$. Then we get an identification $E \cong T^{*}Y|_{C}$, and the cotangent complex of $X$ is given by \eqref{eq:cotangent-complex}. When restricted to $C$, we get
  \begin{equation}
    T^{*}(X)|_{C} = L^{\vee}|_{C}[1] \oplus T^{*}Y|_{C} \cong L^{\vee}|_{C}[1] \oplus E,
  \end{equation}
  since the differential vanishes over $C$ ($C$ is contained in the critical locus of $s$). The presentation also gives us an identification $L_{x}^{\vee} \cong H^{-1}(T_{x}^{*}(X)) = \gsing(X)_{x}$ over a point $x$ of $C$. Thus we deduce
  \begin{equation}
    \deg(E) = \deg(\omega_{X}|_{C}) + \deg(\gsing(X)|_{C}).
  \end{equation}
\end{proof}

\begin{remark}
  The ``triple point formula'' appearing in the definition of a type III degeneration of K3 surfaces is a special case of Proposition \ref{prop:triple-point}. In that case both terms $\deg(\omega_{X}|_{C})$ and $\deg(\gsing(X)|_{C})$ vanish because both $X$ and the total space of the degeneration are Calabi-Yau.
\end{remark}

The preceding proposition can be usefully combined with the following facts.
\begin{proposition}
  Let $X,C,X_{1},X_{2}$ be as in Proposition \ref{prop:triple-point}. If $X$ is isomorphic to the fiber of a morphism $f : Y \to \bA^{1}$ from a smooth $3$-fold $Y$, then $\deg(\gsing(X)|_{C}) = 0$. If $X$ has graph like singular locus, then $\deg(\omega_{X}|_{C}) = 0$.
\end{proposition}

\begin{proof}
  The first assertion follows from the fact that the cotangent complex of $X$ has the form $\cO_{X} \to T^{*}Y|_{X}$, and so $\gsing(X)|_{C}\cong \cO_{C}$.

  For the second assertion, the hypotheses imply $C \cong \bP^{1}$. Let $n_{i} = (C)^{2}_{X_{i}}$, $i = 1,2$. Because the question is local on $C$ we may reduce to the case where $Y = \Tot(\cO_{\bP^{1}}(n_{1}) \oplus \cO_{\bP^{1}}(n_{2}))$ and $X \subset Y$ is the union of the two subbundles and the fibers over $0$ and $\infty$. In this case the assertion follows from the fact that $X$ is the toric boundary anticanonical divisor of $Y$.
\end{proof}

\begin{corollary}
  \label{cor:graph-like-triple-point}
  Let $X$ be a normal crossings surface with graph-like singular locus. Then for any component $C \subset \Sing(X)$ that is isomorphic to $\bP^{1}$, we have
  \begin{equation}
    (C)^{2}_{X_{1}} + (C)^{2}_{X_{2}} + \deg(\gsing(X)|_{C}) = -2,
  \end{equation}
  where $X_{1}$ and $X_{2}$ are the irreducible components of $X$ that contain $C$.
\end{corollary}

The quantity $\deg(\gsing(X)|_{C})$ is an obstruction, local at $C$, for $X$ to arise as the fiber of a morphism $f : Y \to \bA^{1}$ from a smooth $3$-fold $Y$. On the other hand, when $X$ is the zero locus of section $s$ of a line bundle $L$ on $Y$, we have $\deg(\gsing(X)|_{C}) = \deg(L^{\vee}|_{C})$. Corollary \ref{cor:graph-like-triple-point} shows that the degree of $L$ on $C$ is determined by the topology of $X$ itself.

\begin{example}
  The local model for the situation described in Corollary \ref{cor:graph-like-triple-point} is as follows. Choose integers $a$ and $b$, and let $Y = \Tot(\cO_{\bP^{1}}(a) \oplus \cO_{\bP^{1}}(b))$, with projection $\pi: Y \to \bP^{1}$. Set $n = a+b + 2$, and let $L = \pi^{*}\cO_{\bP^{1}}(n)$. Let $s_{a}$ and $s_{b}$ be the tautological sections of $\pi^{*}(\cO_{\bP^{1}}(a))$ and $\pi^{*}(\cO_{\bP^{1}}(b))$ respectively, and let $\sigma$ be a section of $\cO_{\bP_{1}}(2)$ that vanishes at two points, say $0$ and $\infty$. Then $s = s_{a}\cdot s_{b} \cdot(\pi^{*}\sigma)$ is a section of $L$ over $Y$ that vanishes on the surface
  \begin{equation}
    X = \pi^{-1}(0) \cup \pi^{-1}(\infty) \cup \Tot(\cO_{\bP^{1}}(a)) \cup \Tot(\cO_{\bP^{1}}(b)) \subset Y.
  \end{equation}
  If we let $C$ be the zero section, and let $X_{1} = \Tot(\cO_{\bP^{1}}(a))$ and $X_{2} = \Tot(\cO_{\bP^{1}}(b))$, then $(C)^{2}_{X_{1}} = a$, $(C)^{2}_{X_{2}} = b$, and $\deg(\gsing(X)|_{C}) = \deg(L^{\vee}|_{C}) = -n = -a-b-2$.
\end{example}

\begin{remark}
  In the setting of a normal crossings surface with graph-like singular locus that is the zero locus of a section of a line bundle $L$, the triviality of $L$ may be interpreted as a Calabi-Yau condition, due to the following observation.
  
  Suppose that $Y$ is a variety, and $X$ is a divisor in $Y$. Assume that the dualizing complexes of $X$ and $Y$ are invertible sheaves. Consider the following three statements:
\begin{enumerate}
\item \label{cond1} $X$ is Calabi-Yau,
\item \label{cond2} a neighborhood of $X$ in $Y$ is Calabi-Yau.
\item \label{cond3} The normal bundle of $X$ in $Y$ is trivial.
\end{enumerate}
Any two of the statements \eqref{cond1}, \eqref{cond2}, \eqref{cond3} implies the third, due to the adjunction formula.

In our setting, the graph-like condition implies \eqref{cond1}, so in the presence of this condition, \eqref{cond2} and \eqref{cond3} become equivalent. Since the normal bundle is $L|_X$, triviality of this bundle means precisely that a neighborhood of $X$ in $Y$ is Calabi-Yau.
\end{remark}

\subsection{Descent diagram: general shape}

We now assume that $X$ is a normal crossings surface with graph-like singular locus. As the name suggests, we may associate a trivalent graph $G(X)$ to $X$ that encodes the combinatorics of the singular locus of $X$. The graph $G(X)$ has a vertex for each triple point of $\Sing(X)$ and an edge for each irreducible component: $\bP^{1}$ components give compact edges joining two vertices, and $\bA^{1}$ components give noncompact edges touching a single vertex.

The graph $G(X)$ has an open cover by the open stars of its vertices:
\begin{equation}
  C_{G(X)} : \coprod_{v \in \vertices(G(X))} \Star(v) \to G(X)
\end{equation}
For this cover, the overlaps are the interiors of the compact edges. This cover of $G(X)$ corresponds in an obvious way to a Zariski cover of $\Sing(X)$, where the schemes in the cover are isomorphic to three copies of $\bA^{1}$ meeting at triple point, and the overlaps are isomorphic to disjoint unions of copies of $\bG_{m}$. We denote this cover
\begin{equation}
  C_{\Sing(X)}: \coprod_{v \in \vertices(G(X))} \Sing(\{xyz=0\}) \to \Sing(X)
\end{equation}
and call it the \emph{vertex cover} of $\Sing(X)$. Here $\{xyz=0\}$ is the union of the coordinate hyperplanes in $\bA^{3}$.

On the other hand, given a Zariski cover of $X$, it induces a Zariski cover of $\Sing(X)$. We say that a cover of $X$ is \emph{a vertex cover} if the induced cover of $\Sing(X)$ is isomorphic to the vertex cover defined in the previous paragraph. A vertex cover of $X$ has the form
\begin{equation}
  C : \coprod_{v \in \vertices(G(X))} X_{v} \to X
\end{equation}
where $\Sing(X_{v}) \cong \Sing(\{xyz=0\})$, and all overlaps have singular locus isomorphic to disjoint unions of copies of $\bG_{m}$.

We may also consider the case where $X$ arises as the zero locus of a section $s$ of a line bundle $L$ on a smooth $3$-fold $Y$. In this case, a \emph{vertex cover} of $Y$ is a cover
\begin{equation}
  C : \coprod_{v \in \vertices(G(X))} Y_{v} \to Y
\end{equation}
with the property that $X_{v} = s^{-1}(0) \cap Y_{v}$ is a vertex cover of $X$ (that is, the pullback under the inclusion $X \to Y$ is a vertex cover).

Vertex covers always exist: for instance, we may take $Y_{v}$ (respectively $X_{v}$) to be $Y$ (respectively $X$) minus the closed subset of $\Sing(X)$ consisting of all irreducible components of $\Sing(X)$ that do not touch $v$.

We now show that the singularity categories of the local pieces $X_{v}$ are all modeled on $\Dsing(\{xyz=0\})$. There are two versions of the statement, depending on whether we keep track of the (twisted) $2$-periodic structure on the category.

\begin{proposition}
  \label{prop:local-model-vertex}
  Let $C : \coprod_{v \in \vertices(G(X))} X_{v} \to X$ be a vertex cover of $X$. For each $v \in \vertices(G(X))$, there exists a $k$-linear equivalence of categories $\Dsing(X_{v}) \simeq \Dsing(\{xyz=0\})$.

  Suppose additionally that $X$ arises as the zero locus of a section $s$ of a line bundle $L$ on a smooth $3$-fold $Y$, and let $C : \coprod_{v \in \vertices(G(X))} Y_{v} \to Y$ be a vertex cover of $Y$. Then $\Dsing(X) \cong \MF(Y,L,s)$ and $\Dsing(X_{v}) \cong \MF(Y_{v},L|_{Y_{v}},s|_{Y_{v}})$ carry $\cR$-linear structures, and the latter carries an $\cR|_{Y_{v}}$-linear structure.

  By shrinking the cover $\{Y_{v}\}$ to another vertex cover if necessary, the line bundle $L|_{Y_{v}}$ may be trivialized. Any choice of such trivialization induces an isomorphism of algebras $\cR|_{Y_{v}} \simeq \cO_{Y_{v}}[u,u^{-1}]$, $\deg(u) = 2$. Therefore the $\cR|_{Y_{v}}$-linear structure on $\Dsing(X_{v})$ may be regarded as a $2$-periodic structure on this category.
\end{proposition}

\begin{proof}
  The first assertion follows from the fact that (up to idempotent completion, which we always assume) $\Dsing$ only depends on a formal neighborhood of the singular locus, together with the fact that the formal neighborhoods of $\Sing(X_{v})$ in $X_{v}$ and $\Sing(\{xyz=0\})$ in $\bA^{3}$ are isomorphic \cite[Section 5]{PS21}.

  The assertion about $\cR$-linearity is just a restatement of the theory of Section \ref{sec:twisted-lg}.

  For the last assertion, recall that $\cR$ is defined to be $\Sym_{\cO_{Y}}(L[-2])$ with $L[-2]$ inverted. Thus if $L|_{Y_{v}} \simeq \cO_{Y_{v}}$ is trivialized, we have $\Sym_{\cO_{Y_{v}}}(\cO_{Y_{v}}[-2]) \simeq \cO_{Y_{v}}[u]$, $\deg(u) =2$, and so $\cR|_{Y_{v}} \cong \cO_{Y_{v}}[u,u^{-1}]$. Actions of $\cO_{Y_{v}}[u,u^{-1}]\modules$ on a category are the same as $2$-periodic structures on that category, again by the theory of Section \ref{sec:twisted-lg}.
\end{proof}

Next we consider what happens on the overlaps $X_{v_{1}} \cap X_{v_{2}}$. Each component of the singular locus is then a union of copies of $\bG_{m}$. The local model for the singularity category is $\MF(\bA^{3}\setminus\{x=0\}, xyz)$, which is non-canonically equivalent to the $2$-periodic folding $\Perf^{(2)}(\bG_{m})$ by Kn\"{o}rrer periodicity. As we shall see, this category admits many $2$-periodic structures depending on both discrete and continuous parameters. In particular, the $2$-periodic structure on $\Perf^{(2)}(\bG_{m})$ induced by its presentation as the $2$-periodic folding of $\Perf(\bG_{m})$ is not necessarily distinguished.
\begin{proposition}
  \label{prop:local-model-edge}
  Let $v_{1}, v_{2} \in \vertices(G(X))$ be such that $\Sing(X_{v_{1}}\cap X_{v_{2}}) \simeq \coprod_{e} \bG_{m}$, where the disjoint union runs over edges $e$ joining $v_{1}$ to $v_{2}$. Then there is a $k$-linear equivalence of categories $\Dsing(X_{v_{1}}\cap X_{v_{2}}) \simeq \coprod_{e} \Perf^{(2)}(\bG_{m})$, where the functor is given by some choice of Kn\"{o}rrer periodicity.

  Suppose further that $X$ arises from a triple $(Y,L,s)$, and $\{X_{v}\}$ is the pullback of a vertex cover $\{Y_{v}\}$ of $Y$ such that $L|_{Y_{v}}$ is trivial. A choice of trivialization of $L|_{Y_{v_{1}}}$ induces a $2$-periodic structure $u_{1}$ on $\Dsing(X_{v_{1}}\cap X_{v_{2}})$. A choice of trivialization of $L|_{Y_{v_{2}}}$ induces another $2$-periodic structure $u_{2}$ on $\Dsing(X_{v_{1}}\cap X_{v_{2}})$. These $2$-periodic structures $u_{1}$ and $u_{2}$ are, in general, different.

  If we equip $\Perf^{(2)}(\bG_{m})$ with the $2$-periodic structure that is manifest from its presentation as the $2$-periodization of $\Perf(\bG_{m})$, the equivalence $\Dsing(X_{v_{1}}\cap X_{v_{2}}) \simeq \coprod_{e}\Perf^{(2)}(\bG_{m})$ may be taken to be $2$-periodic with respect to either $u_{1}$ or $u_{2}$, but, in general, not both simultaneously.
\end{proposition}
\begin{proof}
  The first assertion is true because, the formal neighborhood of $\Sing(X_{v_{1}}\cap X_{v_{2}})$ in $X_{v_{1}}\cap X_{v_{2}}$ is isomorphic to the formal neighborhood of $\bG_{m} \times \{0\}$ in $\bG_{m} \times \{xy=0\}$.

  The second set of assertions holds because, when $L|_{Y_{v_{1}}}$ is trivialized, $\Dsing(X_{v_{1}}\cap X_{v_{2}})$ may be presented as $\MF(Y_{v_{1}},f_{v_{1}})$, where $f_{v_{1}}$ is a representative of $s|_{Y_{v_{1}}}$ with respect to the local trivialization, and the Kn\"{o}rrer periodicity equivalence may be taken to be $2$-periodic.
\end{proof}

The combination of these results determines the general shape of the descent diagram for $\Dsing(X)$ and $\MF(Y,L,s)$. Theorems \ref{thm:descent-dsing} and \ref{thm:descent-mf} together with Propositions \ref{prop:local-model-vertex} and \ref{prop:local-model-edge} imply that we have equivalences of categories as follows. Recall the notation $J(G(X))$ for the category whose objects are the vertices and edges of $G(X)$, and with an arrow from $v$ to $e$ whenever $(v,e)$ is a flag in $G(X)$.
\begin{proposition}
  Let $X$ and $(Y,L,s)$ be as above. Then there are $J(G(X))$-shaped diagrams of DG categories $F_{X}$ and $F_{(Y,L,s)}$ such that there are equivalences
  \begin{equation}
    \label{eq:dsing-general-shape}
    \Dsing^{\infty}(X) \to \holim \left(F_{X} : J(G(X)) \to \DGCat \right),
  \end{equation}
  \begin{equation}
    \label{eq:mf-general-shape} 
    \MF^{\infty}(Y,L,s) \to \holim \left(F_{(Y,L,s)}: J(G(X)) \to \DGCat\right),
  \end{equation}
  and $F_{X}(v) = \Dsing^{\infty}(\{xyz=0\})$, $F_{(Y,L,s)}(v) = \MF^{\infty}(\bA^{3},xyz)$, $F_{X}(e) = F_{(Y,L,s)}(e) = \Qcoh^{(2)}(\bG_{m})$, 
  and the functors in these diagrams are as described in Proposition \ref{prop:local-model-edge}.
\end{proposition}
\begin{proof}
  Theorems \ref{thm:descent-dsing} and \ref{thm:descent-mf} most immediately give homotopy limit presentations that are indexed by the \v{C}ech nerve of the vertex covering $\{X_{v}\}_{v \in \vertices(G(X))}$. However, in this covering, all triple and higher intersections are smooth varieties, so that $\Dsing$ and $\MF$ are null categories. Since the null category is a terminal object in the category of DG categories, maps from any category into it form a contractible space, and therefore the triple and higher overlaps do not contribute in an essential way to the homotopy limit. (This is analogous to the observation that a fiber product over a final object is the same as an ordinary product.) 
\end{proof}

So far we have not pinned down the functors other than to say that they are given by \emph{some} Kn\"{o}rrer periodicity equivalence for \emph{some} choice of local presentation. The next step in our analysis is to classify all of the possible functors that could arise in this diagram.

\subsection{Autoequivalences and $2$-periodic structures for the local models}
We now propose to classify the autoequivalences and $2$-periodic structures on the local models $\MF(\bA^{3},xyz)$ and $\Perf^{(2)}(\bG_{m})$. The group of triangulated autoequivalences of a triangulated DG category $\cC$ will be denoted $\Auteq(\cC)$; the group of isomorphism classes of autoequivalences is $\pi_{0}\Auteq(\cC)$. The set of (isomorphism classes of) $2$-periodic structures on $\cC$, that is, natural isomorphisms $\id \to \id[2]$, will be denoted $\twoper(\cC)$.

The groups $\Auteq(\cC)$ and $\pi_{0}\Auteq(\cC)$ act on $\twoper(\cC)$ by conjugation as follows. Let $\Phi \in \Auteq(\cC)$. Part of the data of a triangulated functor is a prescribed isomorphism $\Phi \circ \id[n] \cong \id[n] \circ \Phi$, from which we obtain a preferred isomorphism $\Phi \circ \id[n] \circ \Phi^{-1} \cong \id[n]$. A natural transformation $t : \id \to \id[2]$ then gives $\Phi t\Phi^{-1}:  \Phi \circ \id \circ \Phi^{-1}\to \Phi \circ \id[2] \circ \Phi^{-1}$, which becomes a natural transformation $\id \to \id[2]$ using the preferred isomorphisms.

Another fact is that $\twoper(\cC)$ is a torsor over $\HH^{0}(\cC)^{\times}$, which we think of as the set of natural isomorphisms $\id\to\id$, acting by precomposition.

In this section we shall write $\cV = \MF(\bA^{3},xyz) \cong \Dsing(\{xyz=0\})$ and $\cE = \Perf^{(2)}(\bG_{m})$ for the ``vertex'' and ``edge'' categories. This is for brevity but also clarity: these presentations equip the categories with certain $2$-periodic structures, but we want to consider all possible $2$-periodic structures.

\begin{proposition}
  \label{prop:autoequiv-of-vertex}
  Let $\cV = \MF(\bA^{3},xyz)$.

  We have $\pi_{0}\Auteq(\cV) \cong ((k^{\times})^{3} \rtimes S_{3})\times \Z/2\Z$. The action of $(k^{\times})^{3}$ rescales the coordinates $(x,y,z)$, the action of $S_{3}$ permutes the coordinates, and $\Z/2\Z$ acts by the shift.
  
  We have $\HH^{0}(\cV)^{\times} \cong k^{\times}$, so that $\twoper(\cV)$ is a $k^{\times}$-torsor.

  An element $(\lambda_{1},\lambda_{2},\lambda_{3}) \in (k^{\times})^{3}$ acts on $\twoper(\cV)$ by scaling each element by $(\lambda_{1}\lambda_{2}\lambda_{3})^{-1}$. The subgroups $S_{3}$ and $\Z/2\Z$ act trivially on $\twoper(\cV)$. Therefore the stabilizer of any particular $2$-periodic structure is $(\{\lambda_{1}\lambda_{2}\lambda_{3} = 1\} \rtimes S_{3}) \times \Z/2\Z$.
\end{proposition}

\begin{proof}
  By homological mirror symmetry for punctured spheres \cite{abouzaid2013homological}, we know that $\cV \simeq \Fuk(P)$ where $P$ is the pair of pants. The category $\Fuk(P)$ is generated by three embedded arcs joining each pair of punctures. Any autoequivalence must permute this collection of objects up to shift: by \cite{haiden2017flat}, an autoequivalence must send such an arc to an immersed arc with the same endomorphism ring; this is only possible if the arc is in fact embedded, in which case it is isotopic to one of the generators we started with. This explains the $S_{3}$ factor in $\pi_{0}\Auteq(\cV)$. An autoequivalence that fixes these objects acts on the morphism spaces; the category is generated by six morphisms forming two oppositely oriented exact triangles. We can rescale these generating morphisms subject to the condition that the product of the scalars over each exact triangle is one. This gives four (six minus two) apparent degrees of freedom, but there are natural isomorphisms given by multiples of the identity morphisms of the three objects. After taking $\pi_{0}$, there are only three degrees of freedom: a single scalar for each puncture, which is the product of the scalars for the two generating morphisms at this puncture. In the B-model realization $\cV = \MF(\bA^{3},xyz)$, these degrees of freedom come from automorphisms of $\bA^{3}$ that rescale the coordinates.

  To determine $HH^{0}(\cV)^{\times}$, we use the fact that $HH^{*}(\cV) \simeq SH^{*}(P)$ is the ($2$-periodic) symplectic cohomology of $P$. This is nonzero in all degrees, and $SH^{0}(P)$ is isomorphic to the coordinate ring of the union of the axes in $\bA^{3}$. Thus the invertible elements come from constants (the image of the natural map $H^{0}(P) \to SH^{0}(P)$), and so $HH^{0}(\cV)^{\times} \simeq H^{0}(P)^{\times} \simeq k^{\times}$.

  The last assertion may be checked using the realization $\cV = \MF(\bA^{3},xyz)$: If $t$ denotes a given $2$-periodic structure, $t$ must scale oppositely to the defining function $xyz$ when we send $(x,y,z) \mapsto (\lambda_{1}x,\lambda_{2}y,\lambda_{3}z)$.
\end{proof}

\begin{proposition}
  \label{prop:autoequiv-of-edge}
  Let $\cE = \Perf^{(2)}(\bG_{m})$, which we may also write as $\Perf(k[x^{\pm 1},u^{\pm 1}])$ where $\deg(x) = 0$ and $\deg(u) = 2$.

  We have $\pi_{0}\Auteq(\cE) \cong \Aut^{\gr}(k[x^{\pm 1},u^{\pm 1}]) \times \Z/2\Z$, where $\Aut^{\gr}$ denotes the group of graded automorphisms of a graded ring, and $\Z/2\Z$ acts by the shift.

  We have $\HH^{0}(\cE)^{\times} \cong k^{\times} \times \Z$, so that $\twoper(\cE)$ is a $(k^{\times}\times \Z)$-torsor. In terms of the presentation we have identifications
  \begin{equation}
    \HH^{0}(\cE)^{\times} = (k[x^{\pm 1},u^{\pm 1}])_{0}^{\times} = \{c x^{n} \mid c \in k^{\times}, n \in \Z\},
  \end{equation}
  \begin{equation}
    \twoper(\cE) = (k[x^{\pm 1},u^{\pm 1}])_{2}^{\times} = \{c x^{n} u \mid c \in k^{\times}, n \in \Z\}.
  \end{equation}

  The action of $\Aut^{\gr}(k[x^{\pm 1},u^{\pm 1}])$ on $\twoper(\cE)$ is the natural action of automorphisms on the set of invertible degree two elements. Therefore the stabilizer of any particular $2$-periodic structure is $\Aut(k[x^{\pm 1}]) \times \Z/2\Z$.
\end{proposition}

\begin{proof}
  This follows from similar considerations as Proposition \ref{prop:autoequiv-of-vertex}, using the fact that $\cE$ is equivalent to the ($2$-periodic) Fukaya category of the cylinder.
\end{proof}

All of the $2$-periodic structures on $\cV$ and $\cE$ may be realized geometrically in terms of categories of matrix factorizations. The basic observation is that if $(Y,f)$ is a Landau-Ginzburg model, then rescaling $f$ by $c$ corresponds to rescaling the $2$-periodic structure by $c^{-1}$. Thus the family of $2$-periodic DG categories
\begin{equation}
  \MF(\bA^{3},cxyz), \quad c \in k^{\times},
\end{equation}
realizes all possible $2$-periodic structures on $\cV \simeq \Dsing(\{xyz=0\} \subset \bA^{3})$. Likewise, the family
\begin{equation}
  \MF(\bA^{3}\setminus\{x=0\},cx^{m}yz), \quad c \in k^{\times}, m \in \Z,
\end{equation}
realizes all possible $2$-periodic structures on $\cE \simeq \Dsing(\{yz=0\}\subset \bA^{3}\setminus \{x=0\})$.

\subsection{Descent diagram: determination of the functors}
\label{dd:dotf}
The next step is to pin down as closely as possible the restriction functors $R: \cV \to \cE$ appearing in the diagrams \eqref{eq:dsing-general-shape} and \eqref{eq:mf-general-shape}. The index category $J(G(X))$ for these diagrams has one object for each vertex and edge of $G(X)$, and one arrow for each flag $(v,e)$ in $G(X)$. It is helpful to consider a slightly different form for the general diagram. Given an undirected graph $G$, we introduce the following category $I(G)$:
\begin{itemize}
\item $I(G)$ has an object for each vertex of $G$, and also for each flag in $G$,
\item given a vertex $v$ and an incident flag $(v,e)$, there is an arrow from the object $v$ to the object $(v,e)$,
\item for each compact edge $e$ belonging to two distinct flags $(v_{1},e)$ and $(v_{2},e)$, there is an \emph{isomorphism} between the objects $(v_{1},e)$ and $(v_{2},e)$.
\end{itemize}
Essentially, $I(G)$ differs from $J(G)$ because we have duplicated the object for each compact edge, but the duplicate objects are isomorphic. There is a functor $I(G) \to J(G)$ that takes each flag-object to the corresponding edge-object; this functor is an equivalence of categories. Therefore (homotopy) limits over $I(G)$ are equivalent to (homotopy) limits over $J(G)$.

In any case, the category $\Dsing(X_{v})$ or $\MF(Y_{v},L|_{Y_{v}},s|_{Y_{v}})$ associated to a vertex is always $k$-linearly equivalent to $\Dsing(\{xyz= 0\}) = \MF(\bA^{3},xyz)$. In the $L$-twisted case, because the autoequivalences of this category act transitively on the set of $2$-periodic structures, we may assume that the equivalence is $2$-periodic once a trivialization of $L|_{Y_{v}}$ has been chosen.

Now consider a vertex $v$ with incident edges $e_{1},e_{2},e_{3}$. The restriction functors are of the form
\begin{equation}
  (R_{1},R_{2},R_{3}) : \MF(\bA^{3},x_{1}x_{2}x_{3}) \to \prod_{i=1}^{3}\Perf^{(2)}(\bG_{m})
\end{equation}
where the coordinates $x_{1}x_{2}x_{3}$ now correspond to the edges, so that the restriction along $e_{i}$ is of the form
\begin{equation}
  R_{i} : \MF(\bA^{3},x_{1}x_{2}x_{3}) \to \MF(\bA^{3}\setminus \{x_{i}=0\},x_{1}x_{2}x_{3}) \to \Perf^{(2)}(\bG_{m})
\end{equation}
where the first functor is restriction and the second is Kn\"{o}rrer periodicity. The Kn\"{o}rrer periodicity functor depends on an additional choice, which is a maximal isotropic subbundle of $(\bA^{3}\setminus\{x_{i}=0\},x_{1}x_{2}x_{3})$ regarded as a quadratic bundle over the punctured $x_{i}$-line. There are two choices leading to Kn\"{o}rrer periodicity functors that differ by a shift. To fix the choice, we choose a cyclic ordering of the edges $e_{1},e_{2},e_{3}$ of the edges at $v$ (say the one coming from the indexing), and we declare that $\cO_{\{x_{i+1}=0\}}$ maps to $\cO_{\bG_{m}}$ under the restriction functor $R_{i}$. Note that the functors $R_{i}$ are also taken to be $2$-periodic with respect to the evident $2$-periodic structures in these local models.

We must emphasize that we have now made several local choices at the vertex $v$: a choice of local presentation of $\Dsing(X_{v})$ or $\MF(Y_{v},L|_{Y_{v}},s|_{Y_{v}})$ as $\MF(\bA^{3},xyz)$, and this involves a choice of trivialization of $L|_{Y_{v}}$. We have also chosen a cyclic ordering of the edges at $v$. However, we have now put most of the diagram over $I(G(X))$ into a normal form. The only remaining data to determine are the autoequivalences of $\cE = \Perf^{(2)}(\bG_{m})$ living at the edges. In principle, for a given $X$, these could be arbitrary elements of
\begin{equation}
  \pi_{0}\Auteq(\cE) \cong \Aut^{\gr}(k[x^{\pm 1},u^{\pm 1}]) \times \Z/2\Z  
\end{equation}
so the question is how does the geometry of $X$ constrain what these elements could be.

The structure of $\Aut^{\gr}(k[x^{\pm 1}, u^{\pm 1}])$ is that of a semidirect product
\begin{equation}
  \Aut^{\gr}(k[x^{\pm 1}, u^{\pm 1}]) \cong (k^{\times})^{2} \rtimes H
\end{equation}
where $(k^{\times})^{2}$ rescales the coordinates $(x,u)$, and $H < \GL(2,\Z)$ is the subgroup
\begin{equation}
  H = \left\{\begin{pmatrix}\epsilon & n \\ 0 & 1\end{pmatrix} \mid \epsilon \in \{-1,1\}, n \in \Z\right\},
\end{equation}
where an element of $H$ acts by $(x,u) \mapsto (x^{\epsilon},x^{n}u)$. We refer to $H$ as the discrete part.

The first constraint says that, over an edge $e$, we are gluing together two copies of $\bA^{1}$ to a $\bP^{1}$, rather than to an $\bA^{1}$ with doubled origin.
\begin{lemma}
  \label{lem:discrete-twists}
  Let $X$ be a normal crossings surface with graph-like singular locus. Let $\Phi$ be one of the autoequivalences on the edges in the $I(G(X))$-shaped diagram computing $\Dsing(X)$ as set up above. Then the discrete part of $\Phi$ is of the form
  \begin{equation}
    \label{eq:transition-matrix}
    \begin{pmatrix} -1 & n \\ 0 & 1\end{pmatrix},
  \end{equation}
  That is, the discrete part lies in the coset of $\tau = \begin{pmatrix} -1 & 0 \\ 0 & 1\end{pmatrix}$ for the normal subgroup $U < H$ consisting of upper-triangular matrices with ones on the diagonal.
\end{lemma}
\begin{proof}
  The assertion is about an edge $e$ joining two vertices $v_{1}$ and $v_{2}$. In the local models for $v_{1}$, the coordinate $x$ in $k[x^{\pm 1},u^{\pm 1}]$ represents a local coordinate on the $\bP^{1}$ contained in $\Sing(X)$ corresponding to $e$, and this local coordinate vanishes at the triple point corresponding to $v_{1}$. Likewise in the local model for $v_{2}$, $x$ represents a local coordinate on the $\bP^{1}$ vanishing at the other triple point. Therefore these two $x$-coordinates must be inverses of each other up to scaling.
\end{proof}

The next constraint has to do with the shift ambiguity in the Kn\"{o}rrer periodicity functor. This is not entirely local but it has to do with the global topology of $X$. Let $C(X)$ be the dual intersection complex of $X$: this is the simplicial complex with a vertex for each irreducible component of $X$, an edge for each curve in $\Sing(X)$, and a triangle for each triple point. The assumption that $X$ has graph-like singular locus implies that $C(X)$ is a topological manifold. $G(X)$ is the dual graph of the $1$-skeleton of $C(X)$. The question is whether this manifold is orientable.
\begin{lemma}
  \label{lem:always-shift}
  Suppose that $C(X)$ is orientable. Then it is possible to arrange that in the $I(G(X))$-shaped diagram computing $\Dsing(X)$, all autoequivalences along the edges involve the shift $[1]$, that is they are of the form $\Psi[1]$ for some $\Psi \in \Aut^{\gr}(k[x^{\pm 1},u^{\pm 1}])$.
\end{lemma}
\begin{proof}
  Choose an orientation on $C(X)$. Since $G(X)$ embeds in $C(X)$, this induces a ribbon structure on $G(X)$, which is to say a cyclic ordering of the edges at each vertex. If we use these cyclic orderings to set up the $I(G(X))$-shaped diagram as above, we get the desired result, as may be checked from the local picture of a single edge joining two trivalent vertices.
\end{proof}

\begin{remark}
  When $C(X)$ is not orientable, we may consider its first Stiefel-Whitney class $w_{1}(C(X)) \in H^{1}(C(X),\Z/2\Z)$. By restriction we obtain a class $w_{1} \in H^{1}(G(X),\Z/2\Z)$ that may be used to twist the construction relative to the orientable case. 
\end{remark}

The last constraint has to do with the number $n$ appearing in the matrices $\begin{pmatrix} -1 & n \\ 0 & 1\end{pmatrix}$ that are the discrete parts of our autoequivalences. It is helpful to note that this matrix is its own inverse; this means that the number $n$ does not depend on the orientation of the edge. The number $n$ is related to the degree $\deg(\gsing(X)|_{C})$.
\begin{lemma}
  \label{lem:twist-equals-degree}
  Let $e$ be a edge of $G(X)$ corresponding to a curve $C \cong \bP^{1}$ contained in $\Sing(X)$. Let $n$ be the integer entry appearing in \eqref{eq:transition-matrix}, the discrete part of the autoequivalence on the edge $e$. Then
  \begin{equation}
    n = -\deg(\gsing(X)|_{C}) = \deg(L|_{C}),
  \end{equation}
  where the second equality holds when $X$ comes from a triple $(Y,L,s)$.
\end{lemma}
\begin{proof}
  Let $v_{1}$ and $v_{2}$ be the two vertices of $G(X)$ corresponding to the triple points on $C$. We have chosen local presentations $\MF(\bA^{3},x_{1}y_{1}z_{1})$ at $v_{1}$ and $\MF(\bA^{3},x_{2}y_{2}z_{2})$ at $v_{2}$, where the curve $C$ corresponds to the $x_{1}$ or $x_{2}$ axis. We also assume that the component $\{y_{1} = 0\}$ corresponds to the component $\{y_{2} = 0\}$ and likewise that $\{z_{1} = 0\}$ corresponds to $\{z_{2} = 0\}$. In particular $x_{2} = x_{1}^{-1}$ (at least up to scaling) along $C$.

  Restricting from $\bA^{3}$ to $\bG_{m}\times \bA^{2}$, we have two different defining functions, namely $x_{1}y_{1}z_{1}$ and $x_{2}y_{2}z_{2}$, and these two functions induce different $2$-periodic structures $u_{1}$ and $u_{2}$ respectively. The point is that, in order for the underlying DG category $\Dsing$ to remain unchanged, the quantities $u_{1}x_{1}y_{1}z_{1}$ and $u_{2}x_{2}y_{2}z_{2}$ must by identified. Thus we find
  \begin{equation}
    \frac{u_{2}}{u_{1}} = \frac{x_{1}y_{1}z_{1}}{x_{2}y_{2}z_{2}} = x_{1}^{2}\frac{y_{1}z_{1}}{y_{2}z_{2}}
  \end{equation}
  The ratio $(y_{1}z_{1})/(y_{2}z_{2})$ is $x_{1}$ raised to some power which measures the degrees of the normal bundles to $C$ in the two surfaces $\{y = 0\}$ and $\{z = 0\}$. In fact
  \begin{equation}
    \frac{u_{2}}{u_{1}} = x_{1}^{2} x_{1}^{(C)^{2}_{y=0}}x_{1}^{(C)^{2}_{z=0}} = x_{1}^{2 + (C)^{2}_{y=0} + (C)^{2}_{z=0}} = x_{1}^{-\deg(\gsing(X)|_{C})},
  \end{equation}
  where we have used Corollary \ref{cor:graph-like-triple-point}. Thus the desired automorphism takes $(x_{1},u_{1}) \mapsto (x_{2},u_{2}) = (x_{1}^{-1}, x_{1}^{-\deg(\gsing(X)|_{C})}u_{1})$, and we find $n = -\deg(\gsing(X)|_{C})$ as claimed.

  An alternative argument is to regard the pairs $(x_{i},u_{i})$ as local coordinates on the variety $\gsing(X)|_{C}$, which is isomorphic to $\Tot(\cO_{\bP_{1}}(-n))$ for some $n$. We then recognize that the matrix \eqref{eq:transition-matrix} represents the change of coordinates between the two natural toric affine charts on $\Tot(\cO_{\bP^{1}}(-n))$.

  In the $(Y,L,s)$ case, yet another argument involves regarding the variables $u_{i}$ as local generators for the sheaf of algebras $\Sym_{\cO_{Y}}(L[-2])$. Thus the variables $u_{i}$ must transform as sections of $L$, leading again to the equality $n = \deg(L|_{C})$.
\end{proof}

Lemmas \ref{lem:discrete-twists}, \ref{lem:always-shift} and \ref{lem:twist-equals-degree} characterize (up to equivalence of diagrams) the discrete data involved in the autoequivalence $\Phi$, but there are still continuous parameters. In terms of the local coordinates $(x,u)$ for $\Perf^{(2)}(\bG_{m}) \cong \Perf^{(2)}(k[x^{\pm 1},u^{\pm 1}])$. The most general form for the autoequivalence $\Phi$ that is compatible with the previous lemmas is
\begin{equation}
  \Phi = \Psi[1], \quad \Psi : (x,u) \mapsto (\alpha x^{-1}, \beta x^{n}u), \quad \alpha,\beta \in k^{\times}.
\end{equation}
For a given $X$, the values of $\alpha$ and $\beta$ are difficult to pin down because they depend on the choices of local coordinates on the various patches in the vertex covering $\coprod_{v} X_{v} \to X$ we started with, and also the local functions $xyz$ that are used in the equivalence $\Dsing(X_{v}) \cong \MF(\bA^{3},xyz)$.

While we shall not go further into pinning down these continuous parameters directly, they are related to the twisted $2$-periodic structure on the category $\Dsing(X)$:

\begin{itemize}
\item First consider the case where $X$ arises from a triple $(Y,L,s)$. In this case we have supposed that the equivalence $\Dsing(X_{v}) \cong \MF(Y_{v},L|_{Y_{v}},s|_{Y_{v}}) \cong \MF(\bA^{3},xyz)$ arise from a choice of trivialization of $L|_{Y_{v}}$. When $v_{1}$ and $v_{2}$ are vertices joined by an edge $e$, we have an autoequivalence of the form $\Phi_{e} = \Psi_{e}[1]$, $\Psi : (x,u) \mapsto (\alpha_{e}x^{-1},\beta_{e}x^{n_{e}}u)$. The $2$-periodic structures on the parts are encoded by the degree two elements $u$ and $\beta_{e}x^{n_{e}}u$; these elements must match up to the change of trivialization of $L$, since they all come from a single $L$-twisted $2$-periodic structure on $\MF(Y,L,s)$. We can now deduce the following: if $C_{e} \simeq \bP^{1}$ denotes the component of $\Sing(X)$ corresponding to $e$, then $C_{e}$ is covered by two copies of $\bA^{1}$, with transition function $x \mapsto \alpha_{e}x^{-1}$, and with respect to this atlas, the clutching function of $L|_{C_{e}}$ is $\beta_{e}x^{n_{e}}$.
\item Now suppose that $X$ is a normal crossings surface with graph-like singular locus, but we do not assume that a triple $(Y,L,s)$ is given. The autoequivalences have the same form $(x,u) \mapsto (\alpha_{e}x^{-1},\beta_{e}x^{n_{e}}u)$ as before. Then we may \emph{construct} a line bundle $L$ on $\Sing(X)$ such that $L|_{C_{e}}$ has clutching function $\beta_{e}x^{n_{e}}$ with respect to the evident atlas on $C_{e}$. This line bundle $L$ may be extended to a formal neighborhood of $\Sing(X)$ in $X$: a vertex covering of $X$ induces a covering of the formal neighborhood of $\Sing(X)$ where all intersections are modeled on the formal neighborhood of $\{y=z=0\} \subset \bG_{m} \times \bA^{2}$, and the clutching function $\beta_{e}x^{n_{e}}$ may be extended over this formal scheme. Since $\Dsing(X)$ is local to the formal neighborhood of $\Sing(X)$ in $X$, tensoring with $L$ gives a well-defined autoequivalence $\Lambda$ of $\Dsing(X)$. (Note that $L$ may not extend to all of $X$, since the Picard group of an irreducible component of $X$ is typically too small.)
\end{itemize}

The following theorem combines the preceding propositions and lemmas into a summary of the results of this section.

\begin{theorem}
  Let $X$ be a normal crossings surface with graph-like singular locus described by the graph $G(X)$, such that the dual intersection complex is orientable.
  \begin{enumerate}
  \item For each edge $e$ of $G(X)$ corresponding to a curve $C \cong \bP^{1}$ contained in $\Sing(X)$, set $n_{e} = -\deg(\gsing(X)|_{C_{e}})$. Then there is an autoequivalence $\Lambda$ of $\Dsing(X)$ and a natural isomorphism $t: \id \to \Lambda^{-1}[2]$, such that $\Lambda$ has the effect of tensoring with a line bundle of degree $n_{e}$ over $C_{e}$.
  
  \item Let $I(G(X))$ be the indexing category introduced at the beginning of Section \ref{dd:dotf}. Then $\Dsing(X)$ is the homotopy limit of an $I(G(X))$-shaped diagram where the gluing along the edge $e$ is twisted by an autoequivalence of $\cE = \Perf^{(2)}(\bG_{m})$ whose discrete part has the form
  \begin{equation}
    \phi_{n_{e}} = \begin{pmatrix}-1 & n_{e}\\ 0 &1\end{pmatrix}.
  \end{equation}

  \item Now assume that $X$ arises as the zero locus of a section $s$ of a line bundle $L$ on a smooth three-fold $Y$. Then we have $n_{e} = \deg(L|_{C_{e}})$, the autoequivalence $\Lambda$ is given by tensoring with $L$. Further  $\MF(Y,L,s) \simeq \Dsing(X)$, and therefore $\MF(Y,L,s)$ is again given as the homotopy limit of an $I(G(X))$-shaped diagram.  
  \end{enumerate}
\end{theorem}

\section{Relative singularity categories as localizations of $L$-twisted $2$-periodic categories}
\label{sec:relative-sing}

In this section we assume that $X$ is a normal crossings surface with graph-like singular locus that arises as the zero locus of a section $s$ of a line bundle $L$ on a three-fold $Y$. Up to this point, we have assumed that $Y$ is itself smooth, but we have allowed the line bundle $L$ to be non trivial. Instead, we could consider the situation where $L$ is trivial, so that $s$ is just a function, but the ambient variety, call it $Y'$, is singular. The theory of matrix factorizations and singularity categories where the total space $Y'$ is singular presents extra subtlety, since there are now two different \emph{ambient} categories $\Coh(Y')$ and $\Perf(Y')$ that we can use as the starting point, meaning that we may consider ``coherent'' or ``perfect'' matrix factorizations.

We will be primarily concerned with the coherent version, as studied by Efimov-Positselski \cite{efimov-positselski}. Let $i : X \to Y'$ be the inclusion of a hypersurface $X$ in a possibly singular variety $Y'$. This is a local complete intersection morphism and has finite tor dimension. Therefore $i^{*}$ sends $\Coh(Y')$ to $\Coh(X)$ and $\Perf(Y')$ to $\Perf(X)$, and there is an induced functor
\begin{equation}
  \label{eq:dsing-pullback}
  i^{*} : \Dsing(Y') \to \Dsing(X)
\end{equation}
The \emph{(coherent) relative singularity category} $\Dsing(X,Y')$ is defined to be the homotopy cofiber of the functor \eqref{eq:dsing-pullback}, and we have
\begin{equation}
  \Dsing(X,Y') \cong \Coh(X)/\langle\Perf(X),i^{*}\Coh(Y')\rangle.
\end{equation}
By the results of \cite{efimov-positselski}, if $X = f^{-1}(0)$ for some function $f: Y' \to \bA^{1}$, this category is equivalent to the category of \emph{coherent matrix factorizations} of $f$: $\Dsing(X,Y') \cong \MF(Y',f)$, and this category is $2$-periodic.

\begin{remark}
  We mention in passing that the \emph{perfect relative singularity category} $\Dsing(X,Y')_{\Perf}$ is defined to be the homotopy \emph{fiber} of the pushforward functor $i_{*} : \Dsing(X) \to \Dsing(Y')$. If we define $\Coh(X)_{\Perf(Y')}$ to be the subcategory of $\Coh(X)$ consisting of objects $X$ such that $i_{*}(X)$ is perfect on $Y'$, then we have
\begin{equation}
  \Dsing(X,Y')_{\Perf} \cong \Coh(X)_{\Perf(Y')}/\Perf(X).
\end{equation}
The perfect relative singularity category is the version studied in \cite{BRTV, pippi}. In this paper the term ``relative singularity category'' always refers to the first, coherent version.
\end{remark}

%Now suppose that $Y$ is singular, with $f : Y \to \bA^{1}$, and $X = f^{-1}(0)$. Then we may consider $\Dsing(X,Y) \cong \MF(Y,f)$.  Let $\widetilde{Y}$ be a resolution of singularities of $Y$, let $\widetilde{X} \subset \widetilde{Y}$ be the strict transform of $X$ and suppose that there is a line bundle $L$ on $\widetilde{Y}$ with section $s$ such that $\widetilde{X} \cong s^{-1}(0)$ (scheme theoretically). Then $\MF(\widetilde{Y},L,s)$ realizes $\Dsing(X)$, and we may ask how $\Dsing(X,Y)$ can be recovered from $\MF(\widetilde{Y},L,s)$.

\subsection{From absolute to relative: localizations and pencils of divisors}
It turns out that categories of coherent matrix factorizations arise naturally in mirror symmetry. In order to explain the mirror symmetry picture however we need to develop some preliminary results on categories of matrix factorizations, which are of independent interest. We are led to consider various flavors of categories of matrix factorizations: most of this section and the next will be devoted to explaining how (under suitable assumptions)  these different categories compare to each other. A first application to mirror symmetry is given by Corollary \ref{cor:relative-sing-fuk}, and the wider (but partially conjectural) mirror symmetry context is presented in Section \ref{sec:symplectic-conjectures}.

Let us start by explaining  briefly the setting we will be  interested in. Let $Y$ be a smooth variety, equipped with a line bundle $L$. Let $s, s'$ be two linearly independent sections of $L$, and consider the pencil of divisors 
\begin{equation}
  \{ us + vs'=0\} \quad [u:v] \in \mathbb{P}^1
\end{equation}
Let $X=s^{-1}(0)$. For our applications, $Y$ will be a threefold, $X$ will be a normal crossings divisor with graph like singular locus, and $s'$ will be a fixed auxiliary section. We can consider the following categories: 
\begin{enumerate}
\item The first is the category $\MF(Y,L,s)$ of matrix factorizations of $s$. This is $L$-twisted $2$-periodic, meaning that tensoring with $L$ is isomorphic to shift by $2$. This category does not involve $s'$.
\item Let $W =\{us+vs'=0 \} \subset Y \times \bP^{1}$ be the total space of the pencil of divisors. There is a morphism $\pi_{2} : W \to \bP^{1}$ whose fibers are the members of the pencil. We remove from $W$ the fiber at $\infty$, that is, the divisor $(s')^{-1}(0)$. This is a variety $Y'$ which is birational to $Y$, and such that $X$ is the zero level set of a function $f: Y' \to \bA^{1}$. Note that in general $Y'$ is singular. We can consider the category 
\begin{equation}
  \Dsing(X,Y') \cong \MF(Y',f)
\end{equation}
which is genuinely 2-periodic.
\item Next, consider a resolution $\pi : \widetilde{Y}' \to Y'$. The function $f : Y' \to \bA^{1}$ pulls back to a function $\tilde{f} : \widetilde{Y}' \to \bA^{1}$. We let $\widetilde{X}$ be the strict transform of $X = f^{-1}(0)$ under the resolution. Then we have $\widetilde{X} \subset \tilde{f}^{-1}(0)$. Note that $\widetilde{X}$ and $X$ are birational but usually not isomorphic. The third category which is relevant for us is 
  \begin{equation}
    \Dsing(\widetilde{X}) = \MF(\widetilde{Y}',\tilde{f})
  \end{equation}
  which again is genuinely 2-periodic.
  %(QUESTION: do we need to assume that the central fiber $\widetilde{X}$ is still reduced? Or is this automatic?)
\end{enumerate}
We shall see that under certain hypotheses, the second category is equivalent to a localization of the first category and also a localization of the third. The third category is often equivalent to the Fukaya category of a Riemann surface, which allows us to connect the second kind of category to the Fukaya categories of nodal Riemann surfaces \`{a} la Jeffs \cite{jeffs}: see Corollary \ref{cor:relative-sing-fuk} and Section \ref{sec:singular-divisor-in-k3}.

The category $\MF(Y,L,s)$ is $L$-twisted $2$-periodic, whereas the other two categories are $2$-periodic. An obvious way to get from the $L$-twisted $2$-periodic setting to the untwisted $2$-periodic setting is to localize the category with respect to a natural transformation $s' : \id \to L\otimes$, which is to say, a section of the line bundle $L$.\footnote{Since there is preferred isomorphism between $L\otimes $ and $[2]$, it would be equivalent to localize with respect to a natural transformation $t' : \id \to [2]$.}  On the other hand, the choice of $s'$ gives us the data necessary to construct the variety $Y'$ described above. Thus we seek to relate the localization $\MF(Y,L,s)[s'^{-1}]$ to $\MF(Y',f)$.

Take $(Y,L,s)$ with $X = s^{-1}(0)$ and $Y$ smooth. Choose an auxiliary section $s' \in \Gamma(Y,L)$ so that $\langle s, s' \rangle \subset \Gamma(Y,L)$ defines a pencil of divisors on $Y$. Consider $W = \{us + vs' = 0\} \subset Y \times \bP^{1}$, where $[u:v]$ are homogeneous coordinates on $\bP^{1}$. Then $\pi_{2} : W \to \bP^{1}$ is a morphism whose fibers are the members of the pencil. The fibers of $\pi_{1} : W \to Y$ are points except over the base locus $B = s^{-1}(0) \cap s'^{-1}(0)$ where they are lines. The scheme $W$ need not be smooth: the singularities of $W$ correspond to points on the base locus where $X = X_{0}  =s^{-1}(0)$ and $X_{\infty} = s'^{-1}(0)$ fail to be transverse. More precisely, $(p,[u:v])$ is a singular point of $W$ if and only if $p \in B$ and $u\,ds + v\,ds' = 0$ as a differential on $Y$. Equivalently, $(p,[u:v])$ is a singular point of $W$ if and only if $p \in B$ and $X_{[u:v]}$ has a singularity at $p$.

We now have two embeddings of $X$: the original embedding of $X \subset Y$, and the embedding of $X = X_{0} \subset W$. The absolute singularity category of $X$ is
\begin{equation}
  \Dsing(X) \cong \MF(Y,L,s),
\end{equation}
and this category is $L$-twisted $2$-periodic. Let $Y' = W\setminus X_{\infty}$, so that $\pi_{2} : W \to \bP^{1}$ restricts to a map that we denote $f : Y' \to \bA^{1}$. Then the relative singularity category of $X$ in $W$ or $Y'$ is
\begin{equation}
  \Dsing(X,W) \cong \Dsing(X,Y') \cong \MF(Y',f),
\end{equation}
and this category is $2$-periodic without any twisting. Recall that in constructing $W$ and $Y'$ from $(Y,L,s)$, we made the choice of an additional section $s' \in \Gamma(Y,L)$. Thus $\MF(Y,L,s)$ carries both a natural isomorphism $t : \id \to L^{-1}\otimes [2]$ and a natural transformation (not isomorphism) $s': \id \to L\otimes$. Localizing $\MF(Y,L,s)$ by formally inverting $s'$ leads to a $2$-periodic category that is closely related to $\Dsing(X,Y')$, and which is in many cases equivalent to it.

The statement of the next theorem refers to support conditions for $\Dsing$, these are to be interpreted in terms of the theory of Section \ref{sec:sing-supp}. There is a singular support variety $\gsing(X)$ which is a stratified vector bundle over $X$. Objects in $\Coh(X)$ have a notion of singular support in $\gsing(X)$, and objects in $\Dsing(X)$ have a notion of singular support in $\gsing^{\circ}(X) = \gsing(X) \setminus X$. If $Z \subset \Sing(X)$ is a closed subset of the singular locus of $X$, then $\Dsing_{Z}(X)$ denotes the subcategory of $\Dsing(X)$ with objects having singular support contained in $p^{-1}(Z)$, where $p: \gsing^{\circ}(X) \to \Sing(X)$ is the projection. Note that a coherent sheaf that represents an object of $\Dsing_{Z}(X)$ may have ordinary support outside of $Z$: the condition is that the singular support of this coherent sheaf is contained in $X \cup p^{-1}(Z)$.

\begin{theorem}
  \label{thm:relative-localization}
  Let $Y,L,s,s',Y',X$ be as above. Then there is a natural functor
  \begin{equation}
    \psi: \Dsing(X,Y') \to \MF(Y,L,s)[s'^{-1}].
  \end{equation}
  Let $Z = X\cap \Sing(Y')$. If the image of the functor $i^{*}: \Dsing_{Z}(Y') \to \Dsing_{Z}(X)$ split-generates $\Dsing_{Z}(X)$, then $\psi$ is an equivalence.
\end{theorem}

\begin{proof}
  Since $\MF(Y,L,s)[s'^{-1}]$ is the localization at $s'$ of the localization $\Coh(X)/\Perf(X)$, by doing the localizations in the other order, we see that this category is equivalent to $\Coh(X)[s'^{-1}]/\Perf(X)[s'^{-1}]$. Now localizing $\Coh(X)$ at the section $s'^{-1}$ corresponds to deleting the vanishing locus of $s'$ on $X$, which is simply the base locus $B$, so we find
  \begin{equation}
    \MF(Y,L,s)[s'^{-1}] \cong \Dsing(X \setminus B).
  \end{equation}

  Recall that $\Dsing(X,Y') \cong \Dsing(X)/i^{*}\Dsing(Y')$. There is an obvious functor $\Dsing(X) \to \MF(Y,L,s)[s'^{-1}] \cong \Dsing(X\setminus B)$, which amounts to restriction to an open subscheme. In order to show that $\psi$ is well defined, we claim that $i^{*}\Dsing(Y')$ is in the kernel. By a theorem of Orlov \cite{orlov-completion}, the category $\Dsing(Y')$ is split-generated by objects set-theoretically supported on $\Sing(Y')$. Thus $i^{*}\Dsing(Y')$ is split-generated by objects supported on $Z = X \cap \Sing(Y')$. Now by the discussion above $Z = X \cap \Sing(Y')$ may be identified with $\Sing(X) \cap B$, and in particular $Z \subset B$, so $i^{*}\Dsing(Y')$ is sent to zero as claimed. This completes the proof of the first assertion.

  For the second assertion, consider that $\Dsing(X)$ is split-generated by objects supported on $\Sing(X)$, so restriction to the complement $X\setminus B$ has the effect of quotienting by objects supported on $\Sing(X)\cap B = Z$. In other words, the functor $\Dsing_{Z}(X) \to \Dsing(X)$ essentially surjects onto the kernel of the localization $\Dsing(X) \to \Dsing(X\setminus B)$. Suppose that the image of $i^{*}: \Dsing_{Z}(Y')\to \Dsing_{Z}(X)$ split-generates the target category. Then any object that is sent to zero by $\psi$ is a summand of an object in the image of $\Dsing_{Z}(Y')$ in $\Dsing(X)$, hence is a summand of an object in $i^{*}\Dsing(Y')$, hence is zero in $\Dsing(X,Y')$.
\end{proof}

\subsection{Relative singularity categories and resolutions}

Another thing we may consider is a resolution $\pi : \widetilde{Y}' \to Y'$. The function $f : Y' \to \bA^{1}$ pulls back to a function $\tilde{f} : \widetilde{Y}' \to \bA^{1}$. We let $\widetilde{X}$ be the strict transform of $X = f^{-1}(0)$ under the resolution. Then we have $\widetilde{X} \subset \tilde{f}^{-1}(0)$. Note that $\widetilde{X}$ and $X$ are birational but usually not isomorphic.

%DO WE NEED THIS? In what follows we shall assume that $\widetilde{X} = \tilde{f}^{-1}(0)$; this means that the parts of the exceptional locus mapping to $X$ are contained in $\widetilde{X}$.

The question is then to relate the singularity categories $\Dsing(\widetilde{X}) = \MF(\widetilde{Y}',\tilde{f})$ and $\Dsing(X,Y') = \MF(Y',f)$. There is a simple relationship based on the fact that $\widetilde{X}$ is birational to $X$. Recall $Z = X \cap \Sing(Y') = \Sing(X) \cap B$. We assume that the birational map $\pi: \widetilde{X} \to X$ is an isomorphism outside of $Z$, so that $\widetilde{X} \setminus \pi^{-1}(Z) \cong X \setminus Z$; this holds automatically if the resolution $\pi : \widetilde{Y}' \to Y'$ is an isomorphism over the nonsingular locus.

\begin{theorem}
  \label{thm:resolution-localization}
  Let the notation be as above. Suppose that $\pi : \widetilde{Y}' \to Y'$ is an isomorphism over the nonsingular locus, or at least that $\pi : \widetilde{X} \to X$ is an isomorphism outside of $Z$. Then there is an equivalence
  \begin{equation}
    \MF(Y,L,s)[s'^{-1}] \cong \Dsing(\widetilde{X} \setminus \pi^{-1}(Z))
  \end{equation}
  If the image of $i^{*} : \Dsing_{Z}(Y') \to \Dsing_{Z}(X)$ split-generates $\Dsing_{Z}(X)$, these categories are equivalent to $\Dsing(X,Y')$ as well.
\end{theorem}

\begin{proof}
  This follows from the equivalences $\MF(Y,L,s)[s'^{-1}] \cong \Dsing(X\setminus B) \cong \Dsing(X\setminus Z)$ from the proof of the previous theorem, and the isomorphism $X \setminus Z = \widetilde{X} \setminus \pi^{-1}(Z)$.
\end{proof}

\subsection{The case of normal crossings surfaces}

Next we return to the normal crossing surface case. Consider a triple $(Y,L,s)$ with $Y$ a smooth threefold such that $X = s^{-1}(0)$ is a normal crossings surface with graph-like singular locus. We wish to choose the additional section $s' \in \Gamma(Y,L)$ so that it is in general position with respect to $X$. Specifically, we require that for any irreducible component $C \subset \Sing(X)$, the section $s'|_{C} \in \Gamma(C,L|_{C})$ has only simple zeros, and that these zeros are disjoint from the triple points. Note that this implies that $\deg(L|_{C}) \geq 0$. In a certain sense, the base locus $B = s'^{-1}(0) \cap X$ is transverse to $\Sing(X)$. The subscheme $Z = B \cap \Sing(X)$ is a discrete set of points contained in $\Sing(X)$ but disjoint from the set of triple points.

The variety $Y'$ is obtained by blowing up $B$ and deleting the fiber at $\infty$, and the singularities of $Y'$ along $X$ are ordinary double points occurring at $Z = X \cap \Sing(Y')$. We choose a resolution $\pi: \widetilde{Y}' \to Y$ that is an isomorphism over the nonsingular locus (the resolution is not unique, but in our examples there are usually several natural resolutions that are related by flops).

We now analyze the functor $i^{*} : \Dsing_{Z}(Y')\to \Dsing_{Z}(X)$. Since $Z$ is a disjoint union of points, we may work one point at time, so let $z$ be some point of $Z$. Let $k_{z,Y'} \in \Coh_{Z}(Y')$ and $k_{z,X} \in \Coh_{Z}(X)$ be the skyscraper sheaves at $z$ with respect to the two different ambient varieties.
\begin{lemma}
  \label{lem:skyscraper-pullback}
  For the (derived, as always) pullback functor $i^{*}: \Coh_{Z}(Y') \to \Coh_{Z}(X)$, we have
  \begin{equation}
  i^{*}(k_{z,Y'}) \cong k_{z,X} \oplus k_{z,X}[1].    
  \end{equation}
\end{lemma}
\begin{proof}
  Because $i : X \to Y'$ is the inclusion of the hypersurface $X = f^{-1}(0)$, we have the fundamental exact triangle of endofunctors of $\Coh(X)$
  \begin{equation}
    \id[1] \to i^{*}i_{*} \to \id \to \id[2].
  \end{equation}
  Applying this to $k_{z,X}$, and using $i_{*}(k_{z,X}) = k_{z,Y'}$, we obtain an exact triangle
  \begin{equation}
    k_{z,X}[1] \to i^{*}(k_{z,Y'}) \to k_{z,X} \to k_{z,X}[2]
  \end{equation}
  The map $i^{*}(k_{z,Y'}) \to k_{z,X}$ induced from the counit of the adjunction is surjective, so the connecting map $k_{z,X} \to k_{z,X}[2]$ vanishes, and the conclusion follows.
\end{proof}

\begin{proposition}
  \label{prop:skyscraper-split-gen}
  Let $X$ be a normal crossings surface with graph-like singular locus, and let $z$ be a point on the singular locus that is not a triple point. Then $\Dsing_{\{z\}}(X)$ is split-generated by the skyscraper sheaf $k_{z,X}$.
\end{proposition}

\begin{proof}
  It suffices to work formally locally near $z$, so that we can reduce to the case where $X = \{xy=0\}$ in the ambient space $\bA^{2}_{xy} \times \bG_{m,\zeta} = \{(x,y,\zeta) \mid \zeta \neq 0\}$, and $z = (0,0,1)$. The category $\Dsing(X)$ in this case is equivalent by Kn\"{o}rrer periodicity to $\Perf^{(2)}(\bG_{m,\zeta})$, and the condition of having support at $z = (0,0,1)$ goes over to the condition of having support at $\zeta = 1$, so $\Dsing_{\{z\}}(X) \cong \Coh^{(2)}_{\{\zeta = 1\}}(\bG_{m,\zeta})$. The category $\Coh^{(2)}_{\{\zeta = 1\}}(\bG_{m,\zeta})$ is generated by the skyscraper $k_{\{\zeta=1\}}$.

  Passing back through Kn\"{o}rrer periodicity, $k_{\{\zeta=1\}}$ corresponds to (a shift of) $\cO_{L_{y}}$, where $L_{y} = \{x = 0,\zeta =1\}$ is a line in one component of $X$ transverse to the singular locus. Note that $\cO_{L_{y}} \simeq \cO_{L_{x}}[1]$ where $L_{x}=\{y= 0,\zeta = 1\}$. Thus we find that $\cO_{L_{y}}$ or $\cO_{L_{x}}$ generates $\Dsing_{\{z\}}(X)$.

  We now claim that $\cO_{L_{y}}$ is a summand of $k_{z,X}$. In fact $k_{z,X} \simeq \cO_{L_{y}}\oplus \cO_{L_{x}}$ in $\Dsing(X)$. This may be seen by considering $\cO_{\zeta=1}$, the perfect complex $\cO_{X} \to \cO_{X}$ with differential equal to multiplication by $\zeta -1$. Geometrically $\cO_{\zeta=1}$ represents the union $L_{x}\cup L_{y}$. There is a surjective map $\cO_{\zeta=1} \to k_{z,X}$ whose kernel is the ideal sheaf of $z$ inside $\cO_{\zeta=1}$. This kernel is also the direct sum of the ideal sheaves of $z$ inside $\cO_{L_{x}}$ and $\cO_{L_{y}}$ respectively, and those ideal sheaves are in turn isomorphic to the structure sheaves of the two lines.
\end{proof}

\begin{lemma}
  \label{lem:pullback-gens}
  In the situation above, the image of the pullback functor $i^{*}: \Dsing_{Z}(Y') \to \Dsing_{Z}(X)$ split generates $\Dsing_{Z}(X)$.
\end{lemma}

\begin{proof}
  Consider the collection of skyscraper sheaves $\{k_{z,Y'}\}_{z \in Z}$. (Since $Y'$ has ordinary double points at the points of $Z$, these objects split generate $\Dsing(Y')$, but this fact is not used in this argument.) By Lemma \ref{lem:skyscraper-pullback}, these objects map to $\{k_{z,X}\oplus k_{z,X}[1]\}_{z \in Z}$. Hence the objects $\{k_{z,X}\}_{z\in Z}$ are in the split-closure of the image of $i^{*}$, and by Proposition \ref{prop:skyscraper-split-gen}, the split-closure is all of $\Dsing_{Z}(X)$.
\end{proof}

\begin{theorem}
  \label{thm:relative-graph-like}
  Let $Y$ be a smooth threefold, $L$ a line bundle on $Y$, and $s$ section of $L$ such that $X = s^{-1}(0)$ is a normal crossings surface with graph-like singular locus. Let $s'$ be another section of $L$ whose restriction to $\Sing(X)$ has only simple zeros at points which are not triple points. Let $Y'$ be the singular threefold constructed from the pencil spanned by $s$ and $s'$. Then there is an equivalence of categories
  \begin{equation}
    \Dsing(X,Y') \cong \MF(Y,L,s)[s'^{-1}].
  \end{equation}
  
  Let $Z = \Sing(Y') \cap X$. Let $\pi : \widetilde{Y}' \to Y'$ be a resolution of $Y'$ that is an isomorphism over the nonsingular locus, and $\widetilde{X}$ be the strict transform of $X$, then there is an equivalence of categories
  \begin{equation}
    \MF(Y,L,s)[s'^{-1}] \cong \Dsing(\widetilde{X}\setminus \pi^{-1}(Z)).
  \end{equation}
\end{theorem}

\begin{proof}
  This is the combination of Theorem \ref{thm:relative-localization}, Theorem \ref{thm:resolution-localization} and Lemma \ref{lem:pullback-gens}.
\end{proof}

According to our previous work \cite{PS21}, as long as the dual intersection complex of $X$ is orientable, the category $\Dsing(\widetilde{X})$ is equivalent to the Fukaya category of a Riemann surface, since $\widetilde{X}$ is the fiber of a morphism $f : \widetilde{Y}'\to \bA^{1}$ in addition to being a normal crossings surface with graph-like singular locus. Thus we may connect relative singularity categories to Fukaya categories.

\begin{corollary}
  \label{cor:relative-sing-fuk}
  Take the notation as in Theorem \ref{thm:relative-graph-like}, and suppose further that the dual intersection complex of $X$ is orientable. Let $\Sigma$ be the Riemann surface such that $\Fuk(\Sigma) \simeq \Dsing(\widetilde{X})$. Then $\MF(Y,L,s)[s'^{-1}]$ and $\Dsing(X,Y')$ are equivalent to a localization of $\Fuk(\Sigma)$ obtained by quotienting objects supported on simple closed curves in $\Sigma$.
\end{corollary}

\begin{proof}
  The localization $\Dsing(\widetilde{X} \setminus \pi^{-1}(Z))$ is obtained by quotienting by the skyscraper sheaves at the points of $Z$, which correspond to objects in $\Fuk(\Sigma)$ supported on simple closed curves.
\end{proof}

\section{Symplectic constructions and conjectures}
\label{sec:symplectic-conjectures}

We seek to construct a symplectic counterpart to $L$-twisted singularity categories considered above. Our proposal depends on the theory of Rabinowitz Fukaya categories that is currently under development by Ganatra, Gao, and Venkatesh \cite{ganatra-talk}. Since this theory is very recent, we will content ourselves to describing a conjectural picture.

Part of the conjectural story is more general than our intended application. It is about the Fukaya categories of general type symplectic manifolds. Our conjectures, together with other expected properties of Rabinowitz Fukaya categories, imply certain conjectures of Lekili and Ueda \cite{lekili-ueda-milnor,lekili-ueda-complements} about complements of smooth ample divisors in smooth Calabi-Yau varieties.

In Section \ref{sec:singular-divisor-in-k3} below we present a generalization of a conjecture of Lekili and Ueda to complements of nodal ample divisors in smooth K3 or abelian surfaces, drawing on the work of Maxim Jeffs \cite{jeffs}.

Owing to the conjectural nature of this section, it contains statements of varying logical status:
\begin{enumerate}
\item Statements that are rigorously proven or whose proof presents no essential difficulties. The geometric constructions dealing with symplectic four-manifolds are of this kind.
\item Statements involving Rabinowitz Fukaya categories that may be deduced from the expected properties of these categories \cite{ganatra-talk}. We formulate such statements as Propositions with these expected properties as hypotheses.
\item Statements involving Rabinowitz Fukaya categories whose proof will require one to establish new properties of these categories. We present such statements as Conjectures.
\end{enumerate}
It is also the case that some special cases of our conjectures may be deduced from known results.

\subsection{Principal canonical bundles of general type symplectic manifolds}

We begin by recalling the standard story of prequantization spaces associated to general type varieties.

Let $Z$ be a general type variety of dimension $n$, meaning that the canonical line bundle $K_{Z}$ is ample. We can make $Z$ into a monotone symplectic manifold by equipping it with a symplectic form $\omega$ that arises as the curvature of a connection on $K_{Z}$: more precisely we may choose a Hermitian metric $h$ on $K_{Z}$ such that the curvature of the associated Chern connection is $i\omega$ where $\omega$ is a symplectic form. Assuming such choices have been made, we refer to $(Z,\omega)$ as a \emph{general type symplectic manifold}.

There is a principal $S^{1}$-bundle $S_{Z} \subset K_{Z}$ given by vectors of unit $h$-norm. The connection has an associated $1$-form $\theta \in \Omega^{1}(S_{Z})$ that satisfies $d\theta = \pi^{*}\omega$, where $\pi$ denotes the projection to $Z$. It follows that $\theta$ is a contact form on $S_{Z}$. We denote the symplectization of $(S_{Z},\theta)$ by $P_{Z} \simeq S_{Z} \times \R$. Observe that $P_{Z}$ is diffeomorphic to the total space of the canonical bundle minus the zero section: $P_{Z} \simeq K_{Z} \setminus Z$. This is witnessed by the two projections
\begin{equation}
K_{Z} \setminus Z \to S_{Z},\  v \mapsto v/\sqrt{h(v,v)}, \quad K_{Z} \setminus Z \to \R,\  v \mapsto \log h(v,v).
\end{equation}
We use this diffeomorphism to put a symplectic structure on $K_{Z}\setminus Z$.
Being a symplectization, $P_{Z}$ has a convex end and a concave end; the convex end is the one where $h(v,v) \to 0$, and the concave end is the one where $h(v,v) \to \infty$. We could therefore say that the zero section $Z$ sits at the \emph{convex} end. This is because $K_{Z}$ is the total space of a \emph{positive} line bundle. (This is opposite to the situation that arises for the total space of the canonical bundle over a Fano variety).

The presentation of $P_{Z}$ as a symplectization singles out a homotopy class of almost complex structures, namely those compatible with the symplectic structure. This class contains the complex structure on $P_{Z} \simeq K_{Z}\setminus Z$ arising from the fact that $K_{Z}$ is a complex manifold. This allows us to compute the Chern classes of $P_{Z}$. In fact $c_{1}(P_{Z}) = 0$, since the canonical bundle of $K_{Z}$, a fortiori $P_{Z}$, has a canonical homotopy class of trivializations. We denote by $\Omega_{Z}$ the complex volume form that induces this trivialization. The form $\Omega_{Z}$ makes $P_{Z}$ into a graded symplectic manifold.

Since $S_{Z}$ is a prequantum principal $S^{1}$-bundle, $P_{Z}$ carries an $S^{1}$-action that is free and Hamiltonian with moment map $v \mapsto h(v,v)$. The reduced spaces of this action are copies of $(Z,\lambda \omega)$ for $\lambda > 0$.

As mentioned above, $P_{Z}$ is not a convex symplectic manifold, but it belongs to a class of manifolds (completions of Liouville cobordism) for which a version of the Fukaya category can be defined. This is the Rabinowitz Fukaya category developed by Ganatra-Gao-Venkatesh. At least morally, the Rabinowitz Fukaya category is a variant of the wrapped Fukaya category where the wrapping is implemented at both the convex and concave ends.  We denote by
\begin{equation}
  \RFuk^{\gr}(P_{Z},\Omega_{Z})
\end{equation}
the Rabinowitz Fukaya category generated by Lagrangian branes that are graded with respect to $\Omega_{Z}$.

\begin{remark}
  For any ample line bundle $L$ over $Z$, we could form associated the principal bundle $P \cong L \setminus Z$, and give it a symplectic structure and Hamiltonian $S^{1}$-action. We could then contemplate the Rabinowitz Fukaya category $\RFuk(P)$. The case where $L = K_{Z}$ that we are studying here is distinguished by the existence of a natural grading, that is, the complex volume form $\Omega_{Z}$, and hence the possibility of considering the graded Rabinowitz Fukaya category $\RFuk^{\gr}(P_{Z},\Omega_{Z})$.
\end{remark}

While we shall not provide an introduction to the theory of Rabinowitz Fukaya categories, there are some technical points that should be mentioned.\footnote{We thank Mohammed Abouzaid and Sheel Ganatra for discussions of these points.} One is that the construction of $\RFuk(Y)$ requires a filling of the concave end of $Y$, and the resulting category depends on this choice \emph{a priori}. In what follows, we shall tacitly assume that a filling exists; in our applications to complements of divisors in K3 and abelian surfaces there is always a preferred filling. We shall also tacitly assume that the resulting category does not depend up to equivalence on the choice of filling in the cases of interest; should this fail to be true, our conjectures could be weakened to say ``there is some filling such that'' the conclusion holds. Another technical point is that $\RFuk(Y)$ seems to depend on the Liouville structure of $Y$ in a nontrivial way, so that an abstract symplectomorphism $Y \simeq Y'$ does not necessarily imply an equivalence $\RFuk(Y) \simeq \RFuk(Y')$. We believe that our arguments do not run afoul of this difficulty, since the symplectomorphisms we use to compare categories are sufficiently ``tame,'' but this is another assumption we make henceforth.

\subsection{An ``Orlov equivalence''}

We expect on general grounds that the Hamiltonian $S^{1}$-action on $P_{Z}$ induces an action of $C_{*}(\Omega S^{1}) \simeq C_{*}(\Z) = k[u,u^{-1}]$ on $\RFuk^{\gr}(P_{Z},\Omega_{Z})$. More broadly, we expect an action of the wrapped Fukaya category $\Fw(T^{*}S^{1}) \cong C_{*}(\Omega S^{1})\modules$ on $\RFuk^{\gr}(P_{Z},\Omega_{Z})$, which is induced by a certain Lagrangian correspondence constructed from the moment map \cite{teleman-icm,evans-lekili}. A key point involves the grading: the symplectic manifold $T^{*}S^{1} \cong \mathbb{C}^{\times}$ should be graded by $dz$ rather than the more common $z^{-1}\,dz$ in order for the action to be compatible with the grading $\Omega_{Z}$. This is because the grading $\Omega_{Z}$ on $P_{Z} \cong K_{Z}\setminus Z$ extends over the zero section. The main consequence of this is that the variable $u$ must be given degree $2$.

We may also consider $\Fuk(Z)$, the \emph{balanced $2$-periodic Fukaya category} of $Z$: this means that the objects are supported on Lagrangians in $Z$ that are balanced with respect to the chosen one-form $\theta \in \Omega^{1}(S_{Z})$ satisfying $d\theta = \pi^{*}\omega$, and hence that $\Fuk(Z)$ is defined over $k$ rather than the Novikov field. Our first conjecture says that $\Fuk(Z)$ may be recovered from the Rabinowitz Fukaya category of $P_{Z}$. One might consider it to be a kind of Orlov equivalence.
\begin{conjecture}
  \label{conj:orlov}
  With the notation as above, there is a $k[u,u^{-1}]$-linear ($\deg(u) = 2$) equivalence of categories
  \begin{equation}
    \RFuk^{\gr}(P_{Z},\Omega_{Z}) \simeq \Fuk(Z).
  \end{equation}
  The $k[u,u^{-1}]$-action on the left-hand side is the one induced by the Hamiltonian $S^{1}$-action on $P_{Z}$, and the $k[u,u^{-1}]$-action on the right-hand side is the $2$-periodic structure.
\end{conjecture}

We now sketch some ideas for a possible proof of this conjecture.

On the one hand, we expect there to be a functor $\RFuk^{\gr}(P_{Z},\Omega_{Z}) \to \Fuk(Z)$ given by Hamiltonian reduction with respect to the $S^{1}$-action on $P_{Z}$.

To construct a functor in the other direction, we begin with the question: When does a Lagrangian in $L \subset Z$ lift to a gradable Lagrangian  $\widetilde{L} \subset P_{Z}$? Note that restricting the circle bundle to a Lagrangian, that is, taking $S_{Z}|_{L}$, does produce a Lagrangian, but this is \emph{never} gradable with respect to $\Omega_{Z}$. Noting that $P_{Z}\simeq K_{Z}\setminus Z$, we may consider
the real line bundle over $L$ given by $\det_{\R}(T^{*}L) \subset \det_{\mathbb{C}}(T^{*}Z) = K_{Z}$. Removing zero disconnects $\det_{\R}(T^{*}L)$, and choosing an orientation of $L$ selects one component, and thus we obtain an $\R_{+}$-bundle $\widetilde{L}$ over $L$ that is embedded in $P_{Z}$. The manifold $\widetilde{L}$ is not necessarily Lagrangian even when $L$ is, because the canonical line bundle $K_{Z}$ restricted to $L$, which is flat by the Lagrangian condition, may have nontrivial holonomy. In fact, if we equip $Z$ with a K\"{a}hler-Einstein form $\omega$, then the condition that $\widetilde{L}$ be Lagrangian in $P_{Z}$ is equivalent to the condition that $L$ is a minimal Lagrangian in $Z$. On the other hand, it will suffice for our purposes to assume that $L$ is balanced, meaning that $K_{Z}$ restricted to $L$ has trivial holonomy; in this case the construction of $\widetilde{L}$ may be deformed to produce a Lagrangian.

We propose that there is a functor $\Fuk(Z) \to \RFuk^{\gr}(P_{Z},\Omega_{Z})$, which on objects is given by the lifting process $L \mapsto \widetilde{L}$ for balanced Lagrangian $L$ described above, and that this is an inverse to the Hamiltonian reduction functor. Very roughly, given two balanced Lagrangians $L_{0},L_{1}$ in $Z$, and a transverse intersection point $x \in L_{0}\cap L_{1}$, their lifts $\widetilde{L}_{0},\widetilde{L}_{1}$ intersect cleanly in a ray. The Rabinowitz-type wrapping applied to $(\widetilde{L}_{0},\widetilde{L}_{1})$ then produces a family of orbits indexed by $\Z$, where the degree difference between two neighboring orbits is $2$. We expect this forms a free orbit for the $k[u,u^{-1}]$-action, matching what we find when taking the $2$-periodic morphism space of $L_{0},L_{1}$ in $Z$.

\subsection{Relationship with a conjecture of Lekili-Ueda}

Lekili and Ueda have proposed a general conjecture \cite[Conjecture 1.5]{lekili-ueda-milnor} relating the Fukaya category of an ample divisor in a Calabi-Yau manifold to the wrapped Fukaya category of its complement.
\begin{conjecture}[Lekili-Ueda]
  \label{conj:lekili-ueda}
  Let $V$ be a compact Calabi-Yau manifold, and let $Z \subset V$ be a smooth ample divisor. Set $U = V \setminus Z$, so that $U$ admits the structure of a Weinstein manifold. Then there is an equivalence of categories
  \begin{equation}
    \Fw(U)/\Fc(U) \simeq \Fuk(Z)
  \end{equation}
  where the left-hand side is the quotient of the (graded) wrapped Fukaya category of $U$ by the subcategory generated by objects supported on compact Lagrangians in $U$.
\end{conjecture}

This conjecture is attractive because it would give us access to the Fukaya category of the \emph{compact} general type manifold $Z$ by way of the wrapped Fukaya category of the noncompact Weinstein manifold $U$, where many more techniques are available. Also it is very general: many different general type varieties $Z$ may appear as ample divisors in Calabi-Yau manifolds.

While we do not have a proof of Conjecture \ref{conj:lekili-ueda}, we wish to point out that our Conjecture \ref{conj:orlov} is related to it.

To make this connection, we need to spell out one property of the Rabinowitz Fukaya category that is expected to hold in great generality. Suppose that $U$ is a Weinstein manifold with $c_{1}(U) = 0$, and choose a grading on $U$. Then we may consider the graded versions of: the wrapped Fukaya category $\Fw(U)$, the compact Fukaya category $\Fc(U)$, and also the Rabinowitz Fukaya category $\RFuk(U)$. Note that the Rabinowitz Fukaya category probes the geometry of the end of $U$, but it \emph{a priori} depends on the filling. Ganatra-Gao-Venkatesh have shown that $\RFuk(U)$ is related to $\Fw(U)$ by a purely algebraic construction known as the \emph{formal punctured neighborhood at infinity}, denoted $\widehat{\Fw(U)}_{\infty}$ \cite{efimov-punctured}. In some cases, this description may simplify, so that the formal punctured neighborhood is wrapped modulo compact:
\begin{equation}
  \label{eq:wrapped-mod-compact}
  \RFuk(U) \cong \widehat{\Fw(U)}_{\infty} \cong \Fw(U)/\Fc(U).
\end{equation}

The following proposition is, strictly speaking, conjectural, but we call it a proposition since the only gaps are general expected properties of the Rabinowitz Fukaya category.

\begin{proposition}
  Conjecture \ref{conj:lekili-ueda} is true provided that $Z$ satisfies Conjecture \ref{conj:orlov}, and $U$ satisfies the equivalences \eqref{eq:wrapped-mod-compact}.
\end{proposition}

\begin{proof}
  The stated assumptions reduce the problem to showing an equivalence of categories
  \begin{equation}
    \RFuk(U) \simeq \RFuk^{\gr}(P_{Z},\Omega_{Z}).
  \end{equation}
  Note that in our notation, the left hand side is indeed a $\Z$-graded category, graded by a choice of complex volume form $\Omega_{V}$ on $V$. The adjunction formula states
  \begin{equation}
    K_{Z} \simeq K_{V} \otimes NZ \simeq NZ
  \end{equation}
  where $NZ$ is the normal bundle of $Z$ in $V$. Therefore the end of $U$, which is diffeomorphic to a punctured tubular neighborhood of $Z$ in $NZ$, is diffeomorphic to a punctured tubular neighborhood of $Z$ in $K_{Z}$, which is nothing but $P_{Z}$. Note that both the end of $U$ and $P_{Z}$ are equipped with symplectic structures that are convex as one approaches $Z$, and that in both cases the grading form $\Omega_{V}$ and $\Omega_{Z}$ extend over $Z$. Thus the end of $U$ and $P_{Z}$ are equivalent as graded symplectic manifolds, and the desired equivalence follows from general invariance properties of the Rabinowitz Fukaya category.
\end{proof}

\begin{remark}
  The ideas presented here suggest another, less direct, approach to proving Conjectures \ref{conj:orlov} and \ref{conj:lekili-ueda}.  Suppose it were possible to prove that $\RFuk^{\gr}(P_{Z},\Omega_{Z})$ is independent of the filling (restricting perhaps to the class of fillings to those arising from divisor complements in Calabi-Yau manifolds). Then it would suffice to prove Conjectures \ref{conj:orlov} and \ref{conj:lekili-ueda} in a single case which we could take to be as convenient as possible for computation. For instance, when $Z$ is the genus two curve, we could take $Z \subset T^{4}$ to be the embedding of $Z$ as a divisor in its Jacobian.
\end{remark}

\subsection{The case of curves}
\label{sec:case-of-curves}

For the remainder of this section we shall restrict to the case where $Z$ is a curve. Our aim is to generalize the Lekili-Ueda conjecture to the case where $Z$ is a nodal curve.

Let $Z$ be a compact Riemann surface of genus $g$. Suppose that $Z$ is embedded as an ample divisor in a smooth K3 surface $V$, and let $U = V \setminus Z$. Then a special case of Conjecture \ref{conj:lekili-ueda} states that there is an equivalence
\begin{equation}
  \Fw(U)/\Fc(U) \simeq \Fuk(Z).
\end{equation}
As explained by Lekili and Ueda \cite{lekili-ueda-complements}, this statement is related to homological mirror symmetry for $U$. The mirror to $U$ is a singular K3 surface $X$. More precisely, $X$ is a type III K3 surface: in the terminology of Section \ref{sec:graph-like}, this means that $X$ is a normal crossings surface with graph-like singular locus such that $\deg(\gsing(X)|_{C}) = 0$ for each $C \simeq \bP^{1}$ contained in the singular locus. The genus $g$ corresponds to the so-called \emph{index} of $X$. In this setting, homological mirror symmetry states that there are equivalences of categories
\begin{equation}
  \Coh(X) \simeq \Fw(U),\quad \Perf(X) \simeq \Fc(U).
\end{equation}
Of course, $\Dsing(X) \cong \Coh(X)/\Perf(X)$, so we get an equivalence of localization sequences
\begin{equation}
  \label{eq:case-of-curves}
  \xymatrix{
    \Fc(U) \ar[r]\ar[d]& \Fw(U) \ar[r]\ar[d] & \Fuk(Z) \ar[d]\\
    \Perf(X)\ar[r]\ar[u] & \Coh(X) \ar[u] \ar[r]& \Dsing(X) \ar[u]
  }
\end{equation}
We shall seek to generalize this diagram to the case where $Z$ is a nodal ample divisor in a K3 or abelian surface $V$. We shall see that in place of $\Fuk(Z)$ we have $\RFuk(U)$, and we conjecture a definite relationship between this category and $\Fuk(Z)$, where $\Fuk(Z)$ is taken in the sense of Jeffs \cite{jeffs}. In this case, we expect $X$ to be a normal crossings surface with graph-like singular locus, but not necessarily a type III K3 surface.

\begin{remark}
  It is perhaps worth mentioning that this circle of ideas provides a strategy to prove Conjecture \ref{conj:lekili-ueda} for a smooth compact Riemann surface $Z$. Assume that HMS for the pair $(U,X)$ has been established, at least in the sense that the first two columns of diagram \eqref{eq:case-of-curves} are valid. Then we deduce that $\Dsing(X) \cong \Coh(X)/\Perf(X) \cong \Fw(U)/\Fc(U)$. Because $X$ is a type III K3 surface, the result of \cite{PS21} supplies an equivalence $\Dsing(X) \cong \Fuk(Z)$, from which Conjecture \ref{conj:lekili-ueda} follows. This proof is perhaps unnatural since it uses HMS for $U$ and $Z$ separately.
\end{remark}

\subsection{Twisted principal canonical bundles of curves}
\label{sec:twisted-principal}
The conjectures presented so far imply that, when $X$ is a type III K3 surface, there is an equivalence of categories
\begin{equation}
  \Dsing(X) \simeq \RFuk^{\gr}(P_{Z},\Omega_{Z}),
\end{equation}
where $Z$ is a Riemann surface, and we propose to generalize this equivalence. The key point is to incorporate the twists characterized in Lemma \ref{lem:discrete-twists}.

To this end, we shall construct a certain class of symplectic four-manifolds that are associated to trivalent graphs with certain decorations. In fact, we shall describe these manifolds in four different ways in order to highlight various pieces of the structure:
\begin{enumerate}
\item The first topological construction is based on the concept of a \emph{graph manifold} coming from three-manifold theory. The symplectic manifolds we consider are diffeomorphic to a graph manifold times $\R$.
\item We then refine this construction using the theory of integrable systems and affine manifolds. This will show that our manifolds are symplectic, and carry a trivialization of their canonical bundle.
\item In Section \ref{sec:t-duality-gsing}, we use T-duality to describe our manifolds as the SYZ mirrors of $\gsing^{\circ}(X) = \gsing(X) \setminus X$, where $X$ is a normal crossings surface with graph-like singular locus. We find this independently interesting, and it provides some additional support to our HMS proposal, Conjecture \ref{conj:twisted-hms}.
\item In Section \ref{sec:tubular-neighborhood-nodal} we show that our manifolds are also obtained as punctured tubular neighborhoods of nodal curves in K3 and abelian surfaces. This description shows why our proposal works as a generalization of the Lekili-Ueda conjecture for smooth curves.
\end{enumerate}

Our starting point for the construction is a Riemann surface $Z$ equipped with a pair of pants decomposition with associated graph $G$. Let $e$ be an edge of $G$, and let $C_{e}$ be the associated simple closed curve in $Z$. Pick an integer $n_{e} \in \Z$. Recall that $p: S_{Z} \to Z$ is the unit cotangent $S^{1}$-bundle. The preimage $T_{e} = p^{-1}(C_{e})$ is a two-torus. We can cut $S_{Z}$ along $T_{e}$ and reglue using an automorphism of the two torus.

To specify a specific automorphism, we describe certain canonical framings of the torus $T_{e}$, meaning bases of $H_{1}(T_{e},\Z)$. The orientation of the surface determines an orientation of the fiber of $S_{Z}$; we denote this class by $f \in H_{1}(T_{e},\Z)$. Choosing an orientation on $C_{e}$, we obtain a canonical lift of $C_{e}$ to a loop in $T_{e}$; this is given by taking, at each point $p \in C_{e}$, the unit cotangent vector $\xi \in T^{*}_{p}Z$ whose pairing with the positively oriented tangent vector of $C_{e}$ is maximal. This leads to two classes $b_{1},b_{2}\in H_{1}(T_{e},\Z)$ corresponding to the two orientations of $C_{e}$. Note the relation $b_{1} = -b_{2}$. We call $f$ the \emph{fiber class} and we call $b_{1}$ and $b_{2}$ \emph{base classes}. Both $(b_{1},f)$ and $(b_{2},f)$ give framings of the torus $T_{e}$, and they are related by the change of basis matrix
\begin{equation}
  \phi_{0} = \begin{pmatrix} -1 & 0 \\ 0 & 1\end{pmatrix}.
\end{equation}
The idea is to replace this matrix with
\begin{equation}
  \label{eq:twist-matrix}
  \phi_{n_{e}} = \begin{pmatrix} -1 & n_{e} \\ 0 & 1\end{pmatrix}.
\end{equation}
This requires a bit of care about the framings. When we cut $S_{Z}$ along $T_{e}$ we are cutting $Z$ along $C_{e}$. This produces a bordered surface whose boundary contains two copies of $C_{e}$, we assume that each copy is equipped with the boundary orientation induced from the orientation of $Z$. Thus one copy carries orientation $b_{1}$ and the other copy carries orientation $b_{2}$. One copy of $T_{e}$ is therefore framed by $(b_{1},f)$ and the other by $(b_{2},f)$. The twisted gluing is defined by using a diffeomorphism between the two copies of $T_{e}$ such that these framings correspond under the transformation $\phi_{n_{e}}$ from \eqref{eq:twist-matrix}; note that since $\phi_{n_{e}}$ is its own inverse it does not matter in which direction we consider the transformation. 

This construction depends on a choice of pants decomposition of $Z$ and an integer $n_{e}$ for each edge $e$ in the graph $G$ associated to the pants decomposition. We denote the resulting three-manifold by $S(\{n_{e}\})$.

A more intuitive description of these twists is that they \emph{preserve the base classes, while shearing the fiber class in the direction of the base class}. The complementary kind of twist, which preserves the fiber class while shearing the base classes, is the kind that preserves the structure of an $S^{1}$-bundle and hence the $S^{1}$-action on the total space. Conversely, \emph{the twists we consider here break the $S^{1}$-bundle structure and the $S^{1}$-action}.

\begin{remark}
  The theory of three-manifolds provides terminology for describing the construction above. The unit cotangent bundle of a surface is an instance of a \emph{Seifert-fibered} three-manifold. A manifold obtained by gluing together Seifert-fibered manifolds along torus boundaries is an instance of what are termed \emph{graph manifolds}. Thus $S(\{n_{e}\})$ is a graph manifold; the ``graph'' to which this name refers is essentially our graph $G$ decorated by the twists $\{n_{e}\}$. The pair of pants decomposition of $Z$ used in the construction corresponds to the \emph{Jaco-Shalen-Johannson (JSJ) decomposition} of $S(\{n_{e}\})$.
\end{remark}

So far we have constructed a three-manifold $S(\{n_{e}\})$, but in order to take the Rabinowitz Fukaya category we need to produce a graded symplectic manifold that we shall denote $P(\{n_{e}\})$. As a smooth manifold, we shall have $P(\{n_{e}\})\simeq S(\{n_{e}\}) \times \R$, but this presentation does not make the symplectic structure or grading evident.

To construct $P(\{n_{e}\})$ precisely, we use Lagrangian torus fibrations (completely integrable systems) defined near $C_{e}$. (For now this is just a technical convenience, but in the next subsection we shall see how it connects to the SYZ picture.)

Consider a neighborhood $U$ of $C_{e}$ in $Z$, such that $U$ is fibered by circles parallel to $C_{e}$. The fibers of the bundle $q: P_{Z}\to Z$ are also fibered by circles (the orbits of the $S^{1}$-action), and by considering products we see that $q^{-1}(U)$ is fibered by two-tori that are parallel to the torus $T_{e}$ considered above. This torus fibration gives rise to a map $q^{-1}(U) \to B$ where $B$ is an integral affine manifold that is diffeomorphic to an open set in $\R^{2}$. The canonical framings of $T_{e}$ correspond to canonical integral affine coordinate systems on $B$. The set $q^{-1}(C_{e})$ is a codimension one submanifold. We then cut $P_{Z}$ along $q^{-1}(C_{e})$ and reglue after applying the shear matrix $\phi_{n_{e}}$ from \eqref{eq:twist-matrix} to the affine coordinates on $B$. Since integral linear transformations of the affine coordinates induce symplectomorphisms of the torus fibration, the result carries a symplectic structure. When performing the regluing, there is a choice: the two pieces may be translated relative to one another by some translation in (the model vector space of) the affine manifold $B$. This affects the local symplectic area of certain surfaces, and hence the symplectic structure of the resulting manifold does depend on this choice. We denote the resulting manifold by $P(\{n_{e}\})$, suppressing the dependence on the translations used in the gluing.

It is also important to realize that the gluings induced by the shear maps $\phi_{n}$ are precisely the ones the that preserve the grading form $\Omega_{Z}$ on $P_{Z}$, at least in the sense that the homotopy class of trivialization of the canonical bundle is preserved. Thus we may equip $P(\{n_{e}\})$ with a grading form, which we shall also denote $\Omega_{Z}$. We may therefore consider the (graded) Rabinowitz Fukaya category $\RFuk^{\gr}(P(\{n_{e}\}),\Omega_{Z})$, just as we did for $(P_{Z},\Omega_{Z})$.

The following conjecture connects $\RFuk^{\gr}(P(\{n_{e}\},\Omega_{Z})$ to the singularity category of a normal crossings surface $X$ where $\deg(\gsing(X))|_{C}$ is not necessarily zero.

\begin{conjecture}
  \label{conj:twisted-hms}
  Let $X$ be a proper normal crossings surface with graph-like singular locus described by the trivalent graph $G(X)$, such that the dual intersection complex is orientable. For each edge $e$ of $G(X)$ whose corresponding component $C_{e}$ of $\Sing(X)$ is isomorphic $\bP^{1}$, let $n_{e} = -\deg(\gsing(X))|_{C_{e}}$. Let $Z$ be a Riemann surface equipped with a pants decomposition whose graph is $G(X)$. Then there is a graded symplectic manifold $P(\{n_{e}\})$ obtained by the construction described above such that there is an equivalence of categories
  \begin{equation}
  \label{eq:twisted-hms}
    \Dsing(X) \simeq \RFuk^{\gr}(P(\{n_{e}\}),\Omega_{Z}).
  \end{equation}
  In particular, the right-hand side carries an autoequivalence $\Lambda$ and a natural isomorphism $t: \id \to \Lambda^{-1}[2]$, that is, a $\Lambda$-twisted $2$-periodic structure.
\end{conjecture}

\begin{remark}
  \label{rem:non-proper}
  When the surface $X$ is not proper, the corresponding Riemann surface $Z$ has punctures. While the manifold $P(\{n_{e}\})$ still exists, it is not strictly speaking a Liouville cobordism any more, since it is in fact a ``Liouville cobordism between noncompact contact manifolds.'' Therefore in order to generalize Conjecture \ref{conj:twisted-hms} to this case one would need a variant of the Rabinowitz Fukaya category that incorporates wrapping along this other type of noncompactness.

If such a theory were constructed, one might be able to prove Conjecture \ref{conj:twisted-hms} by showing that the categories of the form $\RFuk^{\gr}(P(\{n_{e}\}),\Omega_{Z})$ satisfy descent with respect to the pants decomposition that was used to construct $P(\{n_{e}\})$, and this would directly match the description that we have proved for $\Dsing(X)$ in Section \ref{sec:graph-like}. We leave it to the reader to formulate the precise conjecture.
\end{remark}

\begin{remark}
Let us explain how Conjecture \ref{conj:twisted-hms} is related to other results and conjectures in mirror symmetry.
\begin{itemize}
\item Hori-Vafa HMS predicts that categories of matrix factorizations of toric Calabi-Yau threefolds should be mirror to the Fukaya category of certain \emph{punctured} Riemann surfaces. Let 
$X$ be the toric boundary divisor of a toric Calabi--Yau threefold $Y$; in particular $X$ can be written as $W^{-1}(0)$ where $W$ is a toric superpotential $W:Y \to \mathbb{A}^1$. The fact that $X$ has graph-like singular locus is automatic in the toric setting. Also, since $X$ arises as the fiber of a regular function $n_e=0$ for all edges of $G(X)$, and thus $P(\{n_{e}\})=P_{Z}$. The classical statement of Hori-Vafa HMS is formulated as the equivalence 
\begin{equation}
  \MF(Y,W) \simeq \Fuk^{w}(Z).
\end{equation}
In this form, Hori-Vafa HMS has been fully established, thanks to the contributions of many mathematicians including \cite{abouzaid2013homological}, \cite{lee2016homological}, \cite{bocklandt2016noncommutative}, and \cite{pascaleff2019topological}.
Because $X$ is necessarily not proper, Conjecture \ref{conj:twisted-hms} does not directly apply to this case (see Remark \ref{rem:non-proper}). But suppose for the time being that there is a reasonable generalization of $\RFuk^{\gr}(P_{Z},\Omega_{Z})$ for punctured $Z$, which satisfies the conclusion of Conjecture \ref{conj:orlov}, namely
\begin{equation}
  \RFuk^{\gr}(P_Z,\Omega_{Z}) \simeq \Fuk^{w}(Z).
\end{equation}
The generalization of Conjecture \ref{conj:twisted-hms} to this case would be an equivalence
\begin{equation}
  \Dsing(X) \simeq \RFuk^{\gr}(P_Z,\Omega_{Z}).
\end{equation}
Combining these equivalences with $\MF(Y,W) \simeq \Dsing(X)$ yields the Hori-Vafa HMS statement. 
\item The main result of our previous paper \cite{PS21} can be viewed as a generalization of Hori-Vafa HMS to the case when $X$ is a normal crossing surface with graph-like singular locus that arises as the central fiber of a superpotential $W:Y \to \mathbb{A}^1$, where $Y$ is a smooth threefold, but not necessarily toric. One novelty with respect to the Hori-Vafa setting is that $G(X)$ and $Z$ are allowed to be compact. Suppose that in the hypothesis of Conjecture \ref{conj:twisted-hms} we add the assumption that $X$ arises as the fiber of a morphism $W : Y \to \mathbb{A}^{1}$ from a smooth threefold. This implies that all $n_{e} = 0$, so that the manifold $P(\{n_{e}\})$ reduces to $P_{Z}$. Then the conclusion of Conjecture \ref{conj:twisted-hms} follows from the main result of \cite{PS21} plus Conjecture \ref{conj:orlov}.
%Our result allows us to recover in particular beautiful works of Seidel and Efimov on HMS for compact surfaces, and aswer questions of Lekili-Ueda and Cannizzo.
\item Conjecture  \ref{conj:twisted-hms} gives an  HMS statement for singularity categories of normal crossings surfaces with graph-like singular locus, dropping the assumption of the the existence of a superpotential or ambient threefold $Y$. Surprisingly the mirror is actually a four dimensional  Liouville cobordism, and in general there is no way to cut the dimension down to two. As an interesting class of examples,  Conjecture  \ref{conj:twisted-hms} builds mirrors for singularity categories of toric boundary divisors of toric threefolds; thus in particular it provides one generalization of Hori-Vafa HMS to all toric threefolds, without the Calabi-Yau assumption.
 \end{itemize}\end{remark}
% There is one difficulty, which is that the integral affine coordinate corresponding to the fiber of $P_{Z}$, that is, the logarimoment map for the S^1is not bounded on $q^{-1}(U)$; to handle this, at the very beginning of the construction we can cut off the fibers of $P_{Z}$ to some bounded size with respect to this coordinate (which is the mom NOT SURE IF THIS IS NECESSARY

\subsection{T-duality on $\gsing$}
\label{sec:t-duality-gsing}
We shall now interpret Conjecture \ref{conj:twisted-hms} in terms of T-duality and the SYZ picture. This makes the construction of $P(\{n_{e}\})$ appear more natural. (In order for this to make geometric sense we take the ground field to be $\mathbb{C}$ on the B-side.)

Let $X$ be a normal crossings surface with graph-like singular locus, and recall from Section \ref{sec:sing-supp} the singular support variety $\gsing(X)$. There is a projection $\pi: \gsing(X) \to X$ and a zero section $X \to \gsing(X)$ that makes $\gsing(X)$ into a stratified vector bundle over $X$; the rank is zero over the regular locus of $X$ and one over the singular locus of $X$. Recall that a coherent sheaf on $X$ has a singular support that is a conical subvariety of $\gsing(X)$, and that  perfect complexes are precisely those objects whose singular support is contained in the zero section $X$. In this way we obtain the heuristic principle
\begin{equation}
\Dsing(X) \text{ lives on } \gsing^{\circ}(X) := \gsing(X) \setminus X.
\end{equation}
(See also Remarks \ref{rem:efimov} and \ref{rem:recollement}.)

From this heuristic, we infer directly that in order to obtain a symplectic manifold whose Fukaya category (in an appropriate variant) recovers $\Dsing(X)$, we should find the mirror of $\gsing^{\circ}(X)$, and one way to do that is to find a singular torus fibration on $\gsing^{\circ}(X)$ and apply T-duality. We claim that the result is essentially the manifold $P(\{n_{e}\})$ appearing in Conjecture \ref{conj:twisted-hms}.

It has long been observed that the singular locus $\Sing(X)$ is T-dual to the Riemann surface $Z$ with pants decomposition described by $G(X)$. Since $\Sing(X)$ is a union of copies of $\bP^{1}$ and $\bA^{1}$ meeting at triple points, deleting the triple points yields a toric variety that carries a natural $S^{1}$-fibration. On the mirror side, each pair of pants in the pants decomposition of $Z$ carries an $S^{1}$-fibration on the complement of a $\Theta$-shaped graph in the pair of pants. T-duality takes these fibrations one to the other; from this picture we infer that the triple points on $\Sing(X)$ correspond the certain $\Theta$-shaped graphs in $Z$ that are the core of each pair of pants.

To see the torus fibration on $\gsing^{\circ}(X)$, we note that $\gsing^{\circ}(X)$ is a $\mathbb{C}^{\times}$-bundle over $\Sing(X)$. Deleting the fibers over the triple points gives a toric variety that has a natural $T^{2}$-fibration. Thus the T-dual space, call it $P$ should be constructed by taking the dual fibration and compactifying it with certain singular fibers. We argue that taking $P = P(\{n_{e}\})$ does the trick, where $n_{e} = -\deg(\gsing(X))|_{C_{e}}$.

First consider the case where all $n_{e} = 0$, meaning that $\gsing^{\circ}(X)$ is a trivial $\mathbb{C}^{\times}$-bundle over $\Sing(X)$, and that $P = P_{Z}$ reduces to the principal canonical bundle of $Z$. While it is difficult to describe precisely what happens at the triple point of $\Sing(X)$, we can justify this case by appealing to the known equivalence $\Dsing(X) \simeq \Fuk(Z)$ and the conjectural equivalence $\Fuk(Z) \simeq \RFuk^{\gr}(P_{Z},\Omega_{Z})$.

When some of the integers $n_{e}$ are different from zero, then both $\gsing^{\circ}(X)$ and $P(\{n_{e}\})$ differ from the $n_{e}= 0$ case by modifying the gluings in the affine structures over the edges in $G(X)$. (This was explicit in the way that $P(\{n_{e}\})$ was constructed from torus fibrations.) Note that on the side of $\gsing^{\circ}(X)$, the gluings are twisted by transformations that preserve the structure of the fibration over $\Sing(X)$, while on the side of $P(\{n_{e}\})$, the gluings are twisted in a way that destroys the fibration structure on $P_{Z}$. This is exactly what T-duality would predict.

It remains to check that the gluing matrices match. On the $\gsing^{\circ}(X)$ side, along the edge $e$ we have two coordinate systems $(x_{1},u_{1})$ and $(x_{2},u_{2})$, $\deg(x_{i}) = 0$, $\deg(u_{i}) = 2$, and these are related by the transformation
\begin{equation}
  (x_{1},u_{1}) \mapsto (x_{2},u_{2}) = (x_{1}^{-1},x_{1}^{n_{e}}u_{1}),
\end{equation}
which is nothing but the action of the matrix $\phi_{n_{e}}$ on these monomials.
This transformation preserves the fibration structure because $x_{2}$ is a function of $x_{1}$ only. In terms of affine structures, the monomials appearing in this formula lie in a certain lattice $M \simeq H^{1}(T,\Z)$ where $T$ is the torus fiber in $\gsing^{\circ}(X)$. Under T-duality, this becomes $M \simeq H_{1}(T^{\vee},\Z)$, where $T^{\vee}$ is the torus fiber in the dual space. Thus we must build the mirror space by using $\phi_{n_{e}}$ to twist the bases of $H_1(T^{\vee},\Z)$, and this is precisely how $P(\{n_{e}\})$ was constructed.

\subsection{Tubular neighborhoods of nodal divisors in K3 and abelian surfaces}
\label{sec:tubular-neighborhood-nodal}

Let $V$ be a smooth K3 or abelian surface, and let $Z \subset V$ be an at-worst-nodal divisor, and let $U = V \setminus Z$. In this subsection we show that the end of $U$ is modeled on a twisted principal canonical bundle of a Riemann surface as set out in Section \ref{sec:twisted-principal}. This furnishes yet another geometric description of these manifolds, and it also gives an interpretation of the autoequivalence $\Lambda$ appearing in Conjecture \ref{conj:twisted-hms}.

Let $Z'$ be a smoothing of the nodal surface $Z$. Then $Z'$ carries a collection of pairwise disjoint simple closed curves that are the vanishing cycles of the retraction $Z' \to Z$. There may be parallel curves in this collection; if so we just remember one curve $C_{e}$ from each homotopy class and the number $n_{e}$ of parallel copies. This gives us a collection of pairs $(C_{e},n_{e})$ where the $C_{e}$ are pairwise disjoint and pairwise nonhomotopic simple closed curves, and $n_{e} > 0$. We then use the surface $Z'$ and the pairs $\{(C_{e},n_{e})\}_{e}$ as input for the construction of Section \ref{sec:twisted-principal}, and we write
\begin{equation}
  P_{Z} = P(\{(C_{e},n_{e})\})
\end{equation}
for the resulting twisted principal canonical bundle. Thus, the manifold $P_{Z}$ is obtained from the untwisted principal canonical bundle $P_{Z'} = K_{Z'}\setminus Z'$ by cutting along the curves $C_{e}$ and regluing with a shear parameterized by $n_{e}$.

\begin{lemma}
  \label{lem:dehn-twists}
  The manifold $P_{Z}$ carries a symplectomorphism $\Lambda$ that is the identity outside of the cut-and-glue regions near the curves $C_{e}$, and which acts by the monodromy of the mapping torus of an $n_{e}$-fold Dehn twist near $C_{e}$.
\end{lemma}

\begin{proof}
This follows from the observation that the process of cutting along $C_{e}$ and gluing with a shear by $n_{e}$ is nothing but cutting along $C_{e}$ and inserting the symplectic mapping torus of an $n_{e}$-fold Dehn twist along the curve $C_{e}$.
\end{proof}

\begin{lemma}
  The end of $U$ is modeled on $P_{Z}$ (that is, there is a complement of a compact subset of $U$ that is symplectomorphic to an open subdomain in $P_{Z}$ onto which $P_{Z}$ retracts).
\end{lemma}

\begin{proof}
The end of $U$ is a punctured tubular neighborhood of the nodal curve $Z$ in the surface $V$. We may smooth $Z$ to $Z'$ within $V$ (since $Z$ is ample), and a punctured tubular neighborhood of $Z'$ is modeled on the principal canonical bundle $P_{Z'}$. The process of passing from the punctured tubular neighborhood of $Z$ to $P_{Z'}$ is therefore an instance of Weinstein $2$-handle attachment, where one $2$-handle is attached for each node of $Z$. It may be checked via a local calculation that this is precisely inverse to the process that produced $P_{Z}$ from $P_{Z'}$.  
\end{proof}

These lemmas lead one naturally to the following Conjecture.
\begin{conjecture}
  The category $\RFuk(U)$ is equivalent to $\RFuk(P_{Z})$. The symplectomorphism $\Lambda$ from Lemma \ref{lem:dehn-twists} induces an autoequivalence of $\RFuk(P_{Z})$, and this is precisely the autoequivalence $\Lambda$ appearing in Conjecture \ref{conj:twisted-hms}.
\end{conjecture}

\subsection{Generalization of the Lekili-Ueda conjecture for nodal divisors in K3  and abelian surfaces}
\label{sec:singular-divisor-in-k3}
We now interpret Conjecture \ref{conj:twisted-hms} in terms of singular divisors in K3 and abelian surfaces and their complements. This generalizes the conjecture of Lekili-Ueda for smooth divisors. On the B-side, this involves the theory of relative singularity categories from Section \ref{sec:relative-sing}.

Let $V$ be a smooth K3 or abelian surface, and let $Z \subset V$ be an at-worst-nodal ample divisor. We postulate that the mirror to the pair $(V,Z)$ is a smooth threefold $Y$ with an anticanonical pencil spanned by two sections $s,s' \in \Gamma(Y,L)$, where $L = K^{-1}_{Y}$; this should be such that $X = s^{-1}(0)$ is a normal crossings surface with graph-like singular locus which is mirror to $U = V \setminus Z$ in the sense that there is an equivalence $\Coh(X) \simeq \Fw(U)$. Furthermore, we expect there to be compatible localization sequences as in Section \ref{sec:case-of-curves},
\begin{equation}
  \xymatrix{
    \Fc(U) \ar[r]\ar[d]& \Fw(U) \ar[r]\ar[d] & \RFuk(U) \ar[d]\\
    \Perf(X)\ar[r]\ar[u] & \Coh(X) \ar[u] \ar[r]& \Dsing(X) \ar[u]
  }
\end{equation}

Thus far we have not used the second section $s' \in \Gamma(Y,L)$, nor the compactification of $U$ as $V = U \cup Z$. Our conjecture is that these two choices correspond to each other in a sense we now explain.

On the B-side, the category $\Dsing(X)$ carries an autoequivalence $L\otimes$, and a natural isomorphism $t : \id \to L^{-1}\otimes [2]$. The second section $s'$ induces a natural transformation $s' : \id \to L\otimes$.

On the A-side, we showed in Section \ref{sec:tubular-neighborhood-nodal} that the end of $U$ is modeled on a twisted principal canonical bundle of a Riemann surface as set out in Section \ref{sec:twisted-principal}, and that this latter manifold carries a symplectomorphism $\Lambda$. Thus  $\Lambda$ induces an autoequivalence of the category $\RFuk(U)$ and (assuming Conjecture \ref{conj:twisted-hms}) a natural isomorphism $t : \id \to \Lambda^{-1}[2]$. It remains to see that the choice of compactification $V = U \cup Z$ induces a natural transformation $\id \to \Lambda$.

We are now in a situation that has been studied in detail by Maxim Jeffs \cite{jeffs}. When $\Lambda$ is the monodromy of a symplectic fibration over a punctured disk, one may obtain a natural transformation $\eta : \id \to \Lambda$ by filling in the  central fiber to obtain a Lefschetz fibration over the disk, and counting sections of the Lefschetz fibration. What we propose is that, since the autoequivalence $\Lambda$ acting on $\RFuk(U) \cong \RFuk(P_{Z})$ is obtained locally from Dehn twists, the choice of compactification divisor $Z$ similarly induces a natural transformation $\eta : \id \to \Lambda$.

In his setting Jeffs has proven that, when the fiber is a surface $\Sigma$, and $\Lambda$ is a composite of Dehn twists about pairwise disjoint curves, then the localization $\Fuk(\Sigma)[\eta^{-1}]$ is equivalent to the quotient of $\Fuk(\Sigma)$ by the vanishing cycles; he takes this as the definition of the Fukaya category of the nodal surface obtained by collapsing the vanishing cycles. Note that we have already shown a similar equivalence for the category $\Dsing(X,Y') \cong \MF(Y,L,s)[s'^{-1}]$ in Corollary \ref{cor:relative-sing-fuk}.

We summarize this discussion in our last conjecture.
\begin{conjecture}
  \label{conj:jeffs}
  Let $V$ be a smooth K3 or abelian surface and $Z \subset V$ an at-worst-nodal ample divisor. Let $U = V \setminus Z$. Then $\RFuk(U)$ carries an autoequivalence $\Lambda$, a natural isomorphism $t : \id \to \Lambda^{-1}[2]$, and a natural transformation $\eta: \id \to \Lambda$, where $\Lambda$ and $\eta$ are given by Jeffs' construction locally near the nodes of $Z$. Also, the localization $\RFuk(U)[\eta^{-1}]$ is equivalent to Jeffs' version of $\Fuk(Z)$.

  There is a smooth threefold $Y$ and anticanonical sections $s,s' \in \Gamma(Y,L)$ ($L = K_{Y}$) such that $X = s^{-1}(0)$ is the mirror of $U$ in the sense that there are equivalent localization sequences
  \begin{equation}
    \label{eq:conj-localization-sequences}
  \xymatrix{
    \Fc(U) \ar[r]\ar[d]& \Fw(U) \ar[r]\ar[d] & \RFuk(U) \ar[d]\\
    \Perf(X)\ar[r]\ar[u] & \Coh(X) \ar[u] \ar[r]& \Dsing(X) \ar[u]
  }
\end{equation}
and such that the $\Lambda$-twisted $2$-periodic structure of $\RFuk(U)$ matches with the $L$-twisted $2$-periodic structure of $\Dsing(X) \cong \MF(Y,L,s)$. Under this equivalence, the natural transformation $\eta$ goes over to $s' : \id \to L\otimes$. Hence there is an equivalence of localizations
\begin{equation}
  \Fuk(Z) \cong \RFuk(U)[\eta^{-1}] \cong \MF(Y,L,s)[s'^{-1}] \cong \Dsing(X,Y'),
\end{equation}
where $Y'$ is the singular variety obtained by blowing up the base locus of the pencil $\langle s, s'\rangle$.
\end{conjecture}

To illustrate the general theory, let us describe what this conjecture amounts to in the case of the Batyrev mirror family for the quartic K3 surface $V \subset \bP^{3}$. We assume that $V$ is in general position with respect to the toric boundary divisor $\partial \bP^{3}$; for concreteness we could let $V = \{\sum_{i=0}^{3} x_{i}^{4}=0\}$ be the Fermat quartic. Then $Z = V \cap \partial \bP^{3}$ is a nodal Riemann surface with four irreducible components each of genus three, with each pair of components meeting at four nodes.

Let $\Delta$ denote the reflexive polytope for the anticanonical linear system on $\bP^{3}$ (a tetrahedron). The cone over $\Delta$ gives a fan $\Sigma$, and the dual threefold $Y$ is a toric resolution of the associated toric variety $X_{\Sigma}$. There are multiple resolutions, but there is a highly symmetrical one corresponding to the subdivision of the faces of $\Delta$ by triangles with sides parallel to the edges of $\Delta$. The toric boundary $X = \partial Y$ is a normal crossings K3 surface, but it is not of type III, since the triple point formula is not satisfied along certain curves; we call these curves \emph{defect} curves. There are four defect curves along each edge of the tetrahedron $\Delta$. If we let $G(X)$ be the trivalent graph describing $\Sing(X)$, and we let $Z'$ surface with pants decomposition described by $G(X)$, then $Z'$ is a smoothing of the original divisor $Z$, and the positions of the nodes correspond to the positions of the defect curves in $\Sing(X)$. Along each of these curves the defect is $n_{e} = 1$.

Please see Figure \ref{fig:subdivision} for a visualization of this geometry, using code from \cite{stackexchange}.

\begin{figure}[h]
  \centering
  \includegraphics[width=.4\linewidth]{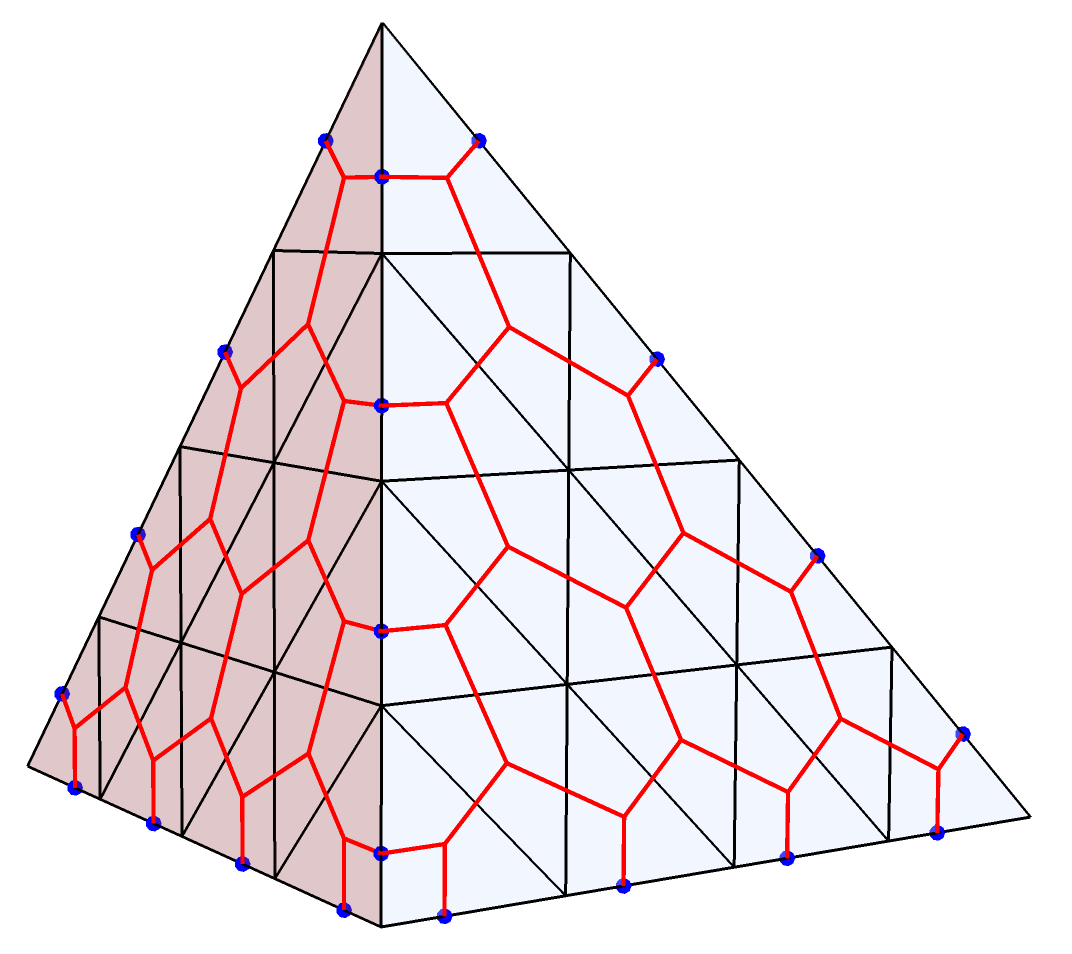}\qquad\includegraphics[width=.4\linewidth]{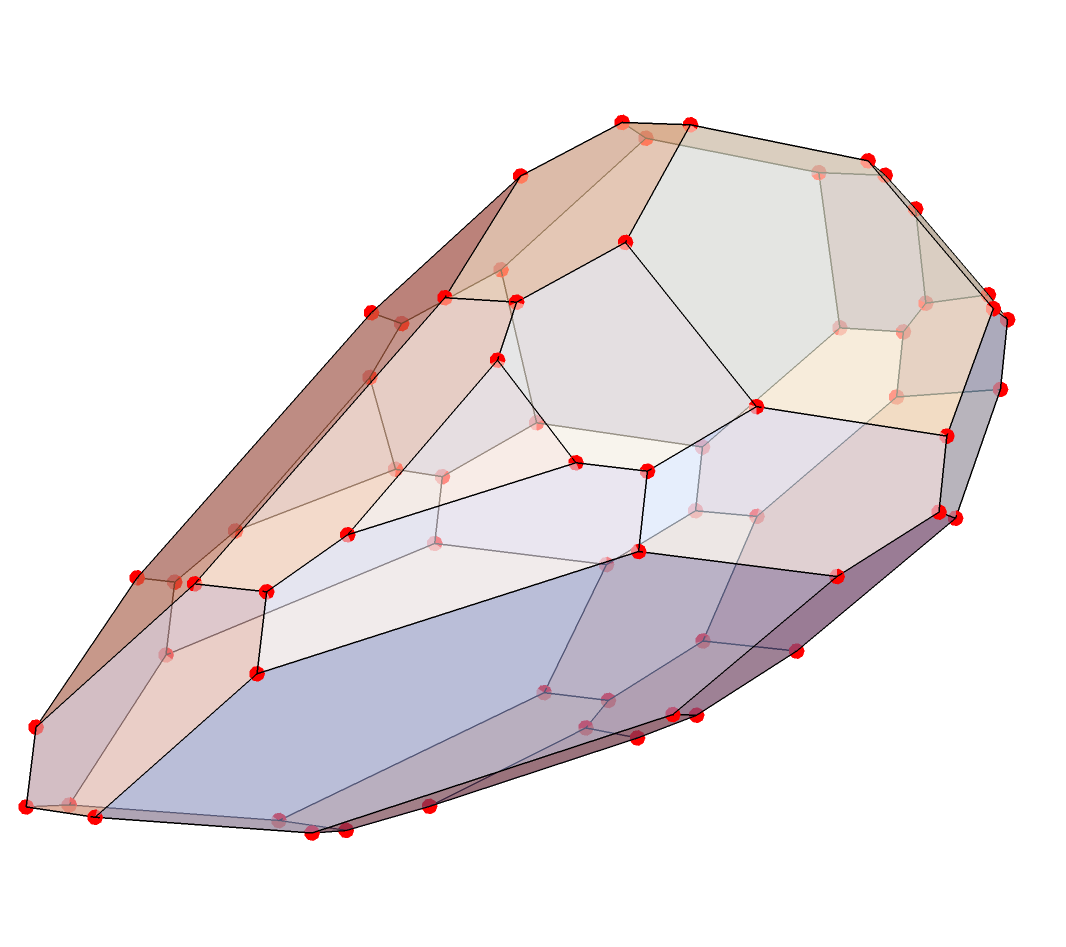}
  \caption{Left: The tetrahedron $\Delta$ in 3D perspective. Each face is subdivided into triangles, and the variety $Y$ is the toric variety whose fan is the cone over this subdivision. The dual graph to the subdivision may be interpreted as a tropical picture of the nodal curve $Z$, and the nodes occur a the points marked along the edges of $\Delta$. Right: The polytope associated to a certain ample line bundle on $Y$. The $1$-skeleton of this polytope is isomorphic to the dual graph in the left-hand picture; this graph is $G(X)$ when $X = \partial Y$. The boundary of the polytope consists of 4 triangles and 30 hexagons; 18 of the hexagons represent $\bP^{2}$ blown up at three torus-fixed points, and 12 represent $\bP^{1}\times \bP^{1}$ blown up at two torus-fixed points lying on a torus-invariant curve  (these surfaces are abstractly isomorphic but carry different toric structures). The defect curves are the torus-invariant curves where the sum of the self-intersection numbers differs from $-2$, and in each case it is $-1$. In the right-hand picture, the defect curves are the edges of the triangles and also the edges opposite to these in faces that neighbor the triangles, while in the left-hand picture they are the edges marked by the blue points.}
  \label{fig:subdivision}
\end{figure}

The variety $X$ is mirror to $U= V\setminus Z$, and in this setting HMS has been established \cite{seidel-quartic, sheridan-cy, gammage-shende-very-affine}. Assuming the expected property that $\RFuk(U) \simeq \Fc(U)/\Fw(U)$, this proves the validity of the diagram \eqref{eq:conj-localization-sequences}. The process of compactifying $U$ to $V$ gives rise to a mirror family, which may be described by a pencil $\langle s,s' \rangle$ where $X = s^{-1}(0)$. More specifically, $Y$ is a resolution of $\bP^{3}/\Gamma_{16}$, where $\Gamma_{16}$ is the group of scalar matrices that preserves the Dwork pencil $y_{0}y_{1}y_{2}y_{3}+q(y_{0}^{4}+y_{1}^{4}+y_{2}^{4}+y_{3}^{4})$, and the pencil on $Y$ is induced from the Dwork pencil itself. The second section $s'$, when restricted to $\Sing(X)$, vanishes exactly once along each defect curve.

We now consider the proposed chain of equivalences
\begin{equation}
  \label{eq:proposed}
  \Fuk(Z) \cong \RFuk(U)[\eta^{-1}] \cong \MF(Y,L,s)[s'^{-1}] \cong \Dsing(X,Y').
\end{equation}
The third equivalence has already be proven in Theorem \ref{thm:relative-graph-like}. Having already established the equivalence $\RFuk(U) \cong \Dsing(X) \cong \MF(Y,L,s)$, it is then clear that the second equivalence in \eqref{eq:proposed} holds for \emph{some} choice of $\eta$. Additionally, Corollary \ref{cor:relative-sing-fuk} implies that there is an equivalence $\Fuk(Z_{1}) \cong \MF(Y,L,s)[s'^{-1}]$ for \emph{some} nodal Riemann surface $Z_{1}$. What remains to be demonstrated is that this $Z_{1}$ really is the compactification divisor $Z$ we started with, and to verify the geometric interpretation of $\eta$ as counting sections of local Lefschetz fibrations. Because the positions of the nodes in $Z$ correspond to the positions of the defect curves in $\Sing(X)$, and because $s'$ vanishes exactly once along each defect curve, we see $Z_{1}$ has the same topology as $Z$. However, there are symplectic moduli given by the symplectic areas of the irreducible components of $Z_{1}$ and $Z$, and these should be matched.

The reader is invited to spell out what Conjecture \ref{conj:jeffs} says for other Batyrev mirror families, where an analogous analysis goes through.

\begin{proposition}
  Let $V,Z,U$ be as in Conjecture \ref{conj:jeffs}. Suppose that $(Y,L,s)$ is given such that $X = s^{-1}(0)$ is a normal crossings surface with graph-like singular locus that is mirror to $U$ in the sense that the first two columns of diagram \eqref{eq:conj-localization-sequences} are valid. Let $s'$ be a section of $L$ whose restriction to $\Sing(X)$ has only simple zeros at points which are not triple points.

  Suppose that the positions of the defect curves in $\Sing(X)$ correspond to the positions of the nodes in $Z$, meaning that the graph $G(X)$ describes a smoothing $Z'$ of $Z$, such that whenever $e$ is an edge corresponding a defect curve with defect $n_{e} > 0$, there are precisely $n_{e}$ parallel vanishing cycles the map $Z' \to Z$ that lie over $e$.

  Then there is a nodal symplectic surface $Z_{1}$ with the same topology as $Z$ such that
  \begin{equation}
    \Fuk(Z_{1}) \cong \MF(Y,L,s)[s'^{-1}] \cong \Dsing(X,Y')
  \end{equation}
  where $Y'$ is the singular variety obtained by blowing up the base locus of the pencil $\langle s, s'\rangle$, and where $\Fuk(Z_{1})$ is understood as the quotient of the Fukaya category of a smoothing of $Z_{1}$ by a collection of objects supported on the vanishing cycles.
\end{proposition}

\begin{proof}
  This follows form Theorem \ref{thm:relative-graph-like} and Corollary \ref{cor:relative-sing-fuk}.
\end{proof}

\bibliographystyle{alpha}
\bibliography{dsing-divisors}

\end{document}